\documentclass[reqno]{amsart}
\usepackage{amscd}
\usepackage{amsmath,amsthm,amssymb,amsfonts}
\usepackage{mathrsfs}
\usepackage{graphicx}
\usepackage{fancyhdr}
\usepackage{stmaryrd}
\usepackage[colorinlistoftodos,prependcaption,textsize=tiny]{todonotes}
\usepackage{color}
\usepackage{amsmath}
\usepackage{amsfonts}
\usepackage{amssymb}
\usepackage{geometry}
\usepackage{mathtools}
\usepackage{chngcntr}
\usepackage{amsfonts,amsmath,latexsym,verbatim,amscd,mathrsfs,color,array}
\usepackage[colorlinks=true,linkcolor=blue,citecolor=blue]{hyperref}

\usepackage{amssymb}
\usepackage{pdfsync}
\usepackage{epstopdf}

\newcommand{\crr}{\color{red}}
\newcommand{\cb}{\color{blue}}

\long\def\hide#1{}

\newcommand{\dif}{\,\mathrm{d}}

\newcommand{\inn}{{\quad\hbox{in } }}

\newcommand{\ttt}{\tilde }

\newcommand{\nn}{ {\nabla}  }

\newcommand{\pp}{ {\partial} }

\newcommand{\vp}{\varphi}
\newcommand{\B}{\beta }

\newcommand{\N}{\mathbb{N}}

\newcommand{\R} {\mathbb R}

\newcommand{\cuad}{{\sqcap\kern-.68em\sqcup}}

\renewcommand{\a}{{\alpha}}

\newcommand{\foral}{\quad\mbox{for all}\quad}
\newcommand{\ve}{\varepsilon}

\newcommand{\be}{\begin{equation}}
\newcommand{\ee}{\end{equation}}

\newcommand{\equ}[1]{(\ref{#1})}

\newtheorem{definition}{Definition}
\newtheorem{lemma}{Lemma}[section]
\newtheorem{proposition}{Proposition}[section]
\newtheorem{theorem}{Theorem}

\newtheorem{remark}{Remark}[section]
\newcommand{\bremark}{\begin{remark} \em}
\newcommand{\eremark}{\end{remark} }

\numberwithin{equation}{section}

\numberwithin{equation}{section}

\def \d{\delta}
\def\R{\mathbb{R}}

\def \p{\partial}
\def \s{\sigma}
\def \t {\theta}
\def \e {\varepsilon}

\def \a {\alpha}
\def \RE {\mathop{\text{Re}}}
\def \IM {\mathop{\text{Im}}}

\def \L {\mathcal{L}}
\def \B {\mathcal{B}}

\def \d {\tilde{d}}
\def \r {w}
\def \v {\theta}
\def \l {\rho}
\def \d {\hat{d}}
\def \dd {\tilde{d}}
\def \foral { \text{ for all } }
\def \N {\mathcal{N}}

\def \z {\tilde{z}}

\makeatletter
\def\@seccntformat#1{\@ifundefined{#1@cntformat}%
   {\csname the#1\endcsname\quad}%      default
   {\csname #1@cntformat\endcsname}%    enable individual control
}
\makeatother

\author{ Juan D\'avila, Manuel del Pino, Maria Medina and R\'emy Rodiac}
\title[ Interacting  helical Ginzburg-Landau filaments]{Interacting helical vortex filaments in the 3-dimensional Ginzburg-Landau equation}
\date{}

\address[Juan D\'avila]{ Departamento de Ingenier\'ia Matem\'atica-CMM Universidad de Chile, Santiago 837-
0456, Chile}
\email{jdavila@dim.uchile.cl}

\address[Manuel del Pino]{Department of Mathematical Sciences University of Bath, Bat
h BA2 7AY, United
Kingdom, Departamento de Ingenier\'ia Matem\'atica-CMM Universidad de Chile, Santiago 837-
0456, Chile, Chile}
\email{m.delpino@bath.ac.uk}

\address[Mar\'ia Medina]{Universidad de Granada, Departamento de An\'alisis Matem\'atico, Campus Fuentenueva, 18071 Granada, Spain}
\email{mamedina@ugr.es}

\address[R\'emy Rodiac]{
Universit\'e catholique de Louvain, Institut de Recherche en Math\'ematique et Physique, Chemin du Cyclotron 2 bte L7.01.02, 1348 Louvain-la-Neuve, Belgium}
\email{remy.rodiac@uclouvain.be}

\begin{document}

\begin{abstract}
For each given $n\ge 2$,
we construct a family of entire solutions $u_\ve (z,t)$, $\ve>0$,  with helical symmetry to the 3-dimensional complex-valued Ginzburg-Landau equation
\begin{equation*}\nonumber
\Delta u+(1-|u|^2)u=0, \quad   (z,t) \in \R^2\times \R \simeq \R^3.
\end{equation*}
These solutions are $2\pi/\ve$-periodic in $t$ and have $n$ helix-vortex curves, with asymptotic behavior as $\ve\to 0$
$$
u_\ve (z,t) \approx  \prod_{j=1}^n  W\left( z-   \ve^{-1} f_j(\ve t) \right),
$$
where $W(z) =w(r) e^{i\theta} $, $z= re^{i\theta},$ is the standard degree $+1$ vortex solution of the planar Ginzburg-Landau equation
$  \Delta W+(1-|W|^2)W=0 \text{ in } \R^2 $ and
$$ f_j(t)  =   \frac {  \sqrt{n-1} e^{it}e^{2 i (j-1)\pi/ n   }}{  \sqrt{|\log\ve|}}, \quad j=1,\ldots, n. $$
Existence of these solutions was previously conjectured in \cite{delPinoKowalczyk2008}, being  ${\bf f}(t) = (f_1(t),\ldots, f_n(t))$ a rotating equilibrium point for
the renormalized energy of vortex filaments there derived,
$$
\mathcal W_\ve ( {\bf f} )  :=\pi \int_0^{2\pi}  \Big ( \, \frac{|\log \e|} 2 \sum_{k=1}^n|f'_k(t)|^2-\sum_{j\neq k}\log |f_j(t)-f_k(t)| \, \Big ) \dif t,
$$
corresponding to that of a planar logarithmic $n$-body problem.  These solutions satisfy
$$
\lim_{|z| \to +\infty }  |u_\ve (z,t)| = 1 \quad \hbox{uniformly in $t$}
$$
and have nontrivial dependence on $t$, thus negatively answering the Ginzburg-Landau analogue of the Gibbons conjecture for the Allen-Cahn equation,
a question originally formulated by H. Brezis.

\end{abstract}

\maketitle
\section{Introduction}

This paper deals with constructing entire solutions to the  complex Ginzburg-Landau equation in the Euclidean space $\R^N$
\begin{equation}\label{eq:GLequation}
\Delta u+(1-|u|^2)u=0\quad \text{ in } \R^N,
\end{equation}
where  $u:\R^N \rightarrow \mathbb{C}$ is a complex-valued function and $N\ge 2$.
It is convenient for our purposes to introduce a small parameter $\ve>0$ and consider the equivalent scaled version of \equ{eq:GLequation}
given by
\begin{equation}\label{eq:GLequationeps}
\e^2\Delta u +(1-|u|^2)u=0  \quad \text{ in } \R^N.
\end{equation}
When regarded in a bounded region $\Omega\subset \R^N$, equation \equ{eq:GLequationeps}
corresponds to the Euler-Lagrange equation for the functional
\be\label{Jve}
J_\ve(u) = \frac 12  \int_\Omega |\nn u|^2    + \frac 1{4\ve^2} \int_\Omega (1-|u|^2)^2,
\ee
which for $N=2,3$
is often considered as a model for the energy arising in the standard
Ginzburg-Landau theory of superconductivity when no external applied magnetic field is
present.  In that setting, the complex-valued state of the system $u$ corresponds to a critical
point of $J_\ve$ in which $|u|^2$ represents the density of the super-conductive property of the sample $\Omega$
(Cooper pairs of electrons). The function $u$ is expected to stay away from zero except on a lower-dimensional zero set,
the vortex set, corresponding to {\em defects} where
superconductivity is not present.

\medskip
In their pioneering work \cite{BethuelBrezisHelein1994}, Bethuel-Brezis-H\'elein  analyzed in dimension $N=2$ the behavior as $\ve\to 0$  of a global minimizer $u_\ve$ of $J_\ve$ when subject to a boundary condition $g: \pp\Omega \to \mathbb{S}^1$ of degree $k\ge 1$. They established that away from a finite number of distinct points $a_1,\ldots, a_k\in \Omega$ one has (up to subsequences)
$$
u_\ve (x) \approx   e^{i\vp(x) }\prod_{j=1}^k \frac {x-a_j}{|x-a_j|},
$$
where $\vp(x) $ is a real harmonic function and the $k$-tuple $(a_1,\ldots, a_k)$ minimizes a functional of points, the {\em renormalized energy} that measures through Green's function the mutual interaction between the points and the boundary. Using the results in \cite{Mironescu1996,Sandier1998,Shafrir1994} one gets the validity of the global approximation
\be\label{forma}
u_\ve (x) \approx   e^{i\vp(x) }\prod_{j=1}^k   W\left ( \frac {x-a_j}{\ve}   \right ),
\ee
where $W(z)$ is the {\em standard degree $+1$ vortex solution} of the equation \equ{eq:GLequation} for $N=2$,
%\begin{equation} \Delta W + (1-|W|^2)W = 0 \quad\hbox{in } \R^2, \label{stationary}\end{equation}
namely its unique solution
of the form  \be\label{v1} W(z) = e^{i\theta} w(r), \quad z=re^{i\theta},\ee where $w>0$ solves
\begin{equation}\label{eq:equation_modulus_w}  \left\{  \begin{aligned}
&w'' +  \frac {w'}r  - \frac w{r^2} +  (1-w^2) w  =  0 \inn (0,\infty),\\
& w(0^+) = 0 , \quad w(+\infty) = 1 ,
\end{aligned} \right. 
\end{equation} see \cite{ChenElliottQi1994,HerveHerve1994}. Thus before reaching the limit, the vortex set of $u_\ve$ is constituted by exactly $k$ distinct points, each with local degree $+1$. The mechanism of vortex formation in two-dimensional Ginzburg-Landau model from the action of an external constant magnetic field has been extensively studied, see \cite{Sandier_Serfaty2007} and references therein.
Critical points of the renormalized energy are in fact in correspondence with other critical points of $J_\ve$ in \equ{Jve}  of the form \equ{forma}
for small $\e$, as it has been found in  \cite{Lin1995,delPinoFelmer1997,almeida,PacardRiviere2000,delPinoKowalczykMusso2006}.
In the higher dimensional case $N\ge 3$ and with suitable  boundary conditions and energy levels, the vortex set of minimizers and more general critical points
have been described when $\ve\to 0$ in \cite{Riviere1996,LinRiviere1999,Sandier2001,BethuelBrezisOrlandi,LinRiviere2001,
jerrardsoner,Alberti-Baldo-Orlandi2005}  as a codimension 2 set with a generalized minimal submanifold structure.  In dimension $N=3$ defects should typically assume the form of curves with a
winding number associated: these are called {\em vortex filaments}. The basic {\em degree $+1$  vortex line} is the solution $u$
of \equ{eq:GLequation} for $N=3$ given by
\be
u(z,t) =  W(z) ,  \quad (z,t)\in \R^2\times \R\simeq \R^3,
%\label{v2}
\nonumber
\ee
with $W(z)$ specified in \equ{v1}.  Its zero set is of course the $t$-axis, and a transversal winding number $+1$ is associated to it.
In dimension $N=3$, under Neumann boundary conditions, it was found in \cite{Montero_Sternberg_Ziemer2004} a local minimizer with energy formally corresponding
to multiple vortex lines collapsing onto a segment. Motivated by this work, in \cite{delPinoKowalczyk2008} an expression for the renormalized energy for the interaction of nearly parallel ``degree +1 vortex lines'' collapsing onto the $t$-axis was derived. Considering $n$ curves
\begin{align}
%\label{para-f}
\nonumber
 t\mapsto(f_i(t),t), \quad 1\leq i \leq n, \quad {\bf f} =(f_1,\cdots,f_n),
\end{align}
which for simplicity we assume $2\pi$-periodic, we look for an approximate solution
of the form
\be \label{form}
u_\ve (z,t)  \approx   W_\ve(z,t; {\bf f} ):=    e^{i\vp(z,t)} \prod_{j=1}^n  W\left ( \frac{z-   f_j(t)}{\ve}    \right ).
\ee
In the cylinder $\Omega =\mathcal C = B_R(0) \times (0,2\pi)$, with $\vp$ harmonic matching lateral zero Neumann boundary conditions it is found in \cite{delPinoKowalczyk2008} that
\be\label{func}
I_\ve({\bf f} ):=J_\ve (W_\ve (\cdot ; {\bf f} ) ) \approx    2\pi \times n\pi|\log\ve|  +   \mathcal W _\e({\bf f}  )
\ee
where
\begin{equation}\label{We}
\mathcal W _\e(\mathbf{ f}):=\pi \int_0^{2\pi}  \Big ( \, |\log \e| \frac12 \sum_{k=1}^n|f'_k(t)|^2-\sum_{j\neq k}\log |f_j(t)-f_k(t)| \, \Big ) \dif t.
\end{equation}
Equilibrium location of these curves should then correspond to an approximate critical point of  the functional  $I_\ve$ and hence of $\mathcal W _\e$, which
 is the action  associated to the $n$-logarithmic body problem in $\R^2$. This energy also appears in related problems in fluid dynamics, see e.g. \cite{Klein_Majda_Damodaran1995, Kenig_Ponce_Vega2003}.
If we set
 $${\bf f}(t) =\frac{1}{\sqrt{|\log \e| }} \tilde{\bf f}(t), \quad  \tilde{\bf f} =(\tilde{f}_1,\cdots,\tilde{f}_n), $$ this corresponds to a
  $2\pi$-periodic solution of the ODE system
\begin{equation}\label{sist}
-\tilde{f}_k''(t)= 2\sum_{i\neq k}\frac{\tilde{f}_k(t)-\tilde{f}_i(t)}{|\tilde{f}_k(t)-\tilde{f}_i(t) |^2}.
\end{equation}
The following $n$-tuple  $ \ttt {\bf f}^0$  is  a standard rotating solution of system \equ{sist}.
\begin{equation}\label{eq:def_fvortex}
\tilde{f}^0_k(t)= \sqrt{n-1} e^{i t}e^{2i(k-1)\pi/n}, \ \ k=1,\cdots,n.
\end{equation}
It is shown in \cite{delPinoKowalczyk2008} that the functional  $I_\ve $  in \equ{func} does have a $2\pi$-periodic critical point ${\bf f}^\ve(t)$ such that
 \be {\bf f}^\ve(t)\,=\, {\bf f}^0(t)+  \frac {o(1)}{\sqrt{|\log\ve|}}   , \quad  {\bf f}^0(t):= \frac 1{\sqrt{|\log\ve|}} \tilde{\bf f}^0(t)\label{f0} \ee  %\equ{eq:def_fvortex},
 uniformly as $\ve\to 0$,
and it is conjectured the existence of a solution $u_\ve (z,t)$ to the system
\begin{equation}\label{eq:GLequationeps3}
\e^2\Delta u +(1-|u|^2)u=0,  \quad  (z,t) \in \R^3,
\end{equation}
which is $2\pi$-periodic in $t$ and has the approximate form \equ{form} for ${\bf f}$ as in \equ{f0}.
The recent work \cite{ContrerasJerrard2017} has established a rigorous connection, in the sense of $\Gamma$-convergence,
between minimizers of the functional \equ{We} and minimizers in cylinders with suitable Dirichlet boundary condition, thus providing evidence towards the conjecture in \cite{delPinoKowalczyk2008}. In this paper we prove this conjecture.

\begin{theorem}\label{th:main1}
For every $n\geq 2$ and for $\e$ sufficiently small there exists a solution $u_\ve (z,t)$ of \eqref{eq:GLequationeps3}, $2\pi$-periodic in the
$t$-variable,   with the following asymptotic profile:
\begin{equation*}
u_\e(z,t)=\prod_{k=1}^n W\left(\frac{z-    {f}^\ve_k(t)}{\e} \right)+  \varphi_\ve(z,t),
\end{equation*}
where ${f}_k^\ve(t)$ is $2\pi$-periodic with the asymptotic behavior \eqref{f0} and
\begin{equation*}
|\varphi_\e(z,t)| \, \le\,\frac{C}{|\log \e|}. 
\end{equation*}
Besides we have
\begin{equation}\label{eq:prop_Gibbons}
\lim_{|z|\rightarrow +\infty} |u_\e(z,t)|=1 \text{ uniformly in } t.
\end{equation}
\end{theorem}

We are also able to construct another family of solutions to \eqref{eq:GLequationeps3}. Until now we have dealt only with vortex filaments of degree \(+1\). However it is believed that, in presence of several vortex filaments of different degrees \(d_k \in \mathbb{Z}\) the energy governing the interaction of the filaments is

\begin{equation}\label{interacting_several}
\mathcal W _\e(\mathbf{ f}):=\pi \int_0^{2\pi}  \Big ( \, |\log \e| \frac12 \sum_{k=1}^n|f'_k(t)|^2-\sum_{j\neq k}d_jd_k\log |f_j(t)-f_k(t)| \, \Big ) \dif  t.
\end{equation}
There exist special critical points of \eqref{interacting_several} analogous to \eqref{eq:def_fvortex} for \(d_1=-1\) and \(d_k=+1\) for \(k=2,\dotsc,n\) when \(n \geq 5\). These critical points can be written as
\begin{equation}\label{eq:def_vortex_2}
g_1^0(t)=0, \quad g_k^0(t)= \sqrt{n-4}e^{it}e^{2i(k-1)\pi/(n-1)}, \ \ \ k=2,\dotsc,n, \ \ \ n\geq 5.
\end{equation}
From these solutions we can obtain: 

\begin{theorem}\label{th:main2}
For every $n\geq 5$ and for $\e$ sufficiently small there exists a solution $u_\ve (z,t)$ of \eqref{eq:GLequationeps3}, $2\pi$-periodic in the
$t$-variable,   with the following asymptotic profile:
\begin{equation*}
u_\e(z,t)=\overline{W}(z)\prod_{k=2}^n W\left(\frac{z-    {g}^\ve_k(t)}{\e} \right)+  \varphi_\ve(z,t),
\end{equation*}
where ${g}_k^\ve(t)$ is $2\pi$-periodic with \(g_k^\e(t)=g_k(t)+\frac{o_\e(1)}{\sqrt{|\log \e|}}\), \(g_k\) defined by \eqref{eq:def_vortex_2}  and
\begin{equation*}
|\varphi_\e(z,t)| \, \le\,\frac{C}{|\log \e|}. 
\end{equation*}
Besides we have
\begin{equation*}
\lim_{|z|\rightarrow +\infty} |u_\e(z,t)|=1 \text{ uniformly in } t.
\end{equation*}
\end{theorem}

The proofs of both theorems give a precise answer to the existence question, with an accurate description of the solution. They take specific advantage of the geometric setting: the configuration predicted is one of multiple helix vortex curves periodically winding around each other.
The Ginzburg-Landau equation has a {\em screw driving} symmetry which we take advantage of to reduce the original problem to a planar one. For constructions of solutions with helical vortex structures we refer also to \cite{Chiron_2005,WeiJun2016}.

\medskip
We observe that in terms of the parameterless equation  \equ{eq:GLequation} for $N=3$, 
what we find is a family of entire solutions $u_\ve(z,t)$,   $\frac{2\pi}{\ve}$-periodic in $t$ with approximate form 
$$
u_\ve (z,t) \approx  \prod_{j=1}^n  W\left( z-   \ve^{-1} f^\ve_j(\ve t) \right).
$$
Equation \eqref{eq:GLequation} is the  complex-valued version of the Allen-Cahn equation of phase transitions,
\begin{equation}\label{ac}
\e^2\Delta u +u-u^3=0 \inn \R^N,
\end{equation}
where $u:\R^N \rightarrow \R.$
The Allen-Cahn model describes transitions of two phases between the values $-1$ and $+1$ essentially separated by a thick wall, which for small $\ve$ should lie close to a minimal hypersurface. Solutions with screw-driving symmetry, whose zero level set is precisely a helicoid have been built in \cite{delPinoMussoPacard2012} (and extended to the fractional case in \cite{CintiDaviladelPino2016}).
Solutions with multiple interfaces whose interactions are governed by mechanical systems (Toda systems) have been built in \cite{delPinoKowalczykWei2008,delPinoKowalczykPacardWei2010,delPinoKowalczykWeiYang2010,agudelo}.
The celebrated {\em De Giorgi conjecture} states that, at least up to dimension $N=8$ solutions of \equ{ac}  which  are monotone in one direction must have
 one dimensional-symmetry %family \equ{1d}, see
 namely their level sets must be parallel  hyperplanes, see
 \cite{GhoussoubGui1998,AmbrosioCabre2000,Savin2009,delPinoKowalczykWei2011} and references therein.
A variant of this statement is the  {\em Gibbons conjecture}: an entire solution $u$  of \equ{ac} such that
$$\lim_{|x_N|\rightarrow +\infty }|u(x',x_N)|= 1 \quad\hbox{uniformly in $x'\in \R^{N-1}$} $$
 must necessarily be a function of $x_N$ only.
 This fact has been proven for any  $N\geq 2$. See  \cite{BarlowBssGui2000,BerestyckiHamelMonneau2000,Farina1999}. The analogous question for
 the Ginzburg-Landau equation in $\R^N$, $N\ge 3$,  originally formulated by H. Brezis is whether or not a solution $u(z,t)$, $(z,t)\in \R^2 \times \R^{N-2}$ with
 \be\label{loc} \lim_{|z|\rightarrow +\infty }|u(z,t)|= 1 \quad\hbox{uniformly in $t \in \R^{N-2}$}, \ee
must necessarily be a function of $z$. This turns out to be false since the solutions in Theorem
  \ref{th:main1} satisfy \equ{loc}. We observe that
solutions $W_n(z)$ of \eqref{eq:GLequation}  with total degree $n$, $|n|\ge 1$  of the form \begin{equation}\nonumber%\label{wk}
W_n(z) =  e^{in\theta} w_n(r),\quad   w_n(0)= 0, \quad w_n(+\infty) = 1, \quad  z= re^{i\theta}, \end{equation} are known to exist for each $n\ge 2$,
see \cite{HerveHerve1994,ChenElliottQi1994}. The solutions in Theorem
  \ref{th:main1} have transversal total degree equal to $n\ge 2$. A natural question is whether or not Brezis' statement holds true under the additional assumption of total transversal degree equal to $\pm 1$. See
  \cite{Mironescu1996,Sandier1998,Shafrir1994} for the corresponding
  question in dimension $N=2$ and \cite{SandierShafrir2017} for a conjecture on the symmetry of entire solutions of \eqref{eq:GLequation} when \(N=3\).

 \medskip
 We will devote the rest of this paper to the proof of Theorem \ref{th:main1}. The proof of Theorem \ref{th:main2} follows the same lines.  As we have mentioned, the key observation is that the invariance under screw-driving symmetry allows to reduce the problem to one in the plane, for which the solution to be found has a finite number of vortices with degree 1. For simplicity we treat only the case $n=2$ in the following but the arguments can be easily adapted. We will look for solutions that are close  to the approximation
\begin{equation}\label{Approx}
u_d(x,y,t)=W\left(\frac{x-d\cos t}{\e},\frac{y-d\sin t}{\e}\right)W\left(\frac{x+d\cos t}{\e},\frac{y+d\sin t}{\e}\right)
\end{equation}
when $\e$ is small. Here $d$ is a parameter of size $1/\sqrt{|\log \e|}$.

It can be observed that the zero set of $u_d$ has the shape of a double helix and the function  $\tilde{u}_d:=e^{2it}u_d$ is screw-symmetric (see Definition \ref{screwsim}). Thus $\tilde{u}_d$ can be expressed as a function of two variables, what reduces the problem to a 2-dimensional case.  We will look for screw-symmetric perturbations of $\tilde{u}_d$.  Our approach will be based on the method in \cite{delPinoKowalczykMusso2006}, devised to build up solutions with isolated vortices when $N=2$ using a  Lyapunov-Schmidt reduction.
A major difficulty that we need to overcome is the presence of very large terms in the error of approximation. In the $2$-dimensional case we immediately obtain errors that are of size $O(\e)$ while here it is  $O(1/|\log \e|)$. This is a major difficulty since the
 vortex-location adjustment arises at essentially $\ve$-order.
This is overcome by carefully decomposing the error created by the nonlinearity in ``odd'' small  and ``even'' large  Fourier modes parts.
The even part has at main order no effect in the solvability conditions needed in the linear theory we devise in Proposition \ref{prop:linearfull}. These steps are rather delicate and we will carry them out in detail in what follows.

\section{Reduction to a 2-dimensional problem by using screw-symmetry}
As a first step we reduce our 3-dimensional problem to a related 2-dimensional one. To do so, we work with a particular type of symmetry. To define this symmetry it is convenient to use cylindrical coordinates $(r,\theta, t)\in \R^+\times \R \times \R$, and to work with functions that are \( 2\pi \) periodic in the second variable.
\begin{definition}\label{screwsim}
We say that a function $u$ is screw-symmetric if
\begin{equation*}
u(r,\theta+h,t+h)=u(r,\theta,t)
\end{equation*}
for any $h \in\R$.
\end{definition}
Notice that this condition is equivalent to $u(r,\theta,t+h)=u(r,\theta-h,t)\mbox{ for any }h\in \R,$
and then a screw-symmetric function can be expressed as a function of two-variables. Indeed, for any \((r,\theta,t) \in \R^+\times\R\times\R\),
\[u(r,\theta,t)=u(r,\theta-t,0)=:U(r,\theta-t). \]
Writing the standard vortex of degree one in polar coordinates $(r,\theta)$, i.e., $W(z)= w(r)e^{i\theta}$, we can see that the approximation $u_d$ defined in \eqref{Approx} satisfies
\[u_d(r,\theta,t+h)=e^{2ih}u_d(r,\theta-h,t) \]
for any \(h\) in \(\R\). That is, $u_d$ is not screw-symmetric but $\tilde{u}_d(r,\theta,t):=e^{-2it}u_d(r,\theta,t)$ is. Hence we can write $u_d$ as $u_d=e^{2it} \tilde{u}_d$, with $\tilde{u}_d$ a screw symmetric function.

This suggests to look for solutions $u$ of \eqref{eq:GLequationeps} that can be written as
\[u(r,\theta,t)=e^{2it}\tilde{u}(r,\theta,t) \]
with $\tilde{u}$ screw-symmetric. Thus $\tilde{u}(r,\theta,t)=U(r,\theta-t)$, being $U:\R^+\times \R$ $2\pi$-periodic in the second variable. Denoting $U=U(r,s)$ we can see that
$$
\p_ru = e^{2it}\p_r U(r,s), \quad \p^2_{rr}u = e^{2it} \p^2_{rr} U(r,s), \, \quad
\p_\theta u = e^{2it}\p_s U(r,s), \quad \p^2_{\theta \theta} u = e^{2it} \p^2_{ss} U(r,s),$$
$$\p_t u  = [2iU-\p_sU]e^{2it}, \ \ \  \p^2_{tt} u =[\p^2_{ss}U-4i\p_sU-4U ]e^{2it}.$$
Recalling that the Laplacian in cylindrical coordinates is expressed by $\p^2_{rr}+\frac{1}{r}\p_r+\frac{1}{r^2}\p^2_{\theta \theta}+\p^2_{tt}$ we conclude that $u$ is a solution of \eqref{eq:GLequationeps} if and only if $U$ is a solution of
\begin{equation*}
\e^2\left( \p^2_{rr}U+\frac{1}{r}\p_rU +\frac{1}{r^2}\p^2_{ss}U +\p^2_{ss}U-4i\p_sU-4U \right)+(1-|U|^2)U=0 \text{ in } \R^*_+\times \R.
\end{equation*}
We  will also work in rescaled coordinates, that is, we define $V(r,s):=U(\e r, s)$ and we search for a solution to the equation
\begin{equation}\label{GL2-dimensionalrescaled}
\p^2_{rr}V +\frac{1}{r}\p_rV+\frac{1}{r^2}\p^2_{ss}V+\e^2(\p^2_{ss}V-4i\p_sV-4V)+(1-|V|^2)V=0 \text{ in } \R^*_+\times \R.
\end{equation}

From now on we will work in the plane $\R^2$, and we will use the notation $z=x_1+ix_2=re^{is}$. We denote by $\Delta$ the Laplace operator in $2$-dimensions, meaning
\begin{equation*}
\Delta=\p^2_{x_1x_1}+\p^2_{x_2x_2}=\p^2_{rr}+\frac{1}{r}\p_r+\frac{1}{r^2}\p^2_{ss}.
\end{equation*}
\noindent Then equation \eqref{GL2-dimensionalrescaled} can be written as
\begin{align}
\label{GL2-dimensionalrescaled-2}
\Delta V
+\e^2(\p^2_{ss}V-4i\p_sV-4V)+(1-|V|^2)V=0 \text{ in } \R^2 ,
\end{align}
and the approximate solution \eqref{Approx} in the new coordinates is given by
\begin{equation}
\label{eq:Approx2-dimensionalrescaled}
V_d(z)=
W(z-\tilde d) W ( z + \tilde d)
% W\left(re^{is}-\frac{d}{\e}\right)W\left(re^{is}+\frac{d}{\e}\right),
\end{equation}
where
\begin{equation}\nonumber%\label{eq:relationd}
\dd:=\frac{d}{\e}=\frac{\hat{d}}{\e \sqrt{|\log \e|}},
\end{equation}
for some new parameter $\hat d = O(1)$.

\section{Formulation of the problem}\label{IV}

\subsection{Additive-multiplicative perturbation}
Let
\[
S(v) := 
\Delta v
+\e^2(\p^2_{ss}v-4i\p_sv-4v)+(1-|v|^2)v.
\]
The equation to be solved can be written as 
\begin{align}
\label{mainEq}
S(v)=0  .
\end{align}
Recall the notation $z = r e^{is} = x_1 + i x_2$ and $\Delta = \partial^2_{x_1x_1} + \partial^2_{x_2x_2} $.
When using the coordinates $(x_1,x_2)$ equation \eqref{mainEq} is posed in $\R^2$, while if we use polar coordinates $(r,s)$ the domain for \eqref{mainEq} is $r>0$, $s\in \R$ with periodicity.

Following  del Pino-Kowalczyk-Musso  \cite{delPinoKowalczykMusso2006},
we look for a solution to \eqref{mainEq} of the form
\begin{align}
\label{eq1}
v = \eta V_d (1+i\psi)  + (1-\eta) V_d e^{i\psi} ,
\end{align}
where $V_d$ is the ansatz \eqref{eq:Approx2-dimensionalrescaled} and $\psi$ is the new unknown.
The cut-off function $\eta$ in \eqref{eq1} is defined as
\begin{align*}
\eta(z) = \eta_1(|z-\tilde d|) + \eta_1(|z+\tilde d|) , \quad z\in \mathbb C=\R^ 2
\end{align*}
and  $\eta_1:\R \rightarrow [0,1]$ is a smooth cut-off function such that
\begin{align}
\label{eta1}
 \eta_1(t)=1 \text{ for } t\leq 1\text{ and }\eta_1(t)=0 \text{ for } t\geq 2.
\end{align}
 The reason for the form of the perturbation term in \eqref{eq1} is the same as in \cite{delPinoKowalczykMusso2006}.
On one hand, the nonlinear terms behave better for the norms that we consider when using the multiplicative ansatz, but near the vortices, an additive ansatz is better since it allows the position of the vortex to be adjusted.

Our objective here is to rewrite \eqref{mainEq} in the form 
\[
\mathcal L^\varepsilon\psi + R + \mathcal N(\psi)=0
\]
and identify the linear operator $\mathcal L^\varepsilon$, the error $R$ and the nonlinear terms  $\mathcal N(\psi)$.

\noindent It will be convenient to write
\begin{align*}
S  = S_0 + S_1
\end{align*}
where
\begin{align}
%S(u) &= S_0(u) + S_1(u).
%\\
S_0(v) = \Delta v + (1-|v|^2) v, \quad
S_1(v) = 
\varepsilon^2(\p^2_{ss}v-4i\p_s v-4v). \label{def:S_0S_1}
\end{align}

\noindent We have 
\begin{align*}
S_0(V_d+ \phi ) &= S_0(V_d) + L_0(\phi) + N_0(\phi)
\\
S_1(V_d + \phi) &= S_1(V_d)  + S_1(\phi) ,
\end{align*}
where
\begin{align}
L_0(\phi) &= \Delta \phi + (1-|V_d|^2) \phi -2 \RE(\overline{ V}_d \phi) V_d \label{L0}
\\
N_0(\phi) & = -2\RE(\overline{V}_d \phi) \phi - |\phi|^2 (V_d+\phi). \label{N0}
\end{align}
Rewrite \eqref{eq1} as 
\begin{align*}
v &= V_d + \phi 
\\
\phi &= i V_d \psi + \gamma(\psi)
\\
 \gamma(\psi) &= (1-\eta) V_d (e^{i\psi}-1-i\psi)
\end{align*}
Then
\begin{align*}
S_0(v) &= S_0(V_d) + L_0(i V_d \psi) + L_0(\gamma(\psi)) + N_0(i V_d \psi + \gamma(\psi))
\\
S_1(v) &= S_1(V_d) + S_1(iV_d\psi) + S_1(\gamma(\psi)).
\end{align*}

We compute
\begin{align*}
L_0(iV_d\psi) &=
iV_d \left[ 
\tilde L_0 \psi 
+ \frac{S_0(V_d)  }{V_d}\psi
\right] ,
\end{align*}
where
\begin{align*}
\tilde L_0 (\psi) = \Delta \psi + 2\frac{\nabla V_d \nabla \psi }{V_d}
-2 i |V_d|^2 \IM(\psi),
\end{align*}
and so 
\begin{align}
\label{S0near}
S_0(v) &= i V_d
\Bigg[
-i\frac{S_0(V_d)}{V_d} 
+\tilde L_0 (\psi)
+ \frac{S_0(V_d)}{V_d}\psi 
- \frac{i}{V_d}L_0(\gamma(\psi)) 
- \frac{i}{V_d}N_0(i V_d \psi + \gamma(\psi))
\Bigg] .
\end{align}

We note that far from the vortices we have
\begin{align}
\label{S0far}
S_0(v) &=  S_0(V_d e^{i\psi} )
%& =
%i V e^{i \psi} 
%\left[
%\Delta \psi
%+2 \frac{\nabla V \nabla \psi}{V}
%- i \frac{S_0(V)}{V}
% + i (\nabla \psi )^2
% + i |V|^2 (e^{-2 \IM(\psi)}-1) 
%\right]
%\\
=
i V_d e^{i \psi} 
\left[
- i \frac{S_0(V_d)}{V_d}
+\tilde L_0 (\psi )
+ \tilde N_0(\psi)
\right] ,
\end{align}
where
\begin{align*}
\tilde N_0(\psi) = i (\nabla \psi)^2 + i |V_d|^2(e^{-2\IM(\psi)}-1+2\IM(\psi) ).
\end{align*}

Similarly we compute
\begin{align}
\label{S1near}
S_1( i V_d \psi ) 
%&= i \varepsilon^2
%( 
%\partial_{ss}^2 V \psi 
%+2\partial_s V \partial_s \psi 
%+ V \partial_{ss}^2 \psi 
%-4 i \partial_{s} V \psi 
%-4 iV  \partial_{s} \psi 
%- 4 V \psi
%) 
%\\
%&=
% i\varepsilon^2 V 
%\left( 
%\frac{\partial_{ss}^2 V}{V} \psi 
%+2\frac{\partial_s V \partial_s \psi }{V}
%+  \partial_{ss}^2 \psi 
%-4 i \frac{\partial_{s} V}{V} \psi 
%-4 i  \partial_{s} \psi 
%- 4 \psi
%\right) 
%\\
&=
 iV _d
\left( 
\frac{S_1(V_d)}{V_d}\psi
+\tilde L_1 (\psi )
\right)
\end{align}
where
\begin{align*}
\tilde L_1 (\psi) = \varepsilon^2 \Bigl( \partial_{ss}^2 \psi
+\frac{2\partial_s V_d }{V_d}\partial_s \psi 
- 4i \partial_s \psi  \Bigr) .
\end{align*}
Far away from the vortices we have
\begin{align}
\label{S1far}
S_1(V_d e^{i\psi} )
&=
 i V_d e^{i\psi}
\left[
-i\frac{S_1(V_d)}{V_d}
+ \tilde L_1 (\psi)
+\varepsilon^2 i (\partial_s\psi)^2 
\right] .
\end{align}

We let 
\[
\tilde \eta(z ) = \eta_1(|z-\tilde d|-1) + \eta_1(|z+\tilde d|-1)
\]
with $\eta_1$ defined in  \eqref{eta1}.
Then we write  $S(v)=0$ as 
\begin{align*}
0 & =\tilde\eta   i V_d
\Bigg[
-i\frac{S_0(V_d)}{V_d} 
+ \tilde L_0 (\psi)
+ \frac{S_0(V_d)}{V_d}\psi 
- \frac{i}{V_d}L_0(\gamma(\psi)) 
- \frac{i}{V_d}N_0(i V_d \psi + \gamma(\psi))
\\
& \qquad \qquad 
-i \frac{S_1(V_d)}{V_d}
+ \tilde L_1 (\psi)
+ \frac{S_1(V_d)}{V_d}\psi
- \frac{i}{V_d} S_1(\gamma(\psi))
\Bigg]
\\
& \quad 
+ (1-\tilde\eta) 
i V_d e^{i \psi} 
\left[
- i \frac{S_0(V_d)}{V_d}
+ \tilde L_0 (\psi)
+ \tilde N_0(\psi)
- i \frac{S_1(V_d)}{V_d}
 + \tilde L_1 (\psi) + \varepsilon^2 i (\partial_s \psi)^2
\right] ,
%\\
%&=
%i V 
%\Bigg\{
%-i ( \tilde\eta + (1-\tilde\eta) e^{i\psi} )\frac{S(V)}{V}
%+( \tilde\eta + (1-\tilde\eta) e^{i\psi} ) (\tilde L_0 + \tilde L_1) \psi 
%+ \tilde\eta \frac{S(V)}{V}\psi
%\\
%& \qquad  \qquad 
%- \frac{i}{V} \tilde\eta L_0(\gamma_1(\psi)) 
%- \frac{i}{V} \tilde\eta N_0(i V \psi + \gamma_1(\psi))
%+(1-\tilde\eta) e^{i\psi} \tilde N_0(\psi)
%\\
%& \qquad  \qquad 
%-\frac{i}{V} \tilde \eta S_1(\gamma_1(\psi))
%+(1-\tilde\eta) \varepsilon^2 i (\partial_s\psi)^2
%\Bigg\} .
%\\
%&= 
%\eta   i V
%\Bigg[
%-i\frac{S(V)}{V} + L_1 \psi + \frac{S(V)}{V}\psi 
%- \frac{i}{V}L_0(\gamma_1(\psi)) 
%- \frac{i}{V}N_0(i V \psi + \gamma_1(\psi))
%\Bigg]
%\\
%& \quad + (1-\eta) 
%i V 
%\left[
%- i \frac{S_0(V)}{V}
%+L_1 \psi + N_1(\psi)
%\right]
%\\
%& \quad + (1-\eta) 
%i V i \psi
%\left[
%- i \frac{S_0(V)}{V}
%+L_1 \psi + N_1(\psi)
%\right]
%\\
%& \quad + (1-\eta) 
%i V (e^{i\psi}-1-i\psi)
%\left[
%- i \frac{S_0(V)}{V}
%+L_1 \psi + N_1(\psi)
%\right]
%\\
%&= 
%i V \Bigg\{
%-i\frac{S_0(V)}{V}
%+L_1\psi + \frac{S(V)}{V}\psi
%-\frac{i}{V} \eta L_0(\gamma_1(\psi))
%- \frac{i}{V} \eta N_0(i V \psi + \gamma_1(\psi))
%\\
%& \quad 
%+(1-\eta) i \psi L_1\psi 
%+(1-\eta) i \psi N_1(\psi) 
%\\
%& \quad 
%+ (1-\eta) 
%i V (e^{i\psi}-1-i\psi)
%\left[
%- i \frac{S_0(V)}{V}
%+L_1 \psi + N_1(\psi)
%\right]
%\Bigg\}
\end{align*}
that is, we use expressions \eqref{S0near}, \eqref{S1near} near the vortices and  \eqref{S0far}, \eqref{S1far} far from them.
Hence $S(v)=0$  is equivalent to 

\begin{align*}
\mathcal L^\varepsilon (\psi) + R + \mathcal N(\psi)=0
\end{align*}
where
\begin{align}
\nonumber
\mathcal L^\varepsilon (\psi)
& := 
 (\tilde L_0 + \tilde L_1)  (\psi )
+ \tilde\eta \frac{S(V_d)}{V_d}\psi
\\
\label{defError}
R &:= -i \frac{S(V_d)}{V_d}
\\
\nonumber
\mathcal N(\psi) &:=
\tilde\eta\Bigl( \frac{1}{\tilde \eta + (1-\tilde\eta) e^{i\psi} }-1\Bigr) \frac{S(V_d)}{V_d}\psi
\\
\nonumber
& \qquad 
- \frac{i}{V_d} \frac{\tilde\eta}{\tilde\eta + (1-\tilde\eta) e^{i\psi} } \Bigl\{ L_0(\gamma(\psi)) + S_1(\gamma(\psi))+ N_0(i V_d \psi + \gamma(\psi)) \Bigr\}
\\
\nonumber
& \qquad 
+ \frac{(1-\tilde\eta)e^{i\psi} }{ \tilde\eta + (1-\tilde\eta) e^{i\psi} } \Bigl\{\tilde N_0(\psi)
+\varepsilon^2 i (\partial_s\psi)^2 \Bigr\}.
\end{align}
Note that explicitly
\begin{align}
\label{calLeps}
\mathcal L ^\varepsilon (\psi) =
 \Delta \psi + 2\frac{\nabla V_d \nabla \psi }{V_d}
-2 i |V_d|^2 \IM(\psi) 
+\varepsilon^2 \Bigl( \partial_{ss}^2 \psi
+\frac{2\partial_s V_d }{V_d}\partial_s \psi 
- 4i \partial_s \psi  \Bigr) 
+ \tilde\eta \frac{S(V_d)}{V_d}\psi ,
\end{align}
and that for $|z\pm \tilde d|\geq 3$ the nonlinear terms take the form
\begin{align*}
\mathcal{N}(\psi) &=
\tilde N_0(\psi)
+\varepsilon^2 i (\partial_s\psi)^2 
\\
&=
 i (\nabla \psi)^2 + i |V_d|^2(e^{-2\IM(\psi)}-1+2\IM(\psi) )
 +\varepsilon^2 i (\partial_s\psi)^2 .
\end{align*}

\subsection{Another form of the equation near each vortex}
In order to analyze the equation near each vortex, it will be useful to write it in a translated variable. Namely, we set $\dd_j:=(-1)^{1+j}\dd$ for \(j=1,2\), we define $\z:=z-\dd_j$ and the function $\phi_j(\z)$ through the relation
\begin{equation}\label{defphi_j}
\phi_j(\z)=iW(\z)\psi(z), \ \ |\z|< \dd.
\end{equation}
That is,
\begin{equation*}\nonumber
\phi(z)=iV_d \psi(z)=\phi_j(\z)\alpha_j(z),\qquad \mbox{where}\qquad \alpha_j(z)=\frac{V_d(z)}{W(z-\tilde{d}_j)}.
\end{equation*}
Hence in the translated variable the unknown \eqref{eq1} becomes, in $|\z|<\dd$,
\begin{equation*}
%\label{ansatznearvortices}
v(z)=\alpha_j(\z)\left(W(\z)+\phi_j(\z)+(1-\eta_1(\z))W(\z)\left[ e^{\frac{\phi_j(\z)}{W(\z)}}-1-\frac{\phi_j(\z)}{W(\z)}\right] \right).
\end{equation*}
We set 
\begin{equation*}
E:=S(V_d).
\end{equation*}
For $\phi_j, \psi$ linked through formula \eqref{defphi_j} we define
\begin{eqnarray}\label{defL_j}
L_j^\e(\phi_j)(\z)&:=&iW(\z)\mathcal{L}^\e(\psi)(\z+\dd_j) =\frac{L^\e_d(\phi)(z)}{\alpha_j(z)}+(\eta_1-1)\frac{E(z)}{V_d(z)}\phi_j(\tilde{z}) \nonumber \\
&=&\frac{L^\e_d(\phi_j \alpha_j)(z)}{\alpha_j(z)}+(\eta_1-1)\frac{E(z)}{V_d(z)}\phi_j(\tilde{z}),
\end{eqnarray}
with $L^\e_d$ defined by 
 \begin{equation}\label{Lepsd}
 L^\e_d(\phi):=\Delta \phi +\e^2(\p^2_{ss} \phi-4i\p_s \phi-4\phi)+(1-|V_d|^2)\phi-2\RE(\overline{V_d}\phi)V_d.
 \end{equation}
Let us also define
\begin{equation*}\begin{split}
S_2(V):=&\, \p^2_{rr}V +\frac{1}{r}\p_rV+\frac{1}{r^2}\p_{ss}V+\e^2(\p^2_{ss}V -4i\p_sV-4V), \nonumber \\
S_3(V):=&\, \p^2_{rr}V +\frac{1}{r}\p_rV+\frac{1}{r^2}\p_{ss}V+\e^2(\p^2_{ss}V-4i\p_sV). \nonumber
\end{split}\end{equation*} Notice that
\begin{equation*}\nonumber
E(z)=S_2(\alpha_j(z)W(\z))+(1-|W|^2|\alpha_j|^2)W(\z)\alpha_j(z),
\end{equation*}
and thus, using the equation satisfied by $W$,
\begin{equation*}\begin{split}
E=&\, WS_2(\alpha_j)+(1-|W|^2|\alpha_j|^2)\alpha_j W+2\nabla \alpha_j \nabla W+2\e^2\p_s\alpha_j \p_s W +\alpha_j S_3(W) \nonumber \\
=&\, WS_3(\alpha_j)-4\e^2W\alpha_j+(1-|W|^2|\alpha_j|^2)\alpha_j W+ 2\nabla \alpha_j \nabla W +2\e^2\p_s\alpha_j \p_s W\\
&\,+\alpha_j[\e^2(\p^2_{ss}W-4i\p_sW)-(1-|W|^2)W ]. \nonumber
\end{split}\end{equation*}
This allows us to conclude
\begin{eqnarray}\label{eq:defL_j}
L_j^\e(\phi_j)& =&L^0(\phi_j)+\e^2(\p^2_{ss} \phi_j-4i\p_s \phi_j-4\phi_j) +2(1-|\alpha_j|^2)\RE(\overline{W}\phi_j)W  \nonumber \\
& &-\left(2\frac{\nabla \alpha_j}{\alpha_j}\frac{\nabla W}{W}+2\e^2\frac{\p_s \alpha_j}{\alpha_j}\frac{\p_s W}{W}+\e^2 \frac{(\p^2_{ss}W-4i\p_sW)}{W}+4\e^2 \right)\phi_j \nonumber  \\
& &+2\frac{\nabla \alpha_j}{\alpha_j}\nabla \phi_j+2\e^2\frac{\p_s \alpha_j}{\alpha_j}\p_s \phi_j
+\tilde{\eta}\frac{E_j}{V_d^j}\phi_j, 
\end{eqnarray}
where $V_d^j=V_d(\z+d_j)$, \(E_j=S(V_d^j)\) and $L^0$ is the linear operator defined by \eqref{L0}.

Let us point out that for $|\z|<2$
\begin{equation}\label{eq:estimatesonalpha_j}
 |\alpha_j(\z)|=1+O_\e(\e^2|\log \e|), \ \ \ \nabla \alpha_j(\z)=O_ \e(\e\sqrt{|\log \e|}), \ \ \  \Delta \a_j=O_\e(\e^2|\log \e|).
\end{equation}
With this in mind, we can see that the linear operator $L_j^\e$ is a small perturbation of $L^0$.

\subsection{Symmetries assumptions on the perturbation}
We end this section by making use of the symmetries of the problem. Using the notation  $z=x_1+ix_2=re^{is}$ we remark that $V_d$  satisfies
%We remark that the approximation $V_d$ satisfies
%\begin{equation}
%V_d(-x_1,x_2)=\overline{V}_d(x_1,x_2) \ \ \ \text{ and } V_d(x_1,-x_2)=\overline{V}_d(x_1,x_2)
%\end{equation}
%where $x_1=r\cos s, x_2=r\sin s$ and where we have written $V_d(x_1,x_2)=V_d(r,s)$ with some abuse of notation. By using the complex notation we have
\begin{equation*}
V_d(-x_1,x_2)=\overline{V}_d(x_1,x_2) \quad \text{ and } \quad V_d(x_1,-x_2)=\overline{V}_d(x_1,x_2).
\end{equation*}
We also remark that these symmetries are compatible with the solution operator $S$, that is, if $S(V)=0$ and $U(z)=\overline{V}(-x_1,x_2)$, then $S(U)=0$, and the same for $U(z)=\overline{V}(x_1,-x_2)$. Thus we look for a solution $V$ satisfying
\begin{equation*}
V(-x_1,x_2)=\overline{V}(x_1,x_2), \ \ \ \ V(x_1,-x_2)=\overline{V}(x_1,x_2),
\end{equation*}
what drives to ask
\begin{equation}\label{eq:eqsymmetriesofpsi}
\psi(x_1,-x_2)=-\overline{\psi}(x_1,x_2), \ \ \ \ \ \psi(-x_1,x_2)=-\overline{\psi}(x_1,x_2).
\end{equation}

\section{Error estimates of the approximated solution}\label{III}
In this section we compute the error of the approximation $V_d$ defined as $R= -i \frac{S(V_d)}{V_d} $ in \eqref{defError}.

In order to measure the size of the error of our approximation we fix  $0<\alpha,\sigma<1$. We recall that
\begin{align}
\label{defTildeDj}
%d_j:=(-1)^{j+1}d,\qquad\hat{d}_j:=(-1)^{1+j}\hat{d},\qquad
\tilde{d}_j:=(-1)^{1+j}\tilde{d},
\end{align}
and we define the norm
\begin{align}
\|h\|_{**}
&:=\sum_{j=1}^2 \| V_d h\|_{C^\alpha(\l_j<3)}
+ \sup_{\rho_1>2,\rho_2>2} 
\Bigl[
\frac{| \RE(h)|}{\rho_1^{-2}+\rho_2^{-2}+\varepsilon^2}
+ \frac{ |\IM(h)|}{\rho_1^{-2+\sigma}+\rho_2^{-2+\sigma}+\varepsilon^{\sigma-2}} 
\Bigr] \nonumber
\\
& \qquad 
+ \sup_{2<|z-\tilde d_1|<2R_\varepsilon , \, 
2<|z-\tilde d_2|<2R_\varepsilon} 
\frac{ [ \RE(h)]_{\alpha,B_{|z|/2}(z)}}{ |z-\tilde d_1|^{-2-\alpha} +  |z-\tilde d_2|^{-2-\alpha}} \nonumber
\\
& \qquad 
+ \sup_{2<|z-\tilde d_1|<2R_\varepsilon , \, 
2<|z-\tilde d_2|<2R_\varepsilon} 
\frac{ [ \IM(h)]_{\alpha,B_{1}(z)}}{ |z-\tilde d_1|^{-2+\sigma} +  |z-\tilde d_2|^{-2+\sigma}} , \label{def:norm_**}
\end{align}
where $\|f\|_{C^\alpha(D)} = \|f\|_{C^{0,\alpha}(D)}$, and where we have used the notation
\begin{align}
[f]_{\alpha,D}
& = \sup_{x,y\in D , \, x\not=y}\frac{|f(x)-f(y)|}{|x-y|^\alpha} \label{def:Holder_norm}
\\
\|f\|_{C^{k,\alpha}(D)} 
&=
\sum_{j=0}^k \|D^jf\|_{L^\infty	(D)} \label{def:Ckalpha_norm}
+ [D^k f]_{\alpha,D} .
\end{align}

\begin{proposition}\label{prop:sizeoferror}
Let $V_d$ given by \eqref{eq:Approx2-dimensionalrescaled}, and denote
$$S(V_d)=E=iV_d R=iV_d(R_1+iR_2).$$
Then
\begin{equation*}
\| R\|_{**}\leq \frac{C}{|\log \e|}.
\end{equation*}
\end{proposition}

\begin{proof}

Let us write
\[
V_d = W^a W^b
\]
where
\[
W^a(z):=W(z-\dd) , \quad W^b(z) :=W(z+\dd) .
\]
We want to estimate $E:=S(V_d)$, i.e., how {\it far} our approximation is to be a solution. 

By symmetry it suffices to compute the error in the region $(x_1,x_2)\in \R^+\times \R$.
We denote
\begin{equation*}
\l_1 e^{i\theta_1}:=re^{is}-\dd, \ \ \ \ \l_2 e^{i \theta_2}:=re^{is}+\dd.
\end{equation*}
We recall that $\Delta (fg)=g\Delta f+f\Delta g+ 2\nabla f \nabla g$ and thus with \(S_0\) defined in \eqref{def:S_0S_1}:
\begin{equation*}\begin{split}
S_0(V_d)=&\,(W^a_{x_1x_1}+W^a_{x_2x_2})W^b+(W^b_{x_1x_1}+W^b_{x_2x_2})W^a  \nonumber \\
&\,+2(W^a_{x_1}W^b_{x_1}+W^a_{x_2}W^b_{x_2}) +(1-|W^aW^b|^2)W^aW^b.
\end{split}\end{equation*}
Using the fact that $\Delta W +(1-|W|^2)W=0$ in $\R^2$ we conclude that
\begin{equation}\label{S0Vd}
S_0(V_d)=2(W^a_{x_1}W^b_{x_1}+W^a_{x_2}W^b_{x_2})+\left(1-|W^aW^b|^2+|W^a|^2-1+|W^b|^2-1\right)W^aW^b.
\end{equation}
We estimate the size of this error separately in two different regions, near the vortices and far from them. Notice first that, since we work in the half-plane $\R^+\times \R$, we have
\begin{equation*}
\l_2\geq \dd \geq \frac{C}{\e \sqrt{|\log \e |}}
\end{equation*}
for some $C>0$ of order $1$.
\medskip

\noindent {\it Step 1:} Estimate of $S_0(V_d)$ near one vortex, i.e., when $|re^{is} -\dd| <3$

\noindent Writing $W=W(\l e^{i\theta})$ we have
\begin{eqnarray*}
W_{x_1}= e^{i\theta} \left( \r'(\l)\cos \theta-i\frac{\r(\l)}{\l}\sin \v \right), \ \ \
W_{x_2} = e^{i\v} \left( \r'(\l)\sin \v+i\frac{\r(\l)}{\l}\cos \v \right). \nonumber
\end{eqnarray*}
We define $\r_1:=\r(\l_1)$ and $\r_2:=\r(\l_2)$ and we obtain
\begin{equation*}\begin{split}
W^a_{x_1}W^b_{x_1} =&\,e^{i(\v_1+\v_2)}\Bigl\{ \r'_1\r'_2\cos \v_1 \cos \v_2 -\frac{\r_1\r_2}{\l_1 \l_2}\sin \v_1 \sin \v_2 \\
&-i\left[ \frac{\r'_1\r_2}{\l_2}\cos \v_1 \sin \v_2+\frac{\r'_2\r_1}{\l_1}\cos \v_2 \sin \v_1 \right]  \Bigr\},
\end{split}\end{equation*}
\begin{equation*}\begin{split}
W^a_{x_2}W^b_{x_2} =&\,e^{i(\v_1+\v_2)}\Bigl\{ \r'_1\r'_2\sin \v_1 \sin \v_2 -\frac{\r_1\r_2}{\l_1 \l_2}\cos\v_1 \cos \v_2 \\
&+i\left[ \frac{\r'_1\r_2}{\l_2}\sin \v_1 \cos \v_2+\frac{\r'_2\r_1}{\l_1}\cos \v_1 \sin \v_2 \right]  \Bigr\}.
\end{split}\end{equation*}
Since $\r'(\l)=\frac{1}{\l^3}+O(\frac{1}{\l^4})$ when $\l\rightarrow +\infty$ (see Lemma \ref{lem:propertiesofrho}) and $\l_2 \geq \frac{C}{\e \sqrt{|\log \e|}}$ we can see that
\begin{equation*}\nonumber
\|W^a_{x_1}W^b_{x_1}+W^a_{x_2}W^b_{x_2} \|_{L^\infty(\l_1<3)}\leq C\e \sqrt{|\log \e|}
\end{equation*}
when $\e$ is small and for some $C>0$. Using now that $\r(\l)=1-\frac{1}{2\l^2}+O(\frac{1}{\l^4})$ when $\l\rightarrow +\infty$ we obtain
\begin{equation*}\nonumber
\|(1-|W^aW^b|^2+ |W^a|^2-1+|W^b|^2-1)W^aW^b \|_{L^\infty(\l_1<3)}\leq C\e^2 |\log \e|,
\end{equation*}
and thus
\begin{equation}\label{eq:Errorestimate01}
\|E_0\|_{L^\infty(\l_1<3)}=\|S_0(V_0) \|_{L^\infty(\l_1<3)} \leq C \e \sqrt{|\log \e|}.
\end{equation}
Similarly,
\begin{equation}\label{eq:GradErrorestimate01}
\|\nabla E_0\|_{L^\infty(\l_1<3)}\leq C \e \sqrt{|\log \e|}.
\end{equation}
%In a similar way we have
%\begin{equation}\label{eq:Errorestimate02}
%\|E_0\|_{L^\infty(\l_2<3)}=\|S_0(V_0) \|_{L^\infty(\l_1<3)} \leq C \e \sqrt{|\log \e|}.
%\end{equation}
\medskip

\noindent {\it Step 2:} Estimate of $S_0(V_d)$ {\it far away} from the vortices, i.e., when $|re^{is}-\dd|>2$

\noindent We write $E_0=S_0(V_d)=iV_d(R_0^1+iR^2_0)$ with
\begin{equation*}\begin{split}
R_0^1=&\,2(\sin \v_1\cos \v_2-\cos \v_1\sin \v_2) \left(\frac{\r'_1}{\l_2\r_1}-\frac{\r'_2}{\l_1\r_2}\right) = 2\sin (\v_1-\v_2)\left(\frac{\r'_1}{\l_2\r_1}-\frac{\r'_2}{\l_1\r_2}\right),\\
R_0^2=&2\cos(\v_1-\v_2)\left(\frac{-\r'_1\r'_2}{\r_1\r_2}+\frac{1}{\l_1\l_2}\right)-\left(1-\r_1^2 \r_2^2+\r_2^2-1+\r_1^2-1 \right).
\end{split}\end{equation*}
Using that $\l_1 \leq \l_2$ and $\l_2 \geq C/\e\sqrt{|\log \e|}$, along with Lemma \ref{lem:propertiesofrho}, we conclude
%We assume first that $\l_1=|re^{is} -\dd| > \frac{1}{\e \sqrt{|\log \e|}}$ and $\l_2=|re^{is}+\dd|> \frac{1}{\e \sqrt{|\log \e|}}$. Then we obtain
\begin{equation}\label{eq:errorr01}
|R^1_0| \leq C\e \sqrt{|\log \e|}\frac{1}{\l_1^3}.
\end{equation}
Using again Lemma \ref{lem:propertiesofrho} we obtain
\begin{equation*}\begin{split}
1-\r_1^2 \r_2^2+\r_2^2-1+\r_1^2-1=&1-\r_1^2 +O\left(\frac{1}{\l_2^2}\right)\r_1^2+O\left(\frac{1}{\l_2^2}\right)+\r_1^2-1 \leq C \frac{1}{\l_2}\frac{1}{\l_1},
\end{split}\end{equation*}
and hence
\begin{equation}\label{eq:errorr02}
|R^2_0| \leq C \frac{1}{\l_2 \l_1} \leq C \left(\e|\log \e|^{1/2}\right)^\sigma \frac{1}{\l_1^{2-\sigma}}, \quad \forall \ 0<\sigma<1.
\end{equation}
To see that the previous inequality holds we can distinguish the cases \(2<\l_1<\dd<\l_2\) and \(\dd<\l_1<\l_2\). We remark that we also have
\begin{equation}
\nonumber
|\nabla R_0^1| \leq \frac{C \e \sqrt{|\log \e|}}{\l_1^4}, \ \ \ |\nabla R_0^2|\leq \frac{C (\e \sqrt{|\log \e|})^\sigma}{\l_1^{3-\sigma}}.
\end{equation}
\medskip

\noindent {\it Step 3:} Estimates of $S_1(V_d)$

\noindent We recall that $\r_1:=\r(\l_1)$ and $\r_2:=\r(\l_2)$. Thus $V_d(r,s)=:\r_1\r_2 e^{i(\v_1+\v_2)}$ with
\begin{equation*}\nonumber
\l_1= \sqrt{(r\cos s -\dd)^2+r^2\sin^2s}, \ \ \ \ \ \l_2= \sqrt{(r\cos s+\dd)^2+r^2\sin^2s},
\end{equation*}
\begin{equation*}
e^{i\v_1}= \frac{(r\cos s-\dd)+ir\sin s}{\l_1}, \ \ \ \ e^{i\v_2}= \frac{(r\cos s+\dd)+ir\sin s}{\l_2}. \nonumber
\end{equation*}
We have
\begin{equation*}\begin{split}
\p_sV_d = &\,\left[\p_s \l_1 \r_1'\r_2
+\p_s\l_2 \r_2' \r_1+i \p_s(\v_1+\v_2)\r_1\r_2 \right]e^{i(\v_1+\v_2)} \nonumber  \\
\p^2_{ss}V_d =&\, \Bigl\{ \Bigl[ \p^2_{ss}\l_1 \r_1'\r_2+\p^2_{ss}\l_2\r_2' \r_1 +(\p_s\l_1)^2\r_1''\r_2+ (\p_s\l_2)^2\r_2''\r_1 \nonumber \\
& \,+2\p_s\l_1\p_s\l_2\r_1'\r_2'
-[\p_s(\v_1+\v_2) ]^2\r_1 \r_2 \Bigr] \nonumber \\
& \,+i\Bigl[ 2\p_s(\v_1+\v_2) (\p_s\l_1 \r_1'\r_2+\p_s\l_2\r_2'\r_1)+\p^2_{ss}(\v_1+\v_2)\r_1\r_2\bigr]\Bigr \}e^{i(\v_1+\v_2)},\nonumber
\end{split}\end{equation*}
and thus
\begin{equation*}\begin{split}
\e^2(\p^2_{ss}&V_d-4i\p_sV_d-4V_d)= \e^2\Bigl\{ (\p_s\l_1)^2\r_1''\r_2 \nonumber +(\p_s\l_2)^2\r_2''\r_1+2(\p_s\l_1)(\p_s\l_2) \r'_1\r'_2 \\
&+\p^2_{ss}\l_1\r_1'\r_2+\p^2_{ss}\l_2\r_2'\r_1-([\p_s(\v_1+\v_2)]^2-4\p_s(\v_1+\v_2)+4 )\r_1\r_2  \Bigr\}e^{i(\v_1+\v_2)}  \\
&+i\e^2 \Bigl\{ \p^2_{ss}(\v_1+\v_2)\r_1 \r_2+\left(2[\p_s(\v_1+\v_2)]-4\right)[\p_s\l_1\r_1'\r_2+\p_s\l_2\r_2'\r_1 ]\Bigr\}e^{i(\v_1+\v_2)}.
\end{split}\end{equation*}
We also need to compute the following derivatives,
\begin{equation}\nonumber%\label{eq:derivative_s_l}
\p_s\l_1 =\frac{r\dd \sin s}{\l_1}=\dd \sin \v_1, \ \  \p_s \l_2 = \frac{-r\dd \sin s}{\l_2}=-\dd \sin \v_2,
\end{equation}
\begin{equation*}\begin{split}
\p_{ss}^2 \l_1 =& \,\frac{r\dd \cos s}{\l_1}-\frac{r^2\dd^2\sin^2s}{\l_1^3}=\dd \cos \v_1+\dd^2\frac{\cos^2 \v_1}{\l_1}, \nonumber \\
\p_{ss}^2 \l_2 =&\, \frac{-r\dd \cos s}{\l_2}-\frac{r^2 \dd^2\sin^2s}{\l_2^3}=-\dd \cos \v_2+\dd^2\frac{\cos^2 \v_2}{\l_2}. \nonumber
\end{split}\end{equation*}
Now we can check that
\begin{equation}\nonumber%\label{eq:derivative_s_t1}
\p_s\v_1 =1+\frac{\dd}{\l_1^2}(r \cos s-\dd)=1+\frac{\dd\cos \v_1}{\l_1}, \quad
%\end{equation}
%\begin{equation}\label{eq:derivative_s_t2}
\p_s \v_2 = 1-\frac{\dd}{\l_2^2}(r\cos s+\dd)=1-\frac{\dd \cos \v_2}{\l_2}
\end{equation}
\begin{equation}\begin{split}
\p^2_{ss} \v_1=&\, \frac{-\dd r\sin s}{\l_1^4}(\l_1^2+2\dd(r\cos s-\dd))= \frac{-\dd\sin \v_1}{\l_1}-\frac{2\dd^2\sin \v_1\cos \v_1}{\l_1^2}, \nonumber \\
\p^2_{ss} \v_2 =&\, \frac{\dd r\sin s}{\l_2^4}(\l_2^2-2\dd(r\cos s+\dd))=\frac{\dd \sin \v_2}{\l_2}-\frac{2\dd^2\sin \v_2\cos \v_2}{\l_2^2}, \nonumber
\end{split}\end{equation}
and besides,
$$
[\p_s(\v_1+\v_2)]^2-4\p_s(\v_1+\v_2)+4=\dd^2\left(\frac{\cos \v_1}{\l_1}-\frac{\cos \v_2}{\l_2} \right)^2,$$
$$
\p^2_{ss}(\v_1+\v_2)=\dd \left(\frac{\sin \v_2}{\l_2}-\frac{\sin \v_1}{\l_1}\right)-2\dd^2\left( \frac{\sin \v_1 \cos \v_1}{\l_1^2}+\frac{\sin \v_2\cos \v_2}{\l_2^2} \right),
$$
$$
2\p_s(\v_1+\v_2)-4=2\dd \left(\frac{\cos \v_1}{\l_1}-\frac{\cos \v_2}{\l_2} \right).$$
Hence we obtain
\begin{equation}\begin{split}\label{eq:decompoerror}
S_1(V_d)= &\,\Bigl\{ \frac{\d^2}{|\log \e|} \Bigl(\r_1''\r_2 \sin^2\v_1 +\r_2''\r_1 \sin^2 \v_2 -\r_1'\r_2'\sin \v_1 \sin \v_2  \\
&\qquad+ \frac{\cos^2\v_1}{\l_1} \r_1' \r_2 +\frac{\cos^2 \v_2}{\l_2}\r_2' \r_1-\left(\frac{\cos \t_1}{\l_1}-\frac{\cos \t_2}{\l_2} \right)^2 \r_1 \r_2  \Bigr) \\
&\,\qquad + \frac{\e \d}{\sqrt{|\log \e|}}\Bigl(\cos \v_1 \r_1' \r_2-\cos \v_2 \r_2' \r_1 \Bigr) \\
+i&\Bigl[ \frac{\e \d}{\sqrt{|\log \e|}}\left( \frac{\sin \v_2}{\l_2}-\frac{\sin \v_1}{\l_1} \right)\r_1 \r_2 -\frac{2\d^2}{|\log \e|}\left(\frac{\sin \v_1 \cos \v_1}{\l_1^2}+\frac{\sin \v_2\cos \v_2}{\l_2^2} \right)\r_1\r_2 \\
&\qquad+\frac{2\d^2}{|\log \e|}\left(\frac{\cos \v_1}{\l_1}-\frac{\cos \v_2}{\l_2} \right)\left( \sin \v_1 \r_1' \r_2-\sin \v_2 \r_2' \r_1 \right)  \Big]\Bigr \}e^{i(\v_1+\v_2)}.
\end{split}\end{equation}
The conclusion of the proof follows from \eqref{eq:Errorestimate01}, \eqref{eq:GradErrorestimate01}, \eqref{eq:errorr01}, \eqref{eq:errorr02} and the following Lemma \ref{lem:orthoe}. We also use the symmetry of the problem.
\end{proof}
\begin{lemma}\label{lem:orthoe}
Let $S_1(V_d)=iV_dR_1=iV_d(R_1^1+iR_1^2))$. In the half-plane $\R^+\times \R$ we have that
\begin{equation}\label{eq:estimateEE}
\|S_1(V_d)\|_{L^\infty(\l_1<3)}\leq \frac{C}{|\log \e|}, \quad \|\nabla S_1(V_d)\|_{L^\infty(\l_1<3)} \leq \frac{C}{|\log \e|},
\end{equation}
and for $\l_1>2$:
\begin{equation}\label{eq:estimatesR_1}
|R_1^1| \leq \frac{C}{|\log \e|} \frac{1}{\l_1^2}, \ \ |\nabla R_1^1|\leq \frac{C}{|\log \e|} \frac{1}{\l_1^3}, \quad
%\end{equation}
%\begin{equation}\label{eq:estimatesR_2}
|R_1^2| \leq \frac{C}{|\log \e|} \frac{1}{\l_1^2}, \ \ |\nabla R_1^2| \leq \frac{C}{|\log \e|} \frac{1}{\l_1^3}.
\end{equation}
\end{lemma}

\begin{proof}
By using Lemma \ref{lem:propertiesofrho} we see that
\begin{equation*}
\|S_1(V_d)\|_{L^\infty(\l_1<3)} \leq \frac{C}{|\log \e|},\qquad \|\nabla S_1(V_d)\|_{L^\infty(\l_1<3)} \leq \frac{C}{|\log \e|}.
\end{equation*}
For $\l_1 > 2$ we have that
\begin{equation*}\begin{split}
-R_1^2=&\frac{\d^2}{|\log \e|} \Bigl(\frac{\r_1''}{\r_1}\sin^2 \theta_1+\frac{\cos^2\theta_1}{\l_1}\frac{\r_1'}{\r_1}-\left(\frac{\cos \t_1}{\l_1}-\frac{\cos \t_2}{\l_2} \right)^2  \Bigr) \\
&\;\;+\frac{\d^2}{|\log \e|} \Bigl( \frac{\r_2''}{\r_2}\sin^2\t_2-\r_1'\r_2' \sin \t_1 \sin \t_2 +\frac{\cos^2\t_2}{\l_2}\frac{\r_2'}{\r_2} \Bigr) \\
&\;\;+\frac{\d \e}{\sqrt{|\log \e|}}(\cos \t_1 \frac{\r_1'}{\r_1}-\cos \t_2 \frac{\r_2'}{\r_2} ).
\end{split}\end{equation*}
By using Lemma \ref{lem:propertiesofrho} and the fact that $\l_2 \geq \l_1>2$ we can see that
\begin{equation*}
\frac{\d^2}{|\log \e|} \Bigl|\frac{\r_1''}{\r_1}\sin^2 \theta_1+\frac{\cos^2\theta_1}{\l_1}\frac{\r_1'}{\r_1}-\left(\frac{\cos \t_1}{\l_1}-\frac{\cos \t_2}{\l_2} \right)^2  \Bigr| \leq \frac{C}{|\log \e|}\frac{1}{\l_1^2},
\end{equation*}
and
\begin{equation*}
\frac{\d \e}{\sqrt{|\log \e|}}\Bigl|\cos \t_1 \frac{\r_1'}{\r_1}-\cos \t_2 \frac{\r_2'}{\r_2} \Bigr| \leq \frac{C \e}{\sqrt{|\log \e|}}\frac{1}{\l_1^3} .
\end{equation*}
Besides by using also that $\l_2 \geq \dd \geq \d/(\e \sqrt{|\log \e|})$ we observe that
\begin{equation*}
\frac{\d^2}{|\log \e|} \Bigl| \frac{\r_2''}{\r_2}\sin^2\t_2-\r_1'\r_2' \sin \t_1 \sin \t_2 +\frac{\cos^2\t_2}{\l_2}\frac{\r_2'}{\r_2} \Bigr| \leq C \frac{\e}{\sqrt{|\log \e|}}\frac{1}{\l_1^2}.
\end{equation*}
Thus we obtain \eqref{eq:estimateEE} and  the third estimate in \eqref{eq:estimatesR_1}. By differentiating we can also obtain the fourth estimate.

Now for $\l_1>2$ we see that
\begin{equation*}\begin{split}
R_1^1=&\frac{\d \e}{\sqrt{|\log \e|}} \left(\frac{\sin \t_2}{\l_2}-\frac{\sin \t_1}{\l_1} \right) -\frac{2\d^2}{|\log \e|}\left(\frac{\sin \t_1 \cos \t_1}{\l_1^2}+\frac{\sin \t_2 \cos \t_2}{\l_2^2} \right) \\
&\;\;+\frac{2\d^2}{|\log \e|}\left[\frac{\cos \t_1 \sin \t_1 \r_1'}{\l_1 \r_1}-\frac{\cos \t_1 \sin \t_2 \r_2'}{\l_1 \r_2}\right] -\frac{2\d^2}{|\log \e|}\frac{\cos \t_2}{\l_2}(\sin \t_1 \frac{\r_1'}{\r_1}-\sin \t_2 \frac{\r_2'}{\r_2}).
\end{split}\end{equation*}
By using Lemma \ref{lem:propertiesofrho} and the fact that $\l_2 \geq \l_1$ we can see that the two first estimates of \eqref{eq:estimatesR_1} hold. Actually to prove the first estimate the only difficult term to handle is
$$\frac{\d \e}{\sqrt{|\log \e|}}\left( \frac{\sin \t_1}{\l_1}-\frac{\sin \t_2}{\l_2}\right).$$
In the region $(\R^+\times \R)\cap \{\l_1<C/(\e\sqrt{|\log \e|}) \}$ we have that $\e\sqrt{|\log \e|}/\l_1 <C/\l_1^2$. By using this and that $\l_2>\l_1$ we find that
\begin{equation}\label{eq:1}
\left|\frac{\dd \e}{\sqrt{|\log \e|}}\r_1\r_2\left( \frac{\sin \t_1}{\l_1}-\frac{\sin \t_2}{\l_2}\right)\right| \leq \frac{C}{|\log \e|} \frac{1}{1+\l_1^2}
\end{equation}
in $(\R^+\times \R)\cap \{\l_1<1/\e\sqrt{|\log \e|} \}$.

Now we use that $\l_1 \sin \t_1 =\l_2 \sin \t_2 =r \sin s$ to obtain that
$$\left( \frac{\sin \t_1}{\l_1}-\frac{\sin \t_2}{\l_2}\right)=\frac{\sin \t_1}{\l_1}\left(1- \frac{\l_1^2}{\l_2^2} \right).$$
But  $\l_2^2=\l_1^2+4\dd r \cos s=\l_1^2+4\dd\l_1 \cos \t_1 +\dd^2$. Thus when $\left| 4\dd\l_1 \cos \t_1 +\dd^2\right| <1$, which is true when $\l_1\geq \frac{C}{\e \sqrt{|\log \e|}}$ for an appropriate  constant $C>0$, we find that
$$\frac{\l_1^2}{\l_2^2}=1+4\frac{\d}{\l_1}\cos \t_1+O\left(\frac{\dd^2}{\l_1^2}\right).$$ Thus
\begin{equation}\label{eq:2}
\left|\frac{\d \e}{\sqrt{|\log \e|}}\r_1\r_2\left( \frac{\sin \t_1}{\l_1}-\frac{\sin \t_2}{\l_2}\right)\right| \leq \frac{C}{|\log \e|}\frac{1}{\l_1^2}
\end{equation}
in $(\R^+\times \R) \cap \{\l_1>C/(\e\sqrt{|\log \e|}) \}$. Combining estimates \eqref{eq:1} and \eqref{eq:2} and differentiating we arrive at the conclusion.
\end{proof}

Recall the polar  coordinates $\rho_j$, $\theta_j$ about $\tilde d_j$ defined by the relation $
z = \rho_j e^{i\theta_j} + \tilde d_j $.
We  can decompose a function \(h\) satisfying \(h(\overline{z})=-\overline{h}(z)\) in Fourier series in $\theta_j$ as
\begin{align}
& h = \sum_{k=0}^\infty h^{k,j} \label{def:h_Fourier}\\
& h^{k,j}(\rho_j,\theta_j) = h_1^{k,j}(\rho_j) \sin (k\theta_j) + i h_2^{k,j}(\rho_j) \cos(k\theta_j), \quad
h_1^{k,j}(\rho_j), h_2^{k,j}(\rho_j) \in \R , \nonumber
\end{align}
and define
\begin{align*}
h^{e,j} = \sum_{k \text{ even}} h^{k,j}, \quad
h^{o,j} = \sum_{k \text{ odd}} h^{k,j} .
\end{align*}

The definitions above can also be expressed by the following. 
Let $\mathcal{R}_j$ denote the reflection  across the line $\RE(z) = \tilde d_j$. Since  $\tilde d_j \in \R$, we have
\begin{equation}\label{def:reflection_R_j}
\mathcal{R}_j z = 2\tilde d_j-\RE(z) + i \IM(z).
\end{equation}

\noindent Then $h^{e,j}$ and $h^{o,j}$ have the symmetries
\begin{align}
\nonumber
h^{o,j} (\mathcal{R}_j z ) =  \overline{h^{o,j}(z)} ,\quad \quad  h^{e,j}(\mathcal{R}_j z ) = -\overline{h^{e,j}(z)} ,
\end{align}
and we can define equivalently
\begin{align}
\label{hoj}
h^{o,j}(z) &= \frac{1}{2}[ h(z) + \overline{h(\mathcal{R}_j z )}] , \quad
h^{e,j}(z) = \frac{1}{2}[ h(z) - \overline{h(\mathcal{R}_j z )}] .
\end{align}

It is convenient to consider a global function  $h^o$ defined as follows.
We introduce cut-off functions $\eta_{j,R}$, as
\begin{align}
\label{def:etajR}
\eta_{j,R}(z) = \eta_1 \Bigl(\frac{|z-\tilde d_j|}{R}\Bigr), 
\end{align}
where $\eta_1:\R\to[0,1]$ is a smooth function such that $\eta_1(t) = 1$ for $t \leq 1$ and $\eta_1(t)=0$ for $t\geq 2$.
Given $h:\mathbb{C}\to \mathbb{C}$ we define
\begin{align}
\label{Reps}
R_\varepsilon = \frac{\alpha_0}{\varepsilon |\log \varepsilon|^{\frac{1}{2}}} ,
\end{align}
where $\alpha_0>0$ is a small fixed constant that ensures $R_\varepsilon \leq \frac{1}{2}\tilde d$ and
\begin{align}
\label{def-ho}
% h^e &= \eta_{1,R_\varepsilon} h^{e,1} 
%+ \eta_{2,R_\varepsilon} h^{e,2} , \\
h^o &= \eta_{1,R_\varepsilon} h^{o,1} 
+ \eta_{2,R_\varepsilon} h^{o,2} \\
h^e& =h^o-h. \nonumber
\end{align}

For a complex function \(h=h_1+ih_2\) we introduce the new semi-norm
\begin{align}
|h|_{\sharp\sharp}
= \sum_{j=1}^2 \|  V_d h \|_{C^{0,\alpha}(\rho_j<4)}
+  \sup_{2<\rho_1< R_\varepsilon, \, 2<\rho_2 <R_\varepsilon}
\Bigl[
\frac{|h_1|}{\rho_1^{-1} + \rho_2^{-1}}
+\frac{ | h_2 | }{\rho_1^{-1+\sigma} + \rho_2^{-1+\sigma}}
\Bigr]. \label{def:semi_norm_sharpsharp}
\end{align}
We then have
\begin{proposition}\label{errorProp2}
Let $V_d$ given by \eqref{eq:Approx2-dimensionalrescaled} and denote $S(V_d)=E=iV_d R.$ Then we can write
$$R=R^o+R^e, \quad R^o=R^o_\alpha+R^o_\beta$$
with \(R^o\) defined analogously to \eqref{def-ho} and
\(R^o(\mathcal{R}_jz) =\overline{R^o (z)} \) in \(B_{R_\e}(\dd)\cup B_{R_\e}(-\dd)\)
\begin{align*}
|R^o_\alpha|_{\sharp\sharp} & \leq C \frac{\e}{\sqrt{|\log \e|}}, \quad \|R^o_\beta\|_{**} \leq C\e\sqrt{|\log \e|}, \quad \|R^e\|_{**}+\|R^o\|_{**}\leq \frac{C}{|\log \e|}
\end{align*}
\end{proposition}

\begin{proof}
The conclusion follows using the expression of $S_1(V_d)$ given by \eqref{eq:decompoerror}. More precisely, we define: 

\begin{align*}
r^{o,1} &:=
-i\frac{\d^2}{|\log \e|}\left( \frac{2\cos \theta_1 \cos \theta_2}{\rho_1 \rho_2}+\frac{\cos^2\theta_2}{\rho_2^2} \right) 
+\Bigl\{\frac{\d \e}{\sqrt{|\log \e|}} \frac{-\sin \theta_1}{\rho_1}-\frac{2\d^2}{|\log \e|}\frac{\sin \theta_2 \cos \theta_2}{\rho_2^2} \Bigr\}, \\
r^{o,2} &:=
-i\frac{\d^2}{|\log \e|}\left( \frac{2\cos \theta_1 \cos \theta_2}{\rho_1 \rho_2}+\frac{\cos^2\theta_1}{\rho_1^2} \right) 
+\Bigl\{\frac{\d \e}{\sqrt{|\log \e|}} \frac{+\sin \theta_2}{\rho_2}-\frac{2\d^2}{|\log \e|}\frac{\sin \theta_1 \cos \theta_1}{\rho_1^2} \Bigr\},
\end{align*}
\begin{equation*}
R^{j,o}_\alpha:= \frac{1}{2}[r^{o,1}(z)+\overline{r^{o,2}\left(\mathcal{R}_jz \right)}] \quad j=1,2
\end{equation*}
and
\begin{equation*}
R^o_\alpha:=\eta_{1,R_\e}R^{1,o}_\alpha+\eta_{2,R_\e}R^{2,o}_\alpha,
\end{equation*}
with \(R_\e\) defined by \eqref{Reps}. We can check that \(R^o_\alpha\) and \(R^o_\beta:=R^o-R^o_\alpha\) satisfy the desired properties.
\end{proof}

In the last step of the proof of Theorem \ref{th:main1}, when we aim at canceling the Lyapunov-Schmidt coefficient, we will need the following:

\begin{lemma}\label{lem:ortho_2}
In the region $B(\dd,\dd)$ we have that
\begin{equation*}
S_1(V_d)=\frac{\d}{|\log \e|}W_{x_2 x_2}^a W^b +\frac{\d \e}{\sqrt{|\log \e|}}W_{x_1}^a W^b + G
\end{equation*}
with
\begin{equation*}
%\label{eq:integralestimate}
\RE \int_{B(\dd,\frac{\d}{\e\sqrt{|\log \e|}})} W_{x_2 x_2}^a \overline{W}_{x_1}^a =0\quad \mbox{ and }\quad
%\end{equation}
%and
% \begin{equation}\label{eq:integralestimate}
 \RE\int_{B(\dd,\frac{\d}{\e\sqrt{|\log \e|}})} \frac{G}{W^b} \overline{W}_{x_1}^a=O_\e\left(\frac{\e}{\sqrt{|\log \e|}}\right).
 \end{equation*}
\end{lemma}
\begin{proof}
It suffices to observe that
\begin{equation*}\begin{split}
S_1(V_d)=E_1=&\,\frac{\d^2 \r_2e^{i(\v_1+\v_2)}}{|\log \e|} \left(\r_1''\sin^2\v_1+\cos^2\v_1\left(\frac{\r_1'}{\l_1}-\frac{\r_1}{\l_1^2}\right)+2i \cos \v_1 \sin \v_1 \left(\frac{\r'_1}{\l_1}-\frac{\r_1}{\l_1^2}\right) \right) \nonumber \\
&\,+ \frac{\d \e}{\sqrt{|\log \e|}} \left(\r'_1\r_2\cos \v_1-i\r_1\r_2\frac{\sin \v_1}{\l_1}\right) e^{i(\v_1+\v_2)}+G, \nonumber
\end{split}\end{equation*}
where
\begin{equation*}\begin{split}
G:=&\, \frac{\d \e e^{i(\t_1+\t_2)}}{\sqrt{|\log \e|}}\r_1\r_2'\cos \t_2  + \frac{\d^2 e^{i(\v_1+\v_2)}}{|\log \e|} \big( \r_2'' \r_1\sin^2\v_2+\r_1'\r_2'\sin \v_1 \sin \v_2 +\frac{\cos^2\v_2}{\l_2}\r_2'\r_1  \\
&+\left( \frac{2\cos \v_1 \cos \v_2}{\l_1 \l_2}+\frac{\cos^2\v_2}{\l_2^2} \right)\r_1 \r_2 \big) +ie^{i(\t_1+\t_2)}\frac{\d \e}{\sqrt{|\log \e|}}\frac{\sin \t_2}{\l_2}\r_1 \r_2\\
&-ie^{i(\v_1+\v_2)}\Bigl\{ \frac{2\d^2}{|\log \e|}\frac{\sin \v_2 \cos \v_2}{\l_2^2}\r_1 \r_2+\frac{2\d^2}{|\log \e|}\frac{\cos \v_2}{\l_2}(\sin \v_1 \r_1'\r_2+\sin \v_2\r_2'\r_1)\\
&+\frac{2\d^2}{|\log \e|}\frac{\cos \t_1}{\l_1}\sin \t_2 \r'_2 \r_1 \Bigr\}.
\end{split}\end{equation*}
\end{proof}

\section{A projected linear problem}

Given $h$ satisfying the symmetries \eqref{eq:eqsymmetriesofpsi} and appropriate decay, our aim in this section is to solve the linear equation
\begin{align}
\label{eq:linear}
\left\{
\begin{aligned}
& \L^\e(\psi)=h+c\sum_{j=1}^2\frac{\chi_j}{iW(z-\dd_j)} (-1)^jW_{x_1}(z-\dd_j)
\quad\text{in }\R^2
\\
& \RE\int_{B(0,4)} \chi \overline{\phi_j}W_{x_1}=0, \text{ with }\phi_j(z)=iW(z)\psi(z+\dd_j)\\
&\psi \text{ satisfies the symmetry }  \eqref{eq:eqsymmetriesofpsi}.
\end{aligned}
\right.
\end{align}
where 
\begin{equation}%\label{de:chi-chij}
\nonumber
\chi(z):=\eta_1\left(\frac{|z|}{2} \right), \ \ \ \chi_j(z):=\eta_1\left(\frac{\l_j}{2}\right)=\eta_1\left(\frac{|z-\dd_j|}{2} \right)
\end{equation}
with $\eta_1$ a smooth cut-off function such that $\eta_1(t)=1$ if $t\leq 1$ and $\eta_1(t)=0$ if $t\geq2$.

Thanks to the symmetries imposed on $\psi$ and $h$ it suffices to use one reduced parameter $c$ and not six as it should be the case  when working with two vortices, since the linearized operator around each has three elements in its kernel.

In order to find estimates on the solution of \eqref{eq:linear} we introduce some norms, for which we use the following notation. Let $\dd_j$, $j=1,2$ denote the center of each vortex as in \eqref{defTildeDj}.
We recall that $(\rho_j,\theta_j)$ are polar coordinates around $\dd_j$, that is, $
z = \rho_j e^{i\theta_j} + \dd_j $.

We will define two sets of norms.
The first one is the following:
given $\alpha , \sigma \in (0,1)$ and
$\psi:\mathbb C\to\mathbb C$ we define
\begin{align*}
\|\psi\|_* = \sum_{j=1}^2 \| V_d \psi \|_{C^{2,\alpha}(\rho_j<3)} + \| \RE(\psi) \|_{1,*} + \| \IM(\psi) \|_{2,*} 
\end{align*}
where, with \(\RE \psi=\psi_1, \IM \psi=\psi_2\),
\begin{align*}
\|\psi_1\|_{1,*} 
&= 
\sup_{\rho_1>2,\rho_2>2}
|\psi_1|
+\sup_{2<\rho_1<\frac{2}{\varepsilon} , \, 2<\rho_2<\frac{2}{\varepsilon}}
\frac{|\nabla \psi_1| }{\rho_1^{-1}+\rho_2^{-1}}
+ \sup_{r>\frac{1}{\varepsilon}}
\Bigl[
\frac{1}{\varepsilon}
|\partial_r \psi_1|
+ |\partial_s \psi_1|
\Bigr]
\\
& \qquad 
+\sup_{2<\rho_1<R_\varepsilon , \, 
2<\rho_2< R_\varepsilon }
\frac{|D^2 \psi_1| }{\rho_1^{-2}+\rho_2^{-2}}
\\
& \qquad 
+\sup_{2<|z-\tilde d_1|<R_\varepsilon , \, 
2<|z-\tilde d_2|<R_\varepsilon} 
\frac{[D^2\psi_1]_{\alpha,B_{|z|/2}(z)}}{ |z-\tilde d_1|^{-2-\alpha} +  |z-\tilde d_2|^{-2-\alpha}}
\\
\|\psi_2\|_{2,*} 
&=
\sup_{\rho_1>2,\rho_2>2}
\frac{|\psi_2|}{\rho_1^{-2+\sigma}+\rho_2^{-2+\sigma}+\varepsilon^{\sigma-2}}
+\sup_{2<\rho_1<\frac{2}{\varepsilon} , \, 2<\rho_2<\frac{2}{\varepsilon}}
\frac{|\nabla \psi_2| }{\rho_1^{-2+\sigma}+\rho_2^{-2+\sigma}}
\\
& \qquad 
+\sup_{r>\frac{1}{\varepsilon}}
\left[ 
\varepsilon^{\sigma-2}|\partial_r \psi_2|
+\varepsilon^{\sigma-1}|\partial_s \psi_2|
\right]
\\
& \qquad 
+\sup_{2<\rho_1<R_\varepsilon ,
 \, 2<\rho_2< R_\varepsilon }
\frac{|D^2 \psi_2| }{\rho_1^{-2+\sigma}+\rho_2^{-2+\sigma}}
\\
& \qquad 
+\sup_{2<|z-\tilde d_1|<R_\varepsilon , \, 
2<|z-\tilde d_2|<R_\varepsilon} 
\frac{[D^2\psi_2]_{\alpha,B_{1}(z)}}{ |z-\tilde d_1|^{-2+\sigma} +  |z-\tilde d_2|^{-2+\sigma}} .
\end{align*}
Here we have used the notation \eqref{def:Holder_norm}-\eqref{def:Ckalpha_norm}. We recall also that the norm for the right hand side $ h $ of \eqref{eq:linear} is defined by \eqref{def:norm_**}. One of the main results in this section is the following.
\begin{proposition}\label{prop:linearfull}
If $h$ satisfies \eqref{eq:eqsymmetriesofpsi} and $\|h\|_{**}<+\infty$ then for $\e>0$ sufficiently small there exists a unique solution $\psi=T_\e(h)$ to \eqref{eq:linear} with $\|\psi\|_*<\infty$.
Furthermore, there exists a constant $C>0$ depending only on  $\alpha, \sigma \in(0,1)$ such that this solution satisfies
\begin{equation*}
\|\psi\|_*\leq C\|h\|_{**}.
\end{equation*}
\end{proposition}

The proof of Proposition~\ref{prop:linearfull} is in \S\ref{sec:prop1}.

Although the existence and estimate in Proposition~\ref{prop:linearfull} are sufficient to solve a non-linear projected problem, the estimates for $\psi$ are too weak to enable us to solve the reduced problem. This means that they are too weak to justify that we can choose the parameter \(d\) such that the Lyapunov-Schmidt coefficient \(c\) in \eqref{eq:linear} vanishes. 
In order to address this difficulty we use that the largest part of the error and $\psi$ have a symmetry that makes them orthogonal to the kernel. To state the extra (partial) symmetry involved, let us consider $\psi:\mathbb C\to \mathbb C$. 
Recall the polar  coordinates $\rho_j$, $\theta_j$ about $\tilde d_j$ defined by the relation $
z = \rho_j e^{i\theta_j} + \tilde d_j $.
We  can decompose \(\psi\) in Fourier series in $\theta_j$ as in \eqref{def:h_Fourier}
%\begin{align*}
%& \psi = \sum_{k=0}^\infty \psi^{k,j} \\
%& \psi^{k,j}(\rho_j,\theta_j) = \psi_1^{k,j}(\rho_j) \sin (k\theta_j) + i \psi_2^{k,j}(\rho_j) \cos(k\theta_j), \quad
%\psi_j^{k,j}(\rho_j) \in \R ,
%\end{align*}
and define
\begin{align*}
\psi^{e,j} = \sum_{k \text{ even}} \psi^{k,j}, \quad
\psi^{o,j} = \sum_{k \text{ odd}} \psi^{k,j}.
\end{align*}

The intuitive idea is that $\psi^{o,j}$ is not orthogonal to the kernel near $\dd_j$ but small, while $\psi^{e,j}$ is large but orthogonal to the kernel near $\dd_j$ by symmetry.
With \(\mathcal{R}_j\) defined in \eqref{def:reflection_R_j}, we have
\begin{align}
\nonumber
\psi^{o,j} (\mathcal{R}_j z ) &=  \overline{\psi^{o,j}(z)} ,\quad 
\psi^{e,j}(\mathcal{R}_j z ) = - \overline{\psi^{e,j}(z)} ,
\end{align}
and we can define equivalently
\begin{align}
\label{psiejoj}
\psi^{o,j}(z) &= \frac{1}{2}[ \psi(z) + \overline{\psi(\mathcal{R}_j z )}] , \quad
\psi^{e,j}(z) = \frac{1}{2}[ \psi(z) - \overline{\psi(\mathcal{R}_j z )}] .
\end{align}

It is convenient to consider a global function  $\psi^o$ defined as follows: with \(R_\e\) given by \eqref{Reps} and \(\eta_{j,R}\) defined in \eqref{def:etajR} we set

\begin{align}
\label{def-psio}
% \psi^e &= \eta_{1,\frac{1}{2}R_\varepsilon} \psi^{e,1} 
%+ \eta_{2,\frac{1}{2}R_\varepsilon} \psi^{e,2} , \\
 \psi^o &= \eta_{1,\frac{1}{2}R_\varepsilon} \psi^{o,1} 
+ \eta_{2,\frac{1}{2}R_\varepsilon} \psi^{o,2}  ,
\end{align}
That is, $ \psi^o $ represents the {\em  odd} part of $\psi$ around each vortex $\dd_j$, localized with a cut-off function.
%We do the same with a right-hand side $h$ but we use a larger radius:
%\begin{align}
%\label{def-ho}
%% h^e &= \eta_{1,R_\varepsilon} h^{e,1} 
%%+ \eta_{2,R_\varepsilon} h^{e,2} , \\
%h^o &= \eta_{1,R_\varepsilon} h^{o,1} 
%+ \eta_{2,R_\varepsilon} h^{o,2}  .
%\end{align}

The part of $\psi$ that will be small, namely $\psi^o$ will be estimated in norms that allow for growth up to a certain distance. We do this because that part arises from terms in the error $R^o$ that are small, but decay slowly. To capture these behaviors we define
\begin{align*}
| \psi |_\sharp &= 
\sum_{j=1}^2 
|\log\varepsilon|^{-1}
\| V_d \psi \|_{C^{2,\alpha}(\rho_j<3)}
+ |\RE(\psi)|_{\sharp,1}+ |\IM(\psi)|_{\sharp,2} ,
\end{align*}
where
\begin{align}
\label{normSharp1}
|\psi_1|_{\sharp,1}
& =
\sup_{2<\rho_1 < R_\varepsilon , 2<\rho_2 < R_\varepsilon}
\Bigl[
\frac{|\psi_1|}{\rho_1 \log(2R_\varepsilon/\rho_1)+ \rho_2 \log(2R_\varepsilon/\rho_2)} 
+\frac{|\nabla \psi_1|}{ \log(2R_\varepsilon/\rho_1)+  \log(2R_\varepsilon/\rho_2)} \Bigr]
\\
\label{normSharp2}
|\psi_2|_{\sharp,2} 
&=
\sup_{2<\rho_1 < R_\varepsilon , 2<\rho_2 < R_\varepsilon}
\Bigl[  \frac{|\psi_2|+|\nabla \psi_2|}{\rho_1^{-1+\sigma} + \rho_2^{-1+\sigma}+\rho_1^{-1} \log(2R_\varepsilon/\rho_1) + \rho_2^{-1} \log(2R_\varepsilon/\rho_2)}
\Bigr] ,
\end{align}
and we recall
\begin{align*}
|h|_{\sharp\sharp}
= \sum_{j=1}^2 \|  V_d h \|_{C^{0,\alpha}(\rho_j<4)}
+  \sup_{2<\rho_1< R_\varepsilon, \, 2<\rho_2 <R_\varepsilon}
\Bigl[
\frac{|h_1|}{\rho_1^{-1} + \rho_2^{-1}}
+\frac{ | h_2 | }{\rho_1^{-1+\sigma} + \rho_2^{-1+\sigma}}
\Bigr] .
\end{align*}

%{\cb
%\begin{proposition}
%\label{prop:sharp2}
%Suppose that $h$ satisfies the symmetries \eqref{eq:eqsymmetriesofpsi}
%and  $\| h \|_{**}<\infty$.
%Suppose furthermore that  $h^o = h^o_1 + h^o_2$ where $|  h^o_1 |_{\sharp\sharp}<\infty$
%and $h^o_1$, $h^o_2$ satisfy 
%\begin{align}
%\nonumber
%h^o_1(\mathcal{R}_j z ) =  \overline{h^o_1(z)}  ,\quad 
%|z-\tilde d_j| < 2 R_\varepsilon.
%\end{align}
%and have support in $B_{2R_\varepsilon}(\dd_1) \cup B_{2R_\varepsilon}(\dd_2) $.
%Then the unique solution $\psi$ to \eqref{eq:linear}  with $\| \psi \|_*<\infty$ satisfies
%\begin{align*}
%| \psi^o |_{\sharp} \leq C ( |  h^o_1 |_{\sharp\sharp}
%+ \varepsilon|\log\varepsilon|^{\frac{1}{2}}( \|h^o_1\|_{**} + \|h-h^o_1
%\|_{**} ) ) .
%\end{align*}
%\end{proposition}
%
%
%The proof of Proposition~\ref{prop:sharp2} is in \S\ref{sec:prop2}.
%}

\begin{proposition}
\label{prop:sharp2b}
Suppose that $h$ satisfies the symmetries \eqref{eq:eqsymmetriesofpsi}
and  $\|h \|_{**}<\infty$.
Suppose furthermore that  $h^o $ defined by \eqref{def-ho} is decomposed as $ h^o = h^o_\alpha + h^o_\beta$ where $|  h^o_\alpha |_{\sharp\sharp}<\infty$ 
and $h^o_\alpha$, $h^o_\beta$ satisfy 
\begin{align}
\nonumber
h^o_k(\mathcal{R}_j z ) =  \overline{h^o_k(z)}  ,\quad 
|z-\tilde d_j| <  R_\varepsilon, \quad j=1,2 \ k=\alpha,\beta,
\end{align}
and have support in $B_{2R_\varepsilon}(\dd_1) \cup B_{2R_\varepsilon}(\dd_2) $.
Let us write $\psi = \psi^e + \psi^o$ with $\psi^o$ defined by \eqref{def-psio}. Then $\psi^o$ can be decomposed as $\psi^o = \psi^o_\alpha + \psi^o_\beta$, with each function supported in 
$B_{R_\varepsilon}(\dd_1) \cup B_{R_\varepsilon}(\dd_2) $ 
%$ \bigcup\limits_{j=1}^2 B_{R_\varepsilon}(\dd_j)$
and satisfying 
\begin{align}
\label{est:prop5.3-1}
|\psi_\alpha^o|_\sharp 
& \lesssim
|h_\alpha^o |_{\sharp\sharp} 
+ \varepsilon | \log\varepsilon|^\frac{1}{2} ( \|h_\alpha^o\|_{**}  
+  
\|  h - h^o\|_{**} )
\\
\label{est:prop5.3-2}
\| \psi_\beta^o \|_* 
&\lesssim
 \|h^o_\beta\|_{**},  
 \\
 \|\psi_\alpha^o\|_*+\|\psi_\beta^o\|_* & \lesssim \|h\|_{**}+\|h_\alpha^o\|_{**}+\|h_\beta^o\|_{**} \nonumber
\end{align}

and
\begin{align}
\nonumber
\psi^o_k(\mathcal{R}_j z ) =  \overline{\psi^o_k(z)}  ,\quad 
|z-\tilde d_j| <  R_\varepsilon, \quad j=1,2,\ k=\alpha,\beta.
\end{align}
\end{proposition}
The proof of Proposition~\ref{prop:sharp2b} is in \S\ref{sec:prop2}.

\subsection{First a priori estimate and proof of Proposition~\ref{prop:linearfull}}
\label{sec:prop1}
Here we obtain a priori estimates for solutions to 
\begin{align}
\label{eq:linearhomogeneous}
\left\{
\begin{aligned}
& \L^\e(\psi)=h \text{ in } \R^2 \\
& \RE\int_{B(0,4)} \chi_j\overline{\phi_j}W_{x_1}=0, \text{ with }\phi_j(z)=iW(z)\psi(z+\dd_j)\\
& \psi \text{ satisfies the symmetry }  \eqref{eq:eqsymmetriesofpsi}.
\end{aligned}
\right.
\end{align}
\begin{lemma}\label{FirstEstimate}
There exists a constant $C>0$  such that for all $\e$ sufficiently small and any solution $\psi$ of \eqref{eq:linearhomogeneous} with $\|\psi\|_*<\infty$ one has
\begin{align}
\label{est0a}
\|\psi\|_*\leq C\|h\|_{**}.
\end{align}
\end{lemma}

\begin{proof}[Proof of Lemma~\ref{FirstEstimate} ]

To prove Lemma~\ref{FirstEstimate} we will use first weaker norms.
For $\psi:\mathbb C\to \mathbb C$
we define
\begin{align*}
\|\psi\|_{*,0} = \sum_{j=1}^2 \| V_d \psi \|_{L^\infty(\rho_j<3)} + \| \RE(\psi) \|_{1,*,0} + \| \IM(\psi) \|_{2,*,0}
\end{align*}
where
\begin{align*}
\|\psi_1\|_{1,*,0} 
&= 
\sup_{\rho_1>2,\rho_2>2}
|\psi_1|
+\sup_{2<\rho_1<\frac{2}{\varepsilon} , \, 2<\rho_2<\frac{2}{\varepsilon}}
\frac{|\nabla \psi_1| }{\rho_1^{-1}+\rho_2^{-1}}
+ \sup_{r>\frac{1}{\varepsilon}}
\Bigl[
\frac{1}{\varepsilon}
|\partial_r \psi_1|
+ |\partial_s \psi_1|
\Bigr]
\\
\|\psi_2\|_{2,*,0} 
&=
\sup_{\rho_1>2,\rho_2>2}
\frac{|\psi_2|}{\rho_1^{-2+\sigma}+\rho_2^{-2+\sigma}+\varepsilon^{\sigma-2}}
+\sup_{2<\rho_1<\frac{2}{\varepsilon} , \, 2<\rho_2<\frac{2}{\varepsilon}}
\frac{|\nabla \psi_2| }{\rho_1^{-2+\sigma}+\rho_2^{-2+\sigma}}
\\
& \qquad 
+\sup_{r>\frac{1}{\varepsilon}}
\left[ 
\varepsilon^{\sigma-2}|\partial_r \psi_2|
+\varepsilon^{\sigma-1}|\partial_s \psi_2|
\right].
\end{align*}

In the expressions above the gradient of $\psi_j$ is $(\partial_{x_1}\psi_j,\partial_{x_2}\psi_j)$, where $z = (x_1,x_2)$. Since $z = x_1+ix_2 = r e^{is} = \rho_1 e^{i\theta_1} + \tilde d  =  \rho_2 e^{i\theta_2} - \tilde d  $ we have 
\begin{align*}
|\nabla \psi_j|^2 &= (\partial_{x_1}\psi_j)^2 + (\partial_{x_2}\psi_j)^2 
= (\partial_r \psi_j)^2 + \frac{1}{r^2} ( \partial_s \psi_j)^2
\\
&= 
 (\partial_{\rho_1}\psi_j)^2 + \frac{1}{\rho_1^2} ( \partial_{\theta_1} \psi_j)^2
= 
 (\partial_{\rho_2}\psi_j)^2 + \frac{1}{\rho_2^2} ( \partial_{\theta_2} \psi_j)^2.
\end{align*}
We define also the norm for the right hand side $ h = h_1 + i h_2$ of \eqref{eq:linear}:
\begin{align*}
\|h\|_{**,0}:=\sum_{j=1}^2 \| V_d)h\|_{L^\infty(\l_j<3)}
+ \sup_{\rho_1>2,\rho_2>2} 
\Bigl[
\frac{|\RE(h)|}{\rho_1^{-2}+\rho_2^{-2}+\varepsilon^2}
+ 
\frac{|\IM(h)|}{\rho_1^{-2+\sigma}+\rho_2^{-2+\sigma}+\varepsilon^{2-\sigma}} 
\Bigr].
\end{align*}

We claim that
there exists a constant $C>0$  such that for all $\e$ sufficiently small and any solution of \eqref{eq:linearhomogeneous} one has
\begin{align}
\label{est0}
\|\psi\|_{*,0}\leq C\|h\|_{**,0}.
\end{align}
To prove this, we assume by contradiction that there exist \(\e_n \rightarrow 0\) and \( \psi^{(n)}, h^{(n)}\) solutions of  \eqref{eq:linearhomogeneous} such that
\begin{equation*}
\|\psi^{(n)} \|_{*,0}=1, \qquad \| h^{(n)}\|_{**,0}=o_n(1).
\end{equation*}
We first work near the vortices $\tilde d_j$,
% and notice that, by symmetry, it is enough to consider the vortex $+\dd$. 
and work  with the function $\phi_j^{(n)}(z)=iW(z)\psi^{(n)}(z+\dd_j)$.

Since $\|\psi^{(n)}\|_{*,0}=1$ from Arzela-Ascoli's Theorem we can  extract a subsequence such that \( \tilde{\phi}^{(n)}_j \rightarrow \phi_0\) in \(C^0_{\text{loc}}(\R^2)\).  Passing to the limit in  \eqref{eq:linearhomogeneous} (we use \eqref{eq:defL_j} and \eqref{eq:estimatesonalpha_j}), we conclude that
\begin{equation*}
L^0(\phi_0)=0 \text{ in } \R^2,
\end{equation*}
with $L^0$ defined in \eqref{L0}. 
Moreover 
$\phi_0$ inherits the symmetry
$\phi_0(\bar{z})=\overline{\phi_0(z)}$.
From the estimate $\|\psi^{(n)}\|_{*,0}=1$ we deduce that $\phi_0 \in L^\infty(\R^2)$ and that $\psi_1 = \RE \left( \phi_0/iW\right)$,  
$\psi_2 = \IM \left( \phi_0/iW\right)$, satisfy
\begin{eqnarray*}
|\psi_1|+ |z| |\nabla \psi_1| \leq  C, \qquad |\psi_2|+|\nabla \psi_2| \leq  \frac{C}{|z|^{2-\sigma}} ,
\quad |z|>1.
\end{eqnarray*}
By  Lemma \ref{lem:ellipticestimatesL0} we deduce
that
\begin{equation*}
\phi_0=c_1W_{x_1},
\end{equation*}
for some $c_1\in \R$.

On the other hand, we can pass to the limit in the orthogonality condition
$$\RE \int_{B(0,4)} \chi\bar{\phi}_j^{(n)} W_{x_1}=0,$$
and obtain that necessarily $c_1=0$. Hence $\phi_j^{(n)} \rightarrow 0$ in $C^0_{\text{loc}}(\R^2)$. 
%By elliptic estimates, $\phi_j^{(n)} \rightarrow 0$ in $W^{2,p}_{\text{loc}}(\R^2)$. 
%We can also apply the same argument near $-\dd$.
% and for any $R>0$ we obtain
%\begin{equation}\label{eq:inner_estimate}
%\sum_{j=1}^2 \|\phi_j^{(n)}\|_{W^{2,p}(B_R(0))}=o_n(1).
%\end{equation}
Therefore 
\begin{align}
\label{eq:inner_estimate-2}
\text{$\psi^{(n)}\to 0$ uniformly on compact subsets of $\{ \rho_1 \geq 2, \rho_2 \geq 2\}$}.
\end{align}

Next we derive estimates \textit{far away} from the vortices. In the following we drop the superscript $n$  for simplicity. In $\{\l_1>2\}\cap \{\l_2>2 \}$ we have that $\psi^{(n)}=\psi$ solves 
\begin{align*}
h=
\Delta\psi  + 2 \frac{\nabla V_d\nabla \psi}{V_d}
-2 i |V_d |^2 \psi_2
+\varepsilon^2 \p^2_{ss}\psi
+ 
\varepsilon^2 
\Bigl(
2\frac{\partial_{s}V_d}{V_d}
-4i
\Bigr)
\p_s\psi 
\end{align*}
which for $\psi_1 = \RE(\psi)$, $\psi_2 = \IM(\psi)$ translates into 
the following system
\begin{multline}
\label{eqPsi100}
h_1 = 
\Delta \psi_1 
+\left(\frac{\nabla w_1}{w_1}+\frac{\nabla w_2}{w_2} \right) \nabla \psi_1
- \nabla (\theta_1+\theta_2) \nabla \psi_2
+ \varepsilon^2 \partial_{ss}^2 \psi_1
 \\  + 2\varepsilon^2 \left[ \left(\frac{\p_sw_1}{w_1}+\frac{\p_s w_2}{w_2} \right) \partial_s \psi_1-\p_s( \theta_1+\theta_2) \p_s \psi_2 \right]
+4 \varepsilon^2 \partial_s \psi_2  
\end{multline}
\begin{multline}\label{eqPsi200}
h_2 =
\Delta\psi_2 
+ \left(\frac{\nabla w_1}{w_1}+\frac{\nabla w_2}{w_2} \right) \nabla \psi_2
+ \nabla (\theta_1+\theta_2) \nabla \psi_1
-2 |V_d|^2 \psi_2 
+ \varepsilon^2 \partial_{ss}^2 \psi_2
 \\ + 2 \varepsilon^2 \left[ \left(\frac{\nabla w_1}{w_1}+\frac{\nabla w_2}{w_2} \right)  \partial_s \psi_2+ \p_s(\theta_1+\theta_2)\p_s \psi_1 \right]
-4 \varepsilon^2 \partial_s \psi_1 .
\end{multline}

We start by estimating $\psi_2$.
Since \( \psi_2\) satisfies \( \psi_2(x_1,-x_2)=\psi_2(x_1,x_2)\) and \( \psi_2(-x_1,x_2)=\psi_2(x_1,x_2)\) it is sufficient to obtain estimates for $\psi_2$ in the quadrant
\( \{x_1>0 , \ x_2>0 \}\).

Let $R>0$ be large fixed and $D_R = \{x_1>0 , \ x_2>0 \} \cap \{\rho_1>R\}$.
By the symmetries, $\psi_2$ satisfies homogeneous Neumann boundary condition at $x_1=0$ or $x_2=0$.

In $D_R$ we have \(|V_d|^2\geq c>0\) for some fixed positive constant \(c\). 
We consider \eqref{eqPsi200} in $D_R$ and rewrite it as
\begin{align}
\nonumber
\Delta \psi_2 +\e^2 \p_{ss}\psi_2-2|V_d|^2 \psi_2   =  p_2 ,
% h_2+4\e^2\p_s \psi_1-2\nabla (\theta_1+\theta_2) \nabla \psi_1 -2\e^2 \p_s(\theta_1+\theta_2)\p_s \psi_1 , 
\quad \text{in }D_R,
\end{align}
where

\begin{align*}
p_2 
&=
h_2 
- \left(\frac{\nabla w_1}{w_1}+\frac{\nabla w_2}{w_2} \right) \nabla \psi_2
- \nabla (\theta_1+\theta_2) \nabla \psi_1
- 2\varepsilon^2 \left[ \left(\frac{\p_sw_1}{w_1}+\frac{\p_s w_2}{w_2} \right) \partial_s \psi_2+\p_s( \theta_1+\theta_2) \p_s \psi_1 \right]
\\&
\quad
+4 \varepsilon^2 \partial_s  \psi_1 .
\end{align*}

We use  polar coordinates  \( (\rho_1,\theta_1) \)
around $\tilde d$ and  \( (\rho_2,\theta_2) \)
around $-\tilde d$, that is,
\begin{equation}
%\label{coords}
\nonumber
%x_1=r\cos s=\l\cos \v +\dd, \ \ \ \ x_2=r\sin s=\l \sin \v
r e^{is} 
= \rho_1 e^{i\theta_1} +\tilde d 
=  \rho_2 e^{i\theta_2} - \tilde d.
\end{equation}
From this we get that
\begin{align*}
\partial_r 
&=
\frac{1}{r}
\Bigl(
\rho_1 + \tilde d \cos\theta_1
\Bigr) \partial_{\rho_1}
- \frac{\tilde d \sin\theta_1}{r \rho_1} \partial_{\theta_1} , \quad
\partial_s  = \tilde d \sin\theta_1 \partial_{\rho_1}
+\Bigl( 1+\frac{\tilde d \cos\theta_1}{\rho_1}\Bigr)\partial_{\theta_1}  
\\
\partial_r 
&=
\frac{1}{r}
\Bigl(
\rho_2 - \tilde d \cos\theta_2
\Bigr) \partial_{\rho_2}
+ \frac{\tilde d \sin\theta_2}{r \rho_2} \partial_{\theta_2} , \quad
\partial_s  = -\tilde d \sin\theta_2 \partial_{\rho_2}
+\Bigl( 1-\frac{\tilde d \cos\theta_2}{\rho_2}\Bigr)\partial_{\theta_2}  .
\end{align*}

\noindent With these expressions and the asymptotic behaviour 
stated in Lemma~\ref{lem:propertiesofrho} we see that
\begin{align*}
\left|
\left(\frac{\nabla w_1}{w_1}+\frac{\nabla w_2}{w_2} \right) 
\nabla \psi_2\right|
&\leq 
\frac{C}{R^3}
\Bigl(
\frac{1}{\rho_1^{2-\sigma}}+\varepsilon^{2-\sigma}
\Bigr) \|\psi_2\|_{2,*,0}
\\
\Bigl|
\nabla (\theta_1+\theta_2) \nabla \psi_1 
\Bigr| 
&
\leq C 
\left (R^{-\sigma}+\varepsilon^\sigma
\right)
\Bigl(
 \frac{1}{\rho_1^{2-\sigma}}
+\varepsilon^{2-\sigma}
\Bigr) \|\psi_1\|_{1,*,0}
\\
\varepsilon^2
\left| 
\left( \frac{\p_s w_1}{w_1}+\frac{\p_s w_2}{w_2}\right)\p_s \psi_2
\right| 
& \leq 
C
\left( 
R^{-3} + \varepsilon
\right) 
\Bigl(
\frac{1}{\rho_1^{2-\sigma}}+\varepsilon^{2-\sigma}
\Bigr) \|\psi_2\|_{2,*,0}
\\ 
\e^2 \left| \p_s(\theta_1+\theta_2) \p_s \psi_1 \right| 
& 
\leq C 
\left (R^{-\sigma}+\varepsilon^\sigma
\right)
\Bigl(
 \frac{1}{\rho_1^{2-\sigma}}
+\varepsilon^{2-\sigma}
\Bigr) \|\psi_1\|_{1,*,0}
\\
\e^2 | \partial_s \psi_1| 
&
\leq C 
\left (R^{-\sigma}+\varepsilon^\sigma
\right)
\Bigl(
 \frac{1}{\rho_1^{2-\sigma}}
+\varepsilon^{2-\sigma}
\Bigr) \|\psi_1\|_{1,*,0}
\end{align*}

\noindent Since we assumed $\|\psi\|_{*,0}=1$, we get 
\begin{align*}
|p_2|\leq C (\|h\|_{**,0} 
+ R^{-\sigma}+\varepsilon^\sigma ) 
\Bigl(
 \frac{1}{\rho_1^{2-\sigma}}
+\varepsilon^{2-\sigma}
\Bigr) .
\end{align*}
We use a barrier of the form
\begin{equation*}
\mathcal{B}_2=M \left( \frac{1}{\rho_1^{2-\sigma}}+\e^{2-\sigma} \right)
\end{equation*}
with \(M=C\left(\|h\|_{**,0}+R^{-\sigma}+\e^{\sigma}+\|\psi_2\|_{L^\infty(B_R(\tilde{d}))} \right)\) and \(C>0\) is a large fixed constant. Note that
\begin{align*}
\partial_{ss}^2 \mathcal B_2 &=
\frac{\partial^2 \mathcal B_2}{\partial \rho_1^2}   \tilde d^2 \sin(\theta_1)^2
+ \frac{\partial \mathcal B_2}{\partial \rho_1}  \tilde d \cos(\theta_1)
\Bigl(1 + \frac{\tilde d}{\rho_1} \cos(\theta_1) 
\Bigr)
\\
&=
M(\sigma-2)(\sigma-3)
\frac{\tilde d^2 \sin(\theta_1)^2}{\rho_1^{4-\sigma}}
+ M  (\sigma-2) \tilde d \frac{\cos(\theta_1)}{\rho_1^{3-\sigma}}
\Bigl(1 + \frac{\tilde d}{\rho_1} \cos(\theta_1)
\Bigr)
\end{align*}
and so
\begin{align*}
\Delta \mathcal B_2 +\e^2 \p_{ss}\mathcal B_2-2|V_d|^2 \mathcal B_2  \leq - \tilde c M\Bigl( \frac{1}{\rho_1^{2-\sigma}}  + \varepsilon^{2-\sigma}\Bigl)
\quad \text{in }D_R 
\end{align*}
for some fixed $\tilde c>0$.
Thanks to a comparison principle in $D_R$ (a slight variant of Lemma \ref{lem:comparison_principle_Neumann}) and standard elliptic estimates
we get that
\begin{align}
\label{estPsi2a}
|\psi_2|
\leq 
C \Bigl(   \frac{1}{\rho_1^{2-\sigma}} + \varepsilon^{2-\sigma}\Bigr)
( \|h\|_{**,0} +R^{-\sigma} + \varepsilon^\sigma + \|\psi_2\|_{L^\infty(B_R(\tilde d))}) ,
\quad \text{in }D_R.
\end{align}
Standard elliptic estimates imply
\begin{align}
\label{estGradPsi2a}
|\nabla \psi_2|
\leq 
C \Bigl(   \frac{1}{\rho_1^{2-\sigma}} + \varepsilon^{2-\sigma}\Bigr)
( \|h\|_{**,0} +R^{-\sigma} + \varepsilon^\sigma + \|\psi_2\|_{L^\infty(B_R(\tilde d))}) ,
\quad \text{in }D_R \cap \Bigl\{ \rho_1 \leq \frac{2}{\varepsilon} \Bigr\}.
\end{align}
For points in $D_R$ with $\rho_1>\frac{1}{\varepsilon}$ 
we use the scaling
\begin{align*}
\tilde \psi(\tilde r, s) = \psi(\e^{-1}r,s)
\end{align*}
and we get the estimate
\begin{align}
\label{estGradPsi2b}
\varepsilon^{-1} |\partial_r \psi | 
+  |\partial_s \psi | \leq 
C \varepsilon^{2-\sigma}
( \|h\|_{**} + R^{-\sigma} + \varepsilon^{\sigma} + \|\psi_2\|_{L^\infty(B_R(\tilde d))}) ,
\end{align}
for points in $D_R$  with $\rho_1>\frac{1}{\varepsilon}$.

\noindent Combining \eqref{estPsi2a}, \eqref{estGradPsi2a} and \eqref{estGradPsi2b} we get
\begin{align}
\label{estPsi2B}
\|\psi_2\|_{2,*,0}
\leq C
( \|h\|_{**,0} +R^{-\sigma} + \varepsilon^\sigma + \|\psi_2\|_{L^\infty(B_R(\tilde d))}) .
\end{align}

We next estimate \(\psi_1\). We also use the symmetries satisfied by \(\psi_1\), that is,
\[
\psi_1(-x_1,x_2)=-\psi_1(x_1,x_2), \ \psi_1(x_1,-x_2)=- \psi_1(x_1,x_2)  ,
\]
to look at the equation for \(\psi_1\) in the quadrant \( \{x_1>0 \, \  x_2>0 \} \). 
Let us rewrite equation \eqref{eqPsi100} as 
\begin{align*}
\Delta\psi_1 + \varepsilon^2 \partial_{ss}^2 \psi_1 = p_1
\end{align*}
where 
\begin{align*}
p_1 & =
h_1 
-\left(\frac{\nabla w_1}{w_1}+\frac{\nabla w_2}{w_2} \right) \nabla \psi_1
+ \nabla (\theta_1+\theta_2) \nabla \psi_2
\\ 
& \quad - 2\varepsilon^2 \left[\left(\frac{\p_s w_1}{w_1}+\frac{\p_s w_2}{w_2} \right)\p_s \psi_1-\p_s(\theta_1+\theta_2)\p_s \psi_2 \right]
- 4 \varepsilon^2 \partial_s \psi_2 .
\end{align*}
We have in $D_R$:
\begin{align*}
\left|
\left(\frac{\nabla w_1}{w_1}+\frac{\nabla w_2}{w_2} \right) \nabla \psi_1
\right|
&\leq \frac{C}{R \rho_1^2}
 \|\psi_1\|_{1,*,0}
\\
\left|
\nabla (\theta_1+\theta_2) \nabla \psi_2
\right|
& \leq  C R^{\sigma-1}
\Bigl( \frac{1}{\rho_1^2} + \varepsilon^2
\Bigr)
\|\nabla \psi_2\|_{2,*,0}
\\
2\varepsilon^2 \left| 
\left(\frac{\p_s w_1}{w_1}+\frac{\p_sw_2}{w_2} \right) \partial_s \psi_1
\right| 
& \leq \frac{C}{R^2}
\Bigl(
\frac{1}{\rho_1^2}+\varepsilon^2
\Bigr) \|\psi_1\|_{1,*,0}
\\
2\e^2\left|\p_s(\theta_1+\theta_2)\p_s \psi_2 \right| 
& \leq C
\Bigl( \varepsilon^{1-\sigma} + R^{\sigma-1} \Bigr)
\Bigl( \varepsilon^2 + \frac{1}{\rho_1^2}\Bigr)
\|\psi_2\|_{2,*,0} 
\\
\e^2 |\p_s \psi_2|
&\leq 
C\Bigl(
\varepsilon^{1-\sigma}+R^{\sigma-1}
\Bigr)
\left(
\varepsilon^{2} + \frac{1}{\rho_1^{2}} 
\right) 
\|\psi_2\|_{2,*,0} .
\end{align*}
Using that $\|\psi\|_{*,0}=1$ we get that 
\begin{align*}
|p_1| \leq C \Bigl( \|h \|_{**,0} + R^{\sigma-1}+\varepsilon^{1-\sigma}  \Bigr)
\left(  \frac{1}{\rho_1^2}+\varepsilon^2\right)
\end{align*}
We use the comparison principle with the barrier
\begin{align*}
\mathcal B_1 = M \theta_1 (\pi-\theta_1)
\end{align*}
with $M = C( \|h \|_{**,0} + R^{\sigma-1}+\varepsilon^{1-\sigma} + \|\psi_1\|_{L^\infty(B_R(\dd))})$ and $C$ a large fixed constant.
We note that 
\begin{align*}
\partial_{ss}^2 \mathcal B_1 &=
- \tilde d \sin\theta_1
\frac{\tilde d \cos\theta_1}{\rho_1^2} M ( \pi - 2 \theta_1)
\\
& \qquad 
+
\Bigl(
1+\frac{\tilde d\cos\theta_1}{\rho_1}
\Bigr) 
\Bigl[
-\frac{\tilde d \sin\theta_1}{\rho_1} M ( \pi - 2 \theta_1)
-2 \Bigl(
1+\frac{\tilde d\cos\theta_1}{\rho_1}
\Bigr) M
\Bigr]
\end{align*}
\hide {\cb
\begin{align*}
\partial_s \mathcal B
&= 
\Bigl(
1+\frac{\tilde d\cos\theta_1}{\rho_1}
\Bigr) \partial_{\theta_1} \mathcal B
\\
&= 
\Bigl(
1+\frac{\tilde d\cos\theta_1}{\rho_1}
\Bigr) M( \pi-2\theta_1)
\\
\partial_{ss}^2 \mathcal B &=
\tilde d \sin\theta_1\partial_{\rho_1} 
\partial_s \mathcal B
+
\Bigl(
1+\frac{\tilde d\cos\theta_1}{\rho_1}
\Bigr) \partial_{\theta_1} \partial_s \mathcal B
\\
\partial_{ss}^2 \mathcal B &=
- \tilde d \sin\theta_1
\frac{\tilde d \cos\theta_1}{\rho_1^2} M ( \pi - 2 \theta_1)
\\
& \qquad 
+
\Bigl(
1+\frac{\tilde d\cos\theta_1}{\rho_1}
\Bigr) 
\Bigl[
-\frac{\tilde d \sin\theta_1}{\rho_1} M ( \pi - 2 \theta_1)
-2 \Bigl(
1+\frac{\tilde d\cos\theta_1}{\rho_1}
\Bigr) M
\Bigr]
\end{align*}
Then
\begin{align*}
\varepsilon^2 
\partial_{ss}^2 \mathcal B &
\leq 
C M \frac{1}{|\log \varepsilon| \rho_1^2}
+ M 
\Bigl(
\varepsilon+\frac{1}{\sqrt{|\log\varepsilon|}\rho_1}
\Bigr) 
\Bigl(
\frac{1}{\sqrt{|\log\varepsilon|}\rho_1} 
- 2 \varepsilon
\Bigr)
\\
& \leq 
C M \frac{1}{ \sqrt { |\log \varepsilon|  }}
\Bigl( \frac{1}{\rho_1^2} + \varepsilon^2
\Bigr)
-2  M \varepsilon^2
\end{align*}
and this implies that 
\begin{align*}
\Delta \mathcal B_1 + \varepsilon^2 \partial_{ss}^2 \mathcal B_1
\leq -\tilde c M \Bigl( \frac{1}{\rho_1^2}+\varepsilon^2\Bigr)
\end{align*}
}
From this we get that 
\begin{align*}
\Delta \mathcal B_1 + \varepsilon^2 \partial_{ss}^2 \mathcal B_1
\leq -\tilde c M \Bigl( \frac{1}{\rho_1^2}+\varepsilon^2\Bigr)
\end{align*}
for some fixed $\tilde c>0$.

Thanks to a comparison principle in $D_R$ (a slight variant of Lemma \ref{lem:comparison_principle_Neumann})
we get that
\begin{align}
\label{estPsi1a}
|\psi_1|
\leq 
C
(  \|h \|_{**,0} + R^{\sigma-1}+\varepsilon^{1-\sigma} + \|\psi_1\|_{L^\infty(B_R(\dd))}  )
\quad \text{in }D_R.
\end{align}
Elliptic estimates and a standard scaling give
\begin{align}
\label{estGradPsi1a}
\rho_1 |\nabla \psi_1|
\leq 
C
( \|h \|_{**,0} + R^{\sigma-1}+\varepsilon^{1-\sigma} + \|\psi_1\|_{L^\infty(B_R(\dd))})
\end{align}
for points in $D_R$ with $2<\rho_1<\frac{2}{\varepsilon}$.
To estimate the gradient for points in $D_R$ with $\rho_1>\frac{1}{\varepsilon}$ we use the scaling
\[
\tilde \psi( r,s) = \psi(\varepsilon^{-1}\tilde r ,s)
\]
and see that 
\begin{align}
\label{estGradPsi1b}
\varepsilon^{-1} |\partial_r \psi_1| + |\partial_s \psi_1| 
\leq
C( \|h \|_{**,0} + R^{\sigma-1}+\varepsilon^{1-\sigma} + \|\psi_1\|_{L^\infty(B_R(\dd))}).
\end{align}
in this region.

Combining \eqref{estPsi1a}, \eqref{estGradPsi1a} and \eqref{estGradPsi1b} we get that
\begin{align*}
\| \psi_1 \|_{1,*,0}
\leq  C
( \|h \|_{**,0} + R^{\sigma-1}+\varepsilon^{1-\sigma} + \|\psi_1\|_{L^\infty(B_R(\dd))}) .
\end{align*}
Then using \eqref{estPsi2B} we conclude that 
\begin{align*}
\| \psi \|_{*,0}
\leq  C
( \|h\|_{**,0} +R^{-\sigma} + \varepsilon^\sigma 
+R^{\sigma-1}+\varepsilon^{1-\sigma} 
+ \|\psi_1\|_{L^\infty(B_R(\tilde d))}
+ \|\psi_2\|_{L^\infty(B_R(\tilde d))}) .
\end{align*}
Using then \eqref{eq:inner_estimate-2} and $\|h\|_{**}=o(1)$, from the previous inequality  we get that $\|\psi\|_{*,0} <\frac{1}{2}$, 
if from the start $R$ is fixed large and we take $\varepsilon>0$ small. This is a contradiction and proves \eqref{est0}.

The full estimate \eqref{est0a} follows from \eqref{est0} and Schauder estimates.

\end{proof}
\begin{proof}[Proof of Proposition~ \ref{prop:linearfull}]
We first solve the problem in bounded domains. We consider the equation
\begin{align}
\label{eq:linearbounded}
\left\{
\begin{aligned}
& \L^\e(\psi)=h+c\sum_{j=1}^2\frac{\chi_j}{iW(z-\dd_j)} (-1)^jW_{x_1}(z-\dd_j)
\quad\text{in }B_M(0)
\\
& \psi = 0 \quad\text{on } \partial B_M(0)
\\
& \RE\int_{B(0,4)} \chi \overline{\phi_j}W_{x_1}=0, \text{ with }\phi_j(z)=iW(z)\psi(z+\dd_j),\quad j=1,2,\\
&\psi \text{ satisfies the symmetry }  \eqref{eq:eqsymmetriesofpsi}.
\end{aligned}
\right.
\end{align}
with $M>10\dd$. We set
\begin{equation*}\begin{split}
\mathcal{H}:=& \Bigl\{\phi=iV_d\psi \in H^1_0(B_M(0),\mathbb{C}); \;\RE \int_{B(0,4)} \chi\bar{\phi}_jW_{x_1}=0, \ j=1,2, \;\;\psi \text{ satisfies \eqref{eq:eqsymmetriesofpsi}}\Bigr\}.
\end{split}\end{equation*}
We equip $\mathcal{H}$ with the inner product
\begin{equation*}
[\phi,\varphi]:= \RE \int_{B_M(0)} \left( \nabla \phi \overline{\nabla\varphi} +\e^2\p_s \phi \overline{\p_s \varphi} \right).
\end{equation*}
With this inner product $\mathcal{H}$ is a Hilbert space. Indeed it is a closed subspace of $H^1_0(B_M(0),\mathbb{C})$ and $[\cdot,\cdot]$ is an  inner product on $H^1(B_M(0),\mathbb{C})$ thanks to the Poincar\'e inequality. In terms of $\phi$ the first equation of \eqref{eq:linearbounded} can be rewritten as
\begin{equation*}\begin{split}
\Delta \phi&+(1-|V_d|^2)\phi-2\RE(\overline{\phi}V_d)V_d +\e^2(\p^2_{ss}\phi-4i\p_s\phi-4\phi)+(\eta-1)\frac{E}{V_d}\phi\\
&=iV_dh+iV_dc\sum_{j=1}^2\tilde{\eta}(-1)^j\chi_j(z)\frac{W_{x_1}(z-\dd_j)}{iW(z-\dd_j)}.
\end{split}\end{equation*}
We can express this equation in its variational form. Namely, for all $\varphi \in \mathcal{H}$
\begin{equation*}\begin{split}
&-\RE \int_{B_M(0)}\left( \nabla \phi \overline{\nabla \varphi}+\e^2 \p_s\phi\overline{\p_s\varphi} \right)
+\e^2\RE\int_{B_M(0)}\left( 4i\phi\overline{\p_s\varphi}-4\phi \overline{\varphi} \right)\\
&-2\RE  \int_{B_M(0)} \RE(\overline{\phi}V_d)V_d \overline{\varphi}  +\RE \int_{B_M(0)} [(\eta-1)\frac{E}{V_d}+(1-|V_d|^2)]\phi \overline{\varphi} \\
&\qquad=\RE \int_{B_M(0)} iV_d\left(h- c\sum_{j=1}^2\chi_j(-1)^j\frac{W_{x_1}(z-\dd_j)}{iW(z-\dd_j)}\right)\overline{\varphi}.
\end{split}\end{equation*}
We now denote by $\langle k(x)\phi, \cdot \rangle$ the linear form on $\mathcal{H}$ defined  by
\begin{equation*}\begin{split}
\langle k(x)\phi,\varphi \rangle:=& \,\e^2\RE\int_{B_M(0)}\left( 4i\phi\overline{\p_s\varphi}-4\phi \overline{\varphi} \right)-2\RE  \int_{B_M(0)} \RE(\overline{\phi}V_d)V_d \overline{\varphi}  \\
&\,+\RE \int_{B_M(0)} [(\eta-1)\frac{E}{V_d}+(1-|V_d|^2)]\phi \overline{\varphi}.
\end{split}\end{equation*}
In the same way we denote by $\langle s,\cdot\rangle$ the linear form defined by
\begin{equation*}
\langle s,\varphi \rangle:= \RE \int_{B_M(0)} iV_d\left(h- c\sum_{j=1}^2\chi_j(-1)^j\frac{W_{x_1}(z-\dd_j)}{iW(z-\dd_j)}\right) \overline{\varphi}.
\end{equation*}
Thus, the equation can be rewritten as
\begin{equation*}
[\phi,\varphi]-\langle k(x)\phi,\varphi \rangle =\langle s,\varphi\rangle, \forall \varphi \in \mathcal{H}.
\end{equation*}
By using the Riesz representation theorem we can find a bounded linear operator $K$ on $\mathcal{H}$ and $S$, an element of $\mathcal{H}$ depending linearly on $s$, such that the equation has the operational form
\begin{equation}\label{eq:linearopform}
\phi-K(\phi)=S.
\end{equation}
Besides, thanks to the compact Sobolev injections $H^1_0(B_M(0),\mathbb{C}) \hookrightarrow L^2(B_M(0),\mathbb{C})$ we know that $K$ is compact. We can then apply Fredholm's alternative to deduce the existence of $\phi$ such that \eqref{eq:linearopform} holds if the homogeneous equation only has the trivial solution. To prove this last point we establish an a priori estimate on $c$. In order to do that we use the following equivalent form of the equation in the region $B(\dd,\dd)$, with the translated variable it becomes:
\begin{equation}
\nonumber
L_j^\e(\phi_j)=h_j+c\chi W_{x_1} \text{ in } B(0,\dd).
\end{equation}
where $L_j^\e$ is defined in \eqref{defL_j}, $\phi_j(\z)=iW(\z)\psi(z-\dd_j)$ and \(h_j(\z)=iW(\z)\psi(z-\dd8j)\) for $|\z|< \dd$.

We can test this equation against $\overline{W}_{x_1}$ to find
\begin{equation*}
c=-\frac{1}{c_*}\left[\RE \int_{B(0,\dd)}  h_j\overline{W}_{x_1}-\RE\int_{B(0,\dd)}L^\e_j(\phi_j)\overline{W}_{x_1}\right],
\end{equation*}
with $c_*:=\RE\int_{B(0,\dd)}\chi|W_{x_1}|^2=\RE \int_{B(0,R)} \chi |W_{x_1}|^2 \simeq C$ for some \(C>0\) of order \(1\) and $L^\e_j$ defined in \eqref{defL_j}.
Integrating by parts we obtain
\begin{equation*}
\RE \int_{B(0,\dd)} L^\e_j(\phi_j) \overline{W}_{x_1} =\RE \int_{B(0,\dd)} \overline{\phi}_j (L_j^\e-L^0)(W_{x_1}) +\RE \Bigl\{ \int_{\p B(0,\dd)} \frac{\p \phi_j}{\p \nu}\overline{W}_{x_1}-\phi_j \frac{\p}{\p \nu}\overline{W}_{x_1} \Bigr\}.
\end{equation*}
In the previous equality we used that \(L^0(W_{x_1})=0\). However, using the expansion of \(L^\e_j-L^o\) in \eqref{eq:defL_j} and the estimates \eqref{eq:estimatesonalpha_j}, we can see that 
\begin{equation}\label{eq:forfutureref}
\left |\RE \int_{B(0,\dd)} \overline{\phi}_j (L_j^\e-L^0)(W_{x_1}) \right| =O_\e(\e \sqrt{|\log \e|})\|\psi\|_*.
\end{equation}
By using the decay of \(\phi_j, \nabla \phi_j\) and \(W_{x_1}, \nabla W_{x_1}\) we can also check that
\begin{equation*}
\left|\RE \Bigl\{ \int_{\p B(0,\dd)} \frac{\p \phi_j}{\p \nu}\overline{W}_{x_1}-\phi_j \frac{\p}{\p \nu}\overline{W}_{x_1} \Bigr\} \right| =O_\e(\e \sqrt{|\log \e|})\|\psi\|_*.
\end{equation*}

\noindent Therefore we arrive at
\begin{equation*}
c=-\frac{1}{c_*}\RE \int_{B(0,\dd))} h_j\overline{W}_{x_1}+O_\e(\e\sqrt{|\log \e|})\frac{\|\psi\|_*}{c_*}.
\end{equation*}
To conclude the proof we note that we can apply Lemma \ref{FirstEstimate} to conclude that a solution of the homogeneous equation satisfies
\begin{equation*}
\|\psi\|_*\leq C\| c\sum_{j=1}^2\chi_j(z)(-1)^j\frac{W_{x_1}(z-\dd_j)}{iW(z-\dd_j)} \|_{**}\leq C \e \sqrt{|\log \e|} \|\psi\|_*,
\end{equation*}
and thus $\psi=0$.  Then for any $M>10\dd$ we obtain the existence of a solution of \eqref{eq:linearbounded} satisfying
\begin{equation*}
\|\psi_M\|_*\leq C \|h\|_{**},
\end{equation*}
with $C$ independent of $M$. Note that in the previous argument the norms $\|\cdot\|_*$, $\|\cdot\|_{**}$ are slightly adapted to deal with the fact that we work on bounded domains. We can extract a subsequence such that $\psi_M\rightharpoonup \psi$ in $H^1_{\text{loc}}(\R^2)$ with $\psi$ solving \eqref{eq:linear}. From Lemma \ref{FirstEstimate} we deduce $\|\psi\|_*\leq C\|h\|_{**}$.
\end{proof}

\subsection{Second a priori estimate and proof of Proposition \ref{prop:sharp2b}}
\label{sec:prop2}

\begin{lemma}
\label{lemma:aprioriSharp}
Let $\alpha \in (0,1)$, $\sigma \in (0,1)$.
Then  there exists a constant $C>0$  such that for all $\e$ sufficiently small and any solution $\psi$ of \eqref{eq:linearhomogeneous} with $\|\psi\|_*<\infty$  one has
\begin{equation}
\label{claimSharp}
|\psi|_\sharp \leq C(  |h |_{\sharp\sharp} + \varepsilon | \log\varepsilon|^\frac{1}{2} \|h\|_{**} )  .
\end{equation}
\end{lemma}
\begin{proof}
We work with the weaker seminorms
\begin{align*}
| \psi |_{\sharp,0} &= 
\sum_{j=1}^2 
|\log\varepsilon|^{-1}
\| V_d \psi \|_{L^\infty(\rho_j<3)}
+ |\RE(\psi)|_{\sharp,1}+ |\IM(\psi)|_{\sharp,2} ,
\end{align*}
where $|\ |_{\sharp,1}$, $|\ |_{\sharp,2}$ are defined in \eqref{normSharp1}, \eqref{normSharp2}
and
\begin{align*}
| h |_{\sharp \sharp,0}
&=
\sum_{j=1}^2 \| V_d h\|_{L^\infty(\l_j<4)}
+\sup_{2<\rho_1 < R_\varepsilon , 2<\rho_2 < R_\varepsilon}
\Bigl[
\frac{|h_1|}{\rho_1^{-1} + \rho_2^{-1}}
+\frac{ | h_2 | }{\rho_1^{-1+\sigma} + \rho_2^{-1+\sigma}}
\Bigr] .
\end{align*}
We claim that  there exists a constant $C>0$  such that for all $\e$ sufficiently small and any solution of \eqref{eq:linearhomogeneous} one has
\begin{equation}
\label{claimSharp0}
|\psi|_{\sharp,0} \leq C(  |h |_{\sharp\sharp,0} + \varepsilon | \log\varepsilon|^\frac{1}{2} \|h\|_{**} )  .
\end{equation}

We argue by contradiction
and assume that there exist \(\e_n \rightarrow 0\) and \( \psi^{(n)}, h^{(n)}\) solutions of  \eqref{eq:linearhomogeneous} such that
\begin{equation}
\label{a}
|\psi^{(n)} |_{\sharp,0}=1, \qquad (  |h^{(n)} |_{\sharp\sharp,0} + \varepsilon_n | \log\varepsilon_n|^\frac{1}{2} \|h^{(n)}\|_{**} ) \to 0,
\end{equation}
as $n\to\infty$.

We first work near the vortices and notice that, by symmetry, it is enough to consider the vortex at $+\dd$. We work  with the function $\phi_j^{(n)}(z)=i |\log\varepsilon|^{-1}  W(z)\psi^{(n)}(z+\dd_j)$.
Since $|\psi^{(n)} |_{\sharp,0}=1$ from Arzela-Ascoli's Theorem we can  extract a subsequence such that \( \tilde{\phi}^{(n)}_j \rightarrow \phi_0\) in \(C^0_{\text{loc}}(\R^2)\).  Passing to the limit in  \eqref{eq:linearhomogeneous}, we see that
\begin{equation*}
L^0(\phi_0)=0 \text{ in } \R^2,
\end{equation*}
with $L^0$ defined in \eqref{L0}. 
The function
$\phi_0$ inherits the symmetry
$\phi_0(\bar{z})=\overline{\phi_0(z)}$
%\begin{equation*}\begin{split}
%\phi^{(n)}_j(\bar{z})=&\ iW(x_1,-x_2)\psi_n(x_1+\dd,-x_2)=-i\overline{W}(x_1,x_2)\overline{\psi}_n(x_1+\dd,x_2) =\overline{\phi}^n_j(x_1+\dd,x_2) \nonumber
%\end{split}\end{equation*}
%and thus 
and satisfies $\phi_0\in L^\infty_{ \text{loc}}$.
Moreover, writing $\phi_0 = i W \psi_0$, $\psi_0 = \psi_0^1 + i \psi_0^2$,  we have
\begin{align*}
|\psi_0^1(z)|\leq |z| , \quad |\psi_0^2(z)|\leq 1,
\quad |z|>2 .
\end{align*}
Thanks to the above estimate and  \ Lemma \ref{lem:ellipticestimatesL0-b} we deduce
\begin{equation*}
\phi_0=c_1W_{x_1},
\end{equation*}
for some $c_1\in \R$.
On the other hand, we can pass to the limit in the orthogonality condition
$$\RE \int_{B(0,4)} \chi\bar{\phi}_j^{(n)} W_{x_1}=0,$$
and obtain that necessarily $c_1=0$. Hence $\phi_j^{(n)} \rightarrow 0$ in $C^0_{\text{loc}}(\R^2)$. 
%By elliptic estimates, $\phi_j^{(n)} \rightarrow 0$ in $W^{2,p}_{\text{loc}}(\R^2)$. 
We can also apply the same argument near $-\dd$ 
%and for any $R>0$ we obtain
%\begin{align*}
%%\label{eq:inner_estimate-b}
%\sum_{j=1}^2 \|\phi_j^{(n)}\|_{W^{2,p}(B_R(0))} \to 0, \quad
%\text{as }n\to\infty.
%\end{align*}
%Therefore
and get
\[
\frac{\psi^{(n)} }{|\log\varepsilon_n|} \to 0 
\]
uniformly on compact sets of  $\{\rho_1 \geq 1, \rho_2\geq 1\}$ as $\varepsilon_n\to0$.

In what follows in this proof we work  in the region 
\[
\tilde D_{R_0} = \{R_0<\rho_1<R_\varepsilon\}\cap \{ x_2 > 0\} ,
\]
where  $R_0>0$ is fixed large and $R_\varepsilon$ is given by \eqref{Reps}.

We use barriers to estimate $\psi_2^{(n)}(z)$ in $\tilde D_R$.
By the symmetries of $\psi_2^{(n)}$ we get the estimates for all $2<\rho_1<R_\varepsilon$.
Let us  write equation \eqref{eqPsi200} as 
\begin{align*}
\Delta\psi_2 
+ \left(\frac{\nabla w_1}{w_1}+\frac{\nabla w_2}{w_2} \right) \nabla \psi_2
-2 |V_d|^2 \psi_2 
+ \varepsilon^2 \partial_{ss}^2 \psi_2
+2 \varepsilon^2 \left(\frac{\p_s w_1}{w_1}+\frac{\p_s w_2}{w_2} \right) \partial_s \psi_2
= \tilde{p}_2
\end{align*}
where
\begin{align*}
\tilde{p}_2=h_2
- \nabla (\theta_1+\theta_2) \nabla \psi_1-2\e^2\p_s(\theta_1+\theta_2)\p_s \psi_1
+4 \varepsilon^2 \partial_s \psi_1.
\end{align*}
%
%\begin{align}
%\nonumber
%\Delta \psi_2^{(n)} 
%+\varepsilon_n^2 \p_{ss}\psi_2^{(n)}
%-2|V_d|^2 \psi_2^{(n)}   =  p_2^{(n)} ,
%\end{align}
%where
%\begin{align*}
%p_2 ^{(n)}
%&=
%h_2 ^{(n)}
%- \left(\frac{\nabla w_1}{w_1}
%+\frac{\nabla w_2}{w_2} \right) \nabla \psi_2^{(n)}
%- \nabla (\theta_1+\theta_2) \nabla \psi_1^{(n)}
%-2 \varepsilon_n^2 \frac{\partial_s V_d}{V_d} \partial_s \psi_2^{(n)}
%+4 \varepsilon_n^2 \partial_s  \psi_1 ^{(n)}.
%\end{align*}
We observe that in $\tilde D_{R_0}$ it holds
\begin{align*}
%\left|
%\left(\frac{\nabla w_1}{w_1}+\frac{\nabla w_2}{w_2} \right) 
%\nabla \psi_2^{(n)}\right|
%&\leq 
%\frac{C}{R^3}
%\\
\Bigl|
\nabla (\theta_1+\theta_2) \nabla \psi_1 ^{(n)}
\Bigr| 
&
\leq   \frac{C }{\rho_1}
\log\Bigl(\frac{2R_\varepsilon}{\rho_1}\Bigr)
|\psi_1^{(n)}|_{\sharp,1}
\\
\varepsilon_n^2
\left| 
\p_s(\theta_1+\theta_2)\p_s \psi_1
\right| 
& \leq \frac{C }{\rho_1 |\log\varepsilon_n|}
\log\Bigl(\frac{2R_\varepsilon}{\rho_1}\Bigr)
|\psi_1^{(n)}|_{\sharp,1} 
\\
\varepsilon_n^2 | \partial_s \psi_1^{(n)}| 
&\leq  
\frac{C}{\rho_1 |\log\varepsilon_n|}
\log\Bigl(\frac{2R_\varepsilon}{\rho_1}\Bigr)
|\psi_1^{(n)}|_{\sharp,1}.
\end{align*}
\hide {\cb
Indeed, 
\begin{align*}
\Bigl|
\nabla (\theta_1+\theta_2) \nabla \psi_1 
\Bigr| 
& \leq 
\frac{C}{\rho_1} |\nabla \psi_1|
\\
& \leq  \frac{C }{\rho_1}
\log\Bigl(\frac{2R_\varepsilon}{\rho_1}\Bigr)
|\psi_1^{(n)}|_{\sharp,1}
\end{align*}
}
\hide {\cb
Indeed
\begin{align*}
\varepsilon^2 |\partial_s \psi_1 |
& \leq 
\varepsilon^2 
\Bigl[
\tilde d  |\partial_{\rho_1} \psi_1|
+
\Bigl(
1+\frac{\tilde d}{\rho_1} 
\Bigr)|\partial_{\theta_1} \psi_1|
\Bigr]
\\
& \leq 
\varepsilon
\Bigl[
|\log\varepsilon|^{-\frac{1}{2}}
|\nabla \psi_1|
+
(
\varepsilon \rho_1+ |\log\varepsilon|^{-\frac{1}{2}}
)|\nabla \psi_1|
\Bigr]
\\
& \leq 
\varepsilon
|\log\varepsilon|^{-\frac{1}{2}}
|\nabla \psi_1|
\\
& \leq 
\varepsilon
|\log\varepsilon|^{-\frac{1}{2}}
\log\Bigl(\frac{2R_\varepsilon}{\rho_1}\Bigr)
|\psi_1^{(n)}|_{\sharp,1}
\\
& \leq 
\frac{C}{\rho_1 |\log\varepsilon|}
\log\Bigl(\frac{2R_\varepsilon}{\rho_1}\Bigr)
|\psi_1^{(n)}|_{\sharp,1}
\end{align*}
}
Using the a priori estimate of Lemma~\ref{FirstEstimate} we find that
\begin{align}
\label{boundPsi2}
\|\psi^{(n)} \|_{*} \leq C \|h^{(n)} \|_{**} = o(1)\varepsilon_n^{-1}
|\log \varepsilon_n|^{-\frac{1}{2}} .
\end{align}
Thus writing
\[
\psi^{(n)} = \psi^{(n)}_1+i\psi^{(n)}_2
\]
we have
\begin{align}
\nonumber
|\psi_1^{(n)}(z)|&\leq o(1)\varepsilon_n^{-1}
|\log \varepsilon_n|^{-\frac{1}{2}}  ,
\\
\nonumber
|\psi_2^{(n)}(z)| + |\nabla \psi_2^{(n)}(z)|&\leq 
o(1)
\varepsilon_n^{-1}
|\log \varepsilon_n|^{-\frac{1}{2}} 
\Bigl( \frac{1}{\rho_1^{2-\sigma}} + \frac{1}{\rho_2^{2-\sigma}}
\Bigr)
\end{align}
for $2<|z|<\frac{1}{\varepsilon_n}$ 
with $o(1)\to0$
as $n\to\infty$.
We note that for $|z-\dd_j|=R_\varepsilon$, 
\begin{align*}
|\psi_2^{(n)}(z)|\leq \frac{\|\psi^{(n)}\|_*}{R_\varepsilon^{2-\sigma}}
= \frac{o(1) R_\varepsilon}{R_\varepsilon^{2-\sigma}}
= o(1) ,
\end{align*}
as $n\to \infty$ by \eqref{boundPsi2}.
We use as a barrier the function
\begin{align*}
\tilde{\mathcal{B}}_2 = 
\frac{C}{\rho_1^{1+\sigma}} ( 
|h^{(n)} |_{\sharp\sharp,0} 
+\|\psi_2^{(n)}\|_{L^\infty(\rho_1=R_{\varepsilon_n})}  )
+ \frac{C}{\rho_1}\log\Bigl(\frac{2R_\varepsilon}{\rho_1}\Bigr)
\left(  |\psi_1^{(n)}|_{\sharp,1} 
+
\frac{\|\psi_2^{(n)}\|_{L^\infty(\rho_1=R_0)} }{ |\log\varepsilon_n |}
\right)
\end{align*}
where \(C>0\) is a large fixed constant.
We note that 
\begin{align}
\label{boundB2}
\tilde{\mathcal B}_2
\leq \frac{b_n}{\rho_1^{1-\sigma}}
+ 
\frac{1}{\rho_1}\log\Bigl(\frac{2R_\varepsilon}{\rho_1}\Bigr)
\left( C  |\psi_1^{(n)}|_{\sharp,1}   + b_n
\right) 
\end{align}
in $\tilde D_{R_0}$ where $b_n\to 0$ as $n\to\infty$.
By the maximum principle  and elliptic  estimates we get
\begin{align}
\label{estPsi2}
|\psi_2^{(n)}|
+|\nabla\psi_2^{(n)}|
\leq  \tilde{\mathcal B}_2
\end{align}
in $\tilde D_{R_0}$.

Next we use barriers to estimate $\psi_1^{(n)}$ in $\tilde D_{R_0}$.
By the symmetries of $\psi_1^{(n)}$ we get the estimates for all $2<\rho_1<R_\e$.
Let us  write Equation \eqref{eqPsi100} as 
\begin{align*}
\Delta \psi_1 
+\left(\frac{\nabla w_1}{w_1}+\frac{\nabla w_2}{w_2} \right) \nabla \psi_1
+ \varepsilon^2 \partial_{ss}^2 \psi_1
+ 2\varepsilon^2 \left(\frac{\p_sw_1}{w_1}+\frac{\p_sw_2}{w_2} \right) \partial_s \psi_1
= p_1
\end{align*}
where
\begin{align*}
p_1 =h_1
+ \nabla (\theta_1+\theta_2) \nabla \psi_2+2 \e^2\p_s(\theta_1+\theta_2)\p_s \psi_2
- 4 \varepsilon^2 \partial_s \psi_2  
\end{align*}
%\begin{align*}
%\Delta \psi_1^{(n)} + \varepsilon_n^2 \partial_{ss}\psi_1^{(n)} = p_1^{(n)},
%\end{align*}
%where
%\begin{align*}
%p_1^{(n)} &=
%h_1^{(n)} 
%-\left(\frac{\nabla w_1}{w_1}+\frac{\nabla w_2}{w_2} \right) \nabla \psi_1^{(n)}
%+ \nabla (\theta_1+\theta_2) \nabla \psi_2^{(n)}
%- 2\varepsilon_n^2 \frac{\partial_s V_d}{V_d} \partial_s \psi_1^{(n)}
%- 4 \varepsilon_n^2 \partial_s \psi_2 ^{(n)}.
%\end{align*}

We find that  in $\tilde D_R$ the following estimates hold:
\begin{align*}
\left|
\nabla (\theta_1+\theta_2) \nabla \psi_2^{(n)}
\right|
& \leq 
\frac{C}{\rho_1} \tilde{\mathcal B}_2\\
\e_n^2\left|\p_s(\theta_1+\theta_2)\p_s \psi_2^{(n)} \right| & \leq \frac{C}{\rho_1 \log\varepsilon_n|} \tilde{\mathcal B}_2
\\
\e_n^2 |\p_s \psi_2^{(n)}|
&\leq 
\frac{C}{\rho_1 |\log\varepsilon_n|}
\tilde{\mathcal B}_2
\end{align*}
\hide {\cb
Indeed, 
\begin{align*}
\Bigl|
\nabla (\theta_1+\theta_2) \nabla \psi_2
\Bigr| 
& \leq 
\frac{C}{\rho_1} |\nabla \psi_2|
\\
& \leq  \frac{C}{\rho_1} \tilde{\mathcal B}_2
\end{align*}
}
\hide {\cb
Indeed
\begin{align*}
\varepsilon^2 |\partial_s \psi_2 |
& \leq 
\varepsilon^2 
\Bigl[
\tilde d  |\partial_{\rho_1} \psi_2|
+
\Bigl(
1+\frac{\tilde d}{\rho_1} 
\Bigr)|\partial_{\theta_1} \psi_2|
\Bigr]
\\
& \leq 
\varepsilon
\Bigl[
|\log\varepsilon|^{-\frac{1}{2}}  |\nabla\psi_2|
+
\Bigl(
\varepsilon \rho_1 + |\log\varepsilon|^{-\frac{1}{2}}
\Bigr)|\nabla \psi_2|
\Bigr]
\\
& \leq
\varepsilon |\log\varepsilon|^{-\frac{1}{2}}
|\nabla \psi_2|
\\
& \leq
\frac{C}{\rho_1 |\log\varepsilon|}
|\nabla \psi_2|
\end{align*}
}

Hence using  \eqref{boundB2} and \eqref{estPsi2} we get
\begin{align*}
|p_1|\leq  
\frac{b_n}{\rho_1}
+ 
\frac{1}{\rho_1^2}\log\Bigl(\frac{2R_\varepsilon}{\rho_1}\Bigr)
\left( C  |\psi_1^{(n)}|_{\sharp,1}   + b_n
\right) 
\end{align*}
for a new sequence $b_n\to0$.

Using Lemma ~\ref{lem:estimates-3b} for part of the right hand side and the supersolution $\log\Bigl(\frac{2R_\varepsilon}{\rho_1}\Bigr)
\left( C  |\psi_1^{(n)}|_{\sharp,1}   + b_n
\right) $
we conclude that 
\begin{align}
\nonumber
|\psi_1^{(n)}(z)| &\leq 
C b_n  \rho_1 \log\Bigl(\frac{2R_\varepsilon}{\rho_1}\Bigr)
+
\log\Bigl(\frac{2R_\varepsilon}{\rho_1}\Bigr)
\left( C  |\psi_1^{(n)}|_{\sharp,1}   + b_n
\right) 
\\
%\label{estPsi1b}
\nonumber
& \leq 
C \rho_1 \log\Bigl(\frac{2R_\varepsilon}{\rho_1}\Bigr)
\Bigl( b_n + \frac{  |\psi_1^{(n)}|_{\sharp,1} }{R_0} \Bigr).
\end{align}
This and standard elliptic estimates yield
\begin{align*}
|\psi_1^{(n)}|_{1,\sharp}
\leq
C \Bigl( b_n + \frac{  |\psi_1^{(n)}|_{\sharp,1} }{R_0} \Bigr).
\end{align*}
Choosing $R_0>0$ large and fixed we get 
\[
|\psi_1^{(n)}|_{1,\sharp}\to0\quad\text{as }n\to\infty.
\]
Using this and \eqref{boundB2}, \eqref{estPsi2} we obtain
\[
|\psi_2^{(n)}|_{2,\sharp}\to0\quad\text{as }n\to\infty.
\]
This contradicts the assumption \eqref{a} and we obtain
\eqref{claimSharp0}. With this inequality and standard Schauder estimates with deduce \eqref{claimSharp}. 
\end{proof}

As an intermediate step to obtain Proposition~\ref{prop:sharp2b} we consider the symmetry properties of the solution constructed in Proposition~\ref{prop:linearfull}, when the right hand side has symmetries. More precisely, 
let us consider the local symmetry condition
\begin{align}
\label{symmetryhe}
h(\mathcal{R}_j z ) = - \overline{h(z)}  ,\quad 
|z-\tilde d_j| < 2 R_\varepsilon , \quad j=1,2.
\end{align}

\begin{lemma}
\label{lemma:sharp-even}
Suppose that $h$ satisfies the symmetries \eqref{eq:eqsymmetriesofpsi} and \eqref{symmetryhe}.
We assume that
\[
\| h \|_{**} <\infty.
\]
Then there exist  $\psi^s $, $ \psi^*$ such that the solution 
$\psi$ to \eqref{eq:linear} with $\|\psi\|_*<\infty$
can be written as $\psi = \psi^s + \psi^*$ with the estimates
\begin{align*}
\| \psi^s \|_* + \| \psi^* \|_* &\leq C  \| h \|_{**} 
\\
| \psi^* |_\sharp &\leq C  \varepsilon |\log\varepsilon|^{\frac{1}{2}}   \| h \|_{**} .
\end{align*}
Moreover $(\psi^s, \psi^*)$ define  linear operators of $h$,  $\psi^s$ has its
support in $B_{R_\varepsilon}(\tilde d_1)  \cup B_{R_\varepsilon}(\tilde d_2)$ 
and satisfies
\begin{align}
\label{symmetryPsyE}
\psi^s(\mathcal{R}_j z ) &= - \overline{\psi^s(z)} ,\quad 
|z-\tilde d_j| <  R_\varepsilon .
\end{align}
\end{lemma}
\begin{proof}%[Proof of Lemma~\ref{lemma:sharp-even}]
To  construct the function $\psi^s$ we split the operator  $\mathcal L^\varepsilon$ (c.f. \eqref{calLeps})
into a part $ \mathcal L^\varepsilon_s $ preserving the symmetry \eqref{symmetryPsyE} and a remainder $\mathcal L^\varepsilon_r$.
This splitting depends on which vortex $\tilde d_j$ we are considering for the symmetry \eqref{symmetryPsyE} and thus we write 
$\mathcal{L}_{s,j}^\varepsilon$, 
$\mathcal{L}_{r,j}^\varepsilon$, $j=1,2$.
It is sufficient to consider the vortex at $\tilde d_1$.
We set 
\begin{align*}
\mathcal L ^\varepsilon _{s,1}(\psi)
&=
\Delta \psi 
+ 2\frac{\nabla W^a \nabla \psi}{W^a}
-2i |W^a|^2 \IM(\psi)
\\
& \quad  
+\varepsilon^2 \Biggl[ \tilde d^2 \partial_{\rho_1 \rho_1}^2 \psi \sin(\theta_1)^2
+\frac{\tilde d^2}{\rho_1}
\partial_{\rho_1 \theta_1}^2 \psi   \sin(\theta_1)   \cos(\theta_1) 
+\partial_{\theta_1 \theta_1}^2 \psi \Bigl(1 + \frac{\tilde d^2}{\rho_1^2} \cos^2(\theta_1) \Bigr)
\\
& 
\qquad \qquad
+ \partial_{\rho_1} \psi  \frac{\tilde d^2}{\rho_1} \cos^2(\theta_1) 
-2 \partial_{\theta_1} \psi 
\frac{\tilde d^2}{\rho_1^2} \sin(\theta)\cos(\theta_1) 
\Biggr]
\end{align*}
\begin{align*}
\mathcal L ^\varepsilon _{r,1}(\psi)
&=2\frac{\nabla W^b \nabla \psi}{W^b}
-2 i ( |V_d|^2 -  |W^a|^2 )  \IM(\psi) 
\\
& 
+ \varepsilon^2
\Biggl[
2 \tilde d
\partial_{\rho_1 \theta_1}^2 \psi  
 \sin(\theta_1) 
+2 \partial_{ \theta_1\theta_1}^2 \psi  \frac{\tilde d}{\rho_1} \cos(\theta_1) 
+ \partial_{\rho_1} \psi  \tilde d \cos(\theta_1)
-  \partial_{\theta_1} \psi
\frac{\tilde d}{\rho_1} \sin(\theta_1) 
\Biggr]
\\
& 
+ 
\varepsilon^2 \Bigl( \frac{2\partial_s V_d }{V_d}
- 4i  \Bigr) 
\Bigl[  \partial_{\rho_1} \psi \tilde d \sin(\theta_1)
+  \Bigl(1 + \frac{\tilde d}{\rho_1} \cos(\theta_1) \partial_{\theta_1}\psi  \Bigr) \Bigr] .
\end{align*}
%and note that indeed 
%\[
%\mathcal L^\varepsilon = \mathcal L^\varepsilon_{s,1} + \mathcal L^\varepsilon_{r,1} .
%\]

We use the same cut-off functions defined in \eqref{def:etajR} and solve
\begin{align}
\nonumber
\left\{
\begin{aligned}
& \L_s^\e(\psi^{2,1})=h  \eta_{1,2R_\varepsilon}
\quad\text{in }\R^2
\\
& \RE\int_{B(0,4)} \chi \overline{\phi^{2,1}}W_{x_1}=0, \text{ with }\phi^{2,1}(z)=iW(z)\psi^{2,1}(z+\dd_1)\\
&\psi^{2,1} \text{ satisfies } \psi^{2,1}(\bar z)  = - \psi^{2,1}(z).
\end{aligned}
\right.
\end{align}
This is obtained as variant of Proposition~\ref{prop:linearfull} with the same proof.
Note that there is no need to project the right hand side, since it is automatically orthogonal to the kernel by symmetry, and note also that the orthogonality condition for the solution holds also by symmetry.
Recall that $h$ satisfies \eqref{symmetryhe} and we get a solution $\psi_{2,1}$ satisfying \eqref{symmetryPsyE}
with the estimate
\[
\| \psi^{2,1} \|_* 
\leq C \| h \|_{**} .
\]
In a similar way we construct $\psi^{2,2}$ centered at the vortex $\tilde d_2$ and define
\begin{align}
\label{defPsiS}
\psi^s = \eta_{1,\frac{1}{2}R_\varepsilon} \psi^{2,1} ,
+ \eta_{2,\frac{1}{2}R_\varepsilon} \psi^{2,2} .
\end{align}
Note that we have the estimate
\[
\| \psi^s \|_* 
\leq C \| h \|_{**} .
\]

Let
\begin{align*}
\tilde h &:= h - \mathcal L^\varepsilon_{s,1}(\eta_{1,\frac{1}{2}R_\varepsilon} \psi^{2,1} ) -  \mathcal L^\varepsilon_{r,1}(\eta_{1,\frac{1}{2}R_\varepsilon} \psi^{2,1} ) 
%\\
%& \qquad 
- \mathcal L^\varepsilon_{s,2}(\eta_{2,\frac{1}{2}R_\varepsilon} \psi^{2,2} ) -  \mathcal L^\varepsilon_{r,2}(\eta_{2,\frac{1}{2}R_\varepsilon} \psi^{2,2} )  .
\end{align*}
Some lengthy but direct calculations show that
\begin{align}
%\label{estTildeH1}
\nonumber
\| \tilde h \|_{**}
& \leq C  ( \| h \|_{**} + \| h^* \|_{**} ) 
\\
%\label{estTildeH2}
\nonumber
|\tilde h |_{\sharp\sharp} 
& \leq 
C \varepsilon |\log \varepsilon|^{\frac{1}{2}}   ( \| h \|_{**} + \| h^* \|_{**} ) .
\end{align}
Then we solve, using Proposition~\ref{prop:linearfull},
\begin{align}
\nonumber
\left\{
\begin{aligned}
& \L^\e(\tilde \psi)=\tilde h +\tilde c\sum_{j=1}^2\frac{\chi_j}{iW(z-\dd_j)} (-1)^jW_{x_1}(z-\dd_j)
\quad\text{in }\R^2
\\
& \RE\int_{B(0,4)} \chi \overline{\tilde \phi_{j}}W_{x_1}=0, \text{ with } \tilde \phi_{j}(z)=iW(z)\tilde \psi(z+\dd_j)\\
&\tilde \psi \text{ satisfies the symmetry }  \eqref{eq:eqsymmetriesofpsi},
\end{aligned}
\right.
\end{align}
and obtain, using also Lemma~\ref{lemma:aprioriSharp},
\begin{align*}
\| \tilde \psi \|_* &\leq C   \| \tilde h \|_{**}
\\
| \tilde \psi |_\sharp &\leq C  
( | \tilde h |_{\sharp\sharp} + \varepsilon |\log\varepsilon|^{\frac{1}{2}} \|\tilde h \|_{**} ).
\end{align*}
Finally we set 
\begin{align}
\label{defPsiStar}
\psi^* &=  \tilde \psi.
\end{align}
The functions $\psi^s$, $\psi^*$ defined in \eqref{defPsiS}, \eqref{defPsiStar} satisfy the stated properties.
\end{proof}

\begin{proof}[Proof of Proposition~\ref{prop:sharp2b}]
Let us define
\[
\tilde h = h-h^o
\]
so that $h=\tilde h + h_\alpha^o + h_\beta^o $.
Let $\tilde \psi$, $\tilde \psi_\alpha$, $\tilde \psi_\beta$ be the solution with finite $\| \ \|_{*}$-norm of \eqref{eq:linear} with right hand sides $\tilde h$, $h_\alpha^o$, $h_\beta^o$ given by Proposition~\ref{prop:linearfull}.
Then  $\psi = \tilde \psi + \tilde \psi_\alpha + \tilde \psi_\beta$ and we have the estimates
\begin{align*}
\|\tilde \psi\|_*  & \lesssim \|\tilde h\|_*\\
\|\tilde \psi_j\|_*  & \lesssim \| h^o_j\|_* , \quad j=\alpha,\beta.
\end{align*}
We have  $\psi^o = \tilde \psi^o + \tilde\psi_\alpha^o + \tilde\psi_\beta^o$.
We define 
\begin{align*}
\psi_\alpha^o = \tilde \psi^o + \tilde\psi_\alpha^o , \quad
\psi_\beta^o =  \tilde\psi_\beta^o .
\end{align*}
Note that by Lemma~\ref{lemma:aprioriSharp}
\begin{align*}
|  \tilde\psi_\alpha^o  |_\sharp 
&\lesssim
|  \tilde\psi_\alpha  |_\sharp  \lesssim
|h_\alpha^o |_{\sharp\sharp} + \varepsilon | \log\varepsilon|^\frac{1}{2} \|h_\alpha^o\|_{**}  .
\end{align*}

According to Lemma~\ref{lemma:sharp-even} we can write $\tilde \psi = \psi^s + \psi^*$ with $\psi^s$, $\psi^*$ satisfying the properties stated in that lemma, from which we get
\begin{align*}
| \tilde \psi^o |_\sharp
& =
| ( \psi^* ) ^o|_\sharp
 \lesssim
|  \psi^* |_\sharp
\lesssim 
\varepsilon |\log\varepsilon|^{\frac{1}{2}}
\| \tilde h \|_{**}.
\end{align*}
Therefore
\begin{align*}
|\psi_\alpha^o|_\sharp 
\lesssim
|h_\alpha^o |_{\sharp\sharp} 
+ \varepsilon | \log\varepsilon|^\frac{1}{2} ( \|h_\alpha^o\|_{**}  
+  
\| \tilde h \|_{**} )
\end{align*}
and this proves \eqref{est:prop5.3-1}.

On the other hand
\begin{align*}
\| \psi_\alpha^o  \|_* 
& \leq
\| \tilde \psi^o  \|_* 
+\| \tilde \psi_\alpha^o  \|_* 
\lesssim
\| \tilde \psi  \|_* 
+\| \tilde \psi_\alpha  \|_* 
\lesssim
\| \tilde h  \|_* 
+\| h_\alpha^o   \|_* 
\\
\| \psi_\beta^o  \|_* 
& =
\| \tilde \psi_\beta^o  \|_* 
\lesssim
\| \tilde \psi_\beta \|_* 
\lesssim
\| h_\beta^o   \|_* 
\end{align*}
and from here \eqref{est:prop5.3-2} follows.
\end{proof}

\section{A projected nonlinear problem}%\label{VI}
We consider now the nonlinear projected problem
\begin{equation}\label{eq:nonlinear}
\left\{
\begin{split}
&\L^\e(\psi)=R+\N(\psi)+c\sum_{j=1}^2\frac{\chi_j(z)}{iW(z-\dd_j)} (-1)^jW_{x_1}(z-\dd_j)\;\; \text{ in } \R^2, \\
&\RE \int_{\R^2} \chi \overline{\phi_j}W_{x_1}=0, \ \text{ with }\phi_j(z)=iW(z)\psi(z+\dd_j), \ j=1,2, \\
&\psi \text{ satisfies } \eqref{eq:eqsymmetriesofpsi}.
\end{split}
\right.
\end{equation}
Using the operator $T_\e$ introduced in Proposition \ref{prop:linearfull} we can rewrite this equation in the form of a fixed point problem as
\begin{equation}\nonumber%\label{eq:fixedpoint}
\psi=T_\e\left(R+\N(\psi) \right)=:G_\e(\psi).
\end{equation}
\begin{proposition}\label{prop:nonlinear}
There exists a constant $C>0$ depending only on $0<\alpha,\sigma<1$, such that for all $\e$ sufficiently small there exists a unique solution $\psi_\e$ of \eqref{eq:nonlinear}, that satisfies
\begin{equation*}
\|\psi_\e\|_*\leq \frac{C}{|\log \e|}.
\end{equation*}
Furthermore $\psi_\e$ is a continuous function of the parameter $\d:=\sqrt{|\log \e|}d$
\begin{eqnarray}\label{estimatesPhi}
|\psi_\e^o|_{\sharp} \leq  C \e \sqrt{|\log \e|},
\end{eqnarray}
where $\psi_\e^o$ is defined according to \eqref{def-psio}.
\end{proposition}
\begin{proof}
We let
\begin{equation*}\begin{split}
\mathcal{F}:=\Bigl\{& \psi:\; \psi \text{ satisfies } \eqref{eq:eqsymmetriesofpsi}, \ \RE \int_{\R^2} \chi_j\overline{\phi_j}W_{x_1}=0, \;\;j=1,2, \;\;\|\psi\|_*\leq\frac{C}{|\log \e|}, \\
& \quad |\psi^o|_{\sharp}\leq C \e \sqrt{|\log \e|} \Bigr\}.
\end{split}\end{equation*}
Endowed with the norm $\| \cdot \|_*$, $\mathcal{F}$ is a Banach space as a closed subset of the Banach space $\{\psi:\;\;\|\psi\|_*<+\infty\}$. We will show that, for $\e$ small enough, $G_\e$ maps $\mathcal{F}$ into itself. Indeed, we need to check that if $\|\psi\|_{*}\leq \frac{C}{|\log \e|}$ then $\|T_\e(E+\N(\psi))\|_*\leq C/|\log \e|$.

Note first that, from Proposition \ref{prop:sizeoferror},
\begin{equation*}
\|R\|_{**}\leq \frac{C}{|\log \e|}.
\end{equation*}
Let us now estimate the size of the nonlinear term. For $\rho_1>3$ and $\rho_2>3$ the nonlinear terms
are
\begin{align*}
i (\nabla \psi)^2  + i |V_d|^2(e^{-2\psi_2}-1+2\psi_2)+i\varepsilon^2 (\partial_s \psi)^2.
\end{align*}

Let us work in the right half plane (so $\rho_1\leq\rho_2$). 
We start with $(\nabla \psi)^2$. 
For $3<\rho_1<\frac{2}{\varepsilon}$  we have
\begin{align*}
|(\nabla \psi)^2| & \leq |\nabla \psi|^2
\leq \frac{\|\psi\|_*^2}{\rho_1^2} .
\end{align*}
For $r>\frac{1}{\varepsilon}$ we use
\begin{align*}
(\nabla \psi)^2 &= (\partial_r \psi)^2 + \frac{1}{r^2} (\partial_s \psi)^2
\end{align*}
and estimate
\begin{align*}
|  (\partial_r \psi)^2 |
= (\partial_r \psi_1)^2 + (\partial_r \psi_2)^2  
\leq \varepsilon^2 \| \psi \|_*^2
\end{align*}
and
\begin{align*}
\frac{1}{r^2} |(\partial_s \psi)^2|
& \leq \frac{1}{r^2}
\Bigl( (\partial_s \psi_1)^2 + (\partial_s \psi_2)^2  \Bigr)
\\
& \leq \frac{1}{r^2} \| \psi \|_*^2 .
\end{align*}
It follows that 
\begin{align*}
\| i (\nabla \psi)^2 \|_{**}  \leq  C \| \psi \|_*^2.
\end{align*}

Next we consider  $i |V_d|^2(e^{-2\psi_2}-1+2\psi_2)$. We note that the real part of this function is zero.
Again we work in the right half plane.
We have, for $\rho_1>3$,
\begin{align*}
| \, |V_d|^2(e^{-2\psi_2}-1+2\psi_2) \, |
& \leq
C |\psi_2|^2 \\
& \leq C(\rho_1^{-2+\sigma}+\varepsilon^{2-\sigma})^2 \| \psi_2 \|_{2,*}^2 ,
\end{align*}
and hence
\begin{align*}
\| i |V_d|^2(e^{-2\psi_2}-1+2\psi_2) \|_{**}
\leq C  \| \psi \|_*^2 ,
\end{align*}

Finally we consider $i \varepsilon^2 (\partial_s \psi)^2$.
We have for $3<\rho_1<\frac{2}{\varepsilon}$
\begin{align*}
| \varepsilon^2 (\partial_s \psi)^2 |
&\leq \varepsilon^2 \tilde d^2 |\partial_{\rho_1} \psi|^2
+ \Bigl( \varepsilon + \frac{\varepsilon\tilde d}{\rho_1}\Bigr)^2
|\partial_{\theta_1} \psi|^2
\\
& \leq 
C (\varepsilon^2 + \rho_1^{-2} ) \|\psi\|_*^2 .
\end{align*}
For $r>\frac{1}{\varepsilon}$
\begin{align*}
| \varepsilon^2 (\partial_s \psi)^2 |
& \leq \varepsilon^2 \|\psi\|_*^2 .
\end{align*}
It follows that 
\begin{align*}
\| i \varepsilon^2 (\partial_s \psi)^2 \|_{**}
\leq  \|\psi\|_*^2 .
\end{align*}

In $\{\l_1 \leq 3 \} \cup \{\l_2\leq 3 \}$ it can be checked that
\begin{equation*}
|iV_d\mathcal{N}(\psi)|\leq C( |D^2 \gamma| +|D \gamma| +|\gamma|
+|\gamma+\phi||\phi|+|\gamma+\phi|^2(1+|\gamma+\phi|+|\gamma|)+|E_d||\phi|+|\nabla \phi|^2)
\end{equation*}
with $\gamma=(1-\eta)V_d \left(e^{i\psi}-1-i\psi \right)$. Thus we obtain that for any $j=1,2$
\begin{equation*}
\|iV_d\N(\psi)\|_{C^\alpha(\{\l_j <3 \}  }\leq C\|\psi\|_*^2+|E||\phi| \leq \frac{C}{|\log \e|^2}.
\end{equation*}
Thus for an appropriate constant $C$ we have that $G_\e:\psi \mapsto T_\e(E+\mathcal{N}(\psi))$ maps the ball $\{ \psi; \|\psi\|_*\leq\frac{C}{|\log\e|}\}$ into itself.

Let us now see  the precise estimates on the ``odd parts" and ``even parts". From Proposition \ref{errorProp2} we know that $R^o$ defined as in \eqref{def-ho},  can be decomposed into $R^o=R_\alpha^o+R^o_\beta$ with
\begin{equation*}
|R^o_\alpha|_{\sharp \sharp}\leq \frac{C\e}{\sqrt{|\log \e|}} \quad \|R^o_\beta\|_{**} \leq C \e \sqrt{|\log \e|}.
\end{equation*}
It remains to prove that 
\begin{equation}\label{eq:estimate_Npsi_sharp}
| \mathcal{N}(\psi)^{o}|_{\sharp\sharp}\leq C\left( (|\psi^o|_{\sharp}+ \e |\log \e|^{1/2})\|\psi^e\|_{*} +|\psi^o|_{\sharp}^2 \right).
\end{equation}
In order to do that we recall that in the decomposition of a function \(f\) in odd and even modes we have that, near \(+\tilde{d}\) the function
\(f^e\) is exactly \(\pi\)-periodic in \(\theta_1\) whereas \(f^o\) is exactly \(2\pi\)-periodic in \(\theta_1\). An analogous statement is true  near \(-\tilde{d}\). Now we can express the product of two functions as
\begin{equation*}
f g =(f^e+f^o)(g^e+g^o)=f^eg^e+f^eg^o+g^ef^o+g^of^o.
\end{equation*}
We see that \(f^eg^e\) is exactly \(\pi\)-periodic and hence \( (fg)^o=[f^eg^o+g^ef^o+f^og^o ]^o\). Thus we arrive at 
\begin{equation}\label{eq:key_decomposition}
\left| (fg)^o \right| \leq \left( |f^o||g^e|+|f^e||g^o|+|f^o||g^o|\right).
\end{equation}
To estimate \(\mathcal{N}(\psi)\) we use a change of variables \((r,s) \rightarrow (\rho_j,\theta_j)=(\rho,\theta)\) and we observe that 
\begin{eqnarray*}
(\nabla \psi)^2= (\p_r \psi)^2+\frac{1}{r^2}(\p_s \psi)^2 =(\p_\l \psi)^2+\frac{1}{\l^2}(\p_\theta\psi)^2,
\end{eqnarray*}
and
\begin{equation*}\begin{split}
\e^2(\p_s \psi)^2 =&\,\e^2 (\p_\theta \psi)^2 +\e^2\dd \left( \sin \theta \p_\l \psi \p_\t \psi+\frac{\cos \t}{\l} (\p_\t \psi)^2 \right) \\
&\,+\e^2 \dd^2 \left(\sin^2 \t (\p_\l \psi)^2+\frac{4\cos \t \sin \t}{\l}\p_\l \psi \p_\t \psi +\frac{\cos^2 \t}{\l^2}(\p_\t \psi)^2\right).
\end{split}\end{equation*}
Thus component-wise we obtain
\begin{equation*}\begin{split}
\left(\tilde{\mathcal{N}}(\psi) \right)_1=&\,2(\p_\l \psi_1)(\p_\l \psi_2)+2(\p_\t \psi_1)(\p_\t \psi_2)\left(\e^2+\frac{1}{\l^2} \right) \\
&\,+\e^2\dd \left( \sin \t[\p_\l \psi_1 \p_\t \psi_2+\p_\t \psi_1 \p_\l \psi_2 ]+\frac{2\cos \t}{\l}\p_\t \psi_1 \p_\t \psi_2 \right) \nonumber \\
&\,+\e^2\dd^2\left( 2\sin^2\t \p_\l \psi_1 \p_\l\psi_2+\frac{4\sin \t \cos \t}{\l}[\p_\l \psi_1\p_\t \psi_2+\p_\t \psi_1\p_\l \psi_2]\right. \\
&\,\left.+\frac{2\cos^2\t}{\l^2}\p_\t\psi_1 \p_\t \psi_2\right),
\end{split}\end{equation*}
\begin{equation*}\begin{split}
\left(\tilde{\mathcal{N}}(\psi)  \right)_2=&\,-(\p_\l \psi_1)^2+(\p_\l \psi_2)^2-\left(\e^2+\frac{1}{\l^2} \right)((\p_\t \psi_2)^2-(\p_\t \psi_1)^2) \\
&\,+\e^2\dd\left(\sin \t (\p_\l \psi_1\p_\t \psi_1+\p_\l\psi_2\p_\t \psi_2)+\frac{\cos \t}{\l}[(\p_\t \psi_2)^2-(\p_\t \psi_1)^2] \right) \\
&\,+\e^2\dd^2 \left(\sin^2\t ((\p_\l \psi_2)^2-(\p_\l \psi_1)^2)+\frac{4\cos \t \sin \t}{\l}(\p_\l \psi_1\p_\t \psi_1+\p_\l \psi_2\p_\t \psi_2)\right. \\
&\,\left.+\frac{\cos^2\t}{\l^2}[(\p_\t \psi_2)^2-(\p_\t \psi_1)^2 ] \right) +|V_d|^2(1-e^{2\psi_2}-2\psi_2). \nonumber
\end{split}\end{equation*}

We define
\begin{equation*}
\mathcal{A}_1(\psi):= 2(\p_\l \psi_1)(\p_\l \psi_2)+2(\p_\t \psi_1)(\p_\t \psi_2)\left(\e^2+\frac{1}{\l^2} \right),
\end{equation*}
\begin{equation*}
\mathcal{B}_1(\psi):= \e^2\dd \left( \sin \t[\p_\l \psi_1 \p_\t \psi_2+\p_\t \psi_1 \p_\l \psi_2 ]+\frac{2\cos \t}{\l}\p_\t \psi_1 \p_\t \psi_2 \right),
\end{equation*}
$$
\mathcal{C}_1(\psi):=\e^2\dd^2\left( 2\sin^2\t \p_\l \psi_1 \p_\l\psi_2+\frac{4\sin \t \cos \t}{\l}[\p_\l \psi_1\p_\t \psi_2+\p_\t \psi_1\p_\l \psi_2]+\frac{2\cos^2\t}{\l^2}\p_\t\psi_1 \p_\t \psi_2\right).
$$
We have $\left(\mathcal{N}(\psi)\right)_1= \mathcal{A}_1(\psi)+\mathcal{B}_1(\psi)+\mathcal{C}_1(\psi)$. Besides we can see that
\begin{equation*}
\frac{\left|\mathcal{B}_1(\psi)\right|}{\rho_1^{-1}+\rho_2^{-1}} \leq  \frac{C\e}{\sqrt{|\log \e|}}\left( \frac{|\nabla\psi_1|}{\rho_1^{-1}+\rho_2^{-1}} \times \frac{|\nabla \psi_2|}{\rho_1^{-1}+\rho_2^{-1}} \right) \leq  \frac{C \e}{\sqrt{|\log \e}|} \|\psi\|_*^2
\end{equation*}
 for \(3<\rho_1<R_\e\). Now by using the argument that a product of two $\pi$-periodic functions is $\pi$-periodic and the products of one $\pi$-periodic function and one $2\pi$-periodic function is $2\pi$-periodic and \eqref{eq:key_decomposition} we find that
\begin{equation}
\nonumber
\frac{\left|[\mathcal{A}_1(\psi)+\mathcal{C}_1(\psi)]^o\right|}{\rho_1^{-1}+\rho_2^{-1}}\leq C \left(\|\psi\|_{*}|\psi^{o}|_{\sharp}+|\psi^{o}|_{\sharp}^2\right).
\end{equation}
Thus we obtained that
\begin{equation*}
\frac{\left|\left(\mathcal{N}(\psi)\right)_1^o \right|}{\rho_1^{-1}+\rho_2^{-1}}\leq C\left(\|\psi\|_{*}|\psi^{o}|_{\sharp}+|\psi^{o}|_{\sharp}^2+\frac{C \e}{\sqrt{|\log \e}|} \|\psi\|_*^2 \right)
\end{equation*}
for \(3<\rho_1<R_\e\). 
We also have that
\begin{equation*}
1-e^{2\psi_2}-2\psi=\left( 1-e^{2\psi_2^e}-2\psi_2^e \right)+\left(1-e^{2\psi_2^o}-2\psi_2^o \right)e^{2\psi_2^e}+2\psi_2^o\left( e^{2\psi_2^e}-1\right).
\end{equation*}
We notice that \(1-e^{2\psi_2^e}-2\psi_2^e\) is a \(\pi\)-periodic function. Thus we find that
\begin{equation*}
\left|(1-e^{2\psi_2}-2\psi_2)^o \right|\leq C\left( |\psi_2^o||\psi^e|+|\psi^o|^2 \right).
\end{equation*}
By using again that a product of two $\pi$-periodic functions is $\pi$-periodic and the products of one $\pi$-periodic function and one $2\pi$-periodic function is $2\pi$-periodic and \eqref{eq:key_decomposition} we can obtain that
\begin{equation*}
\left|\left(\mathcal{N}(\psi)\right)_2^o\right| \leq C\left(\|\psi\|_{*}|\psi^{o}|_{\sharp}+|\psi^{o}|_{\sharp}^2+\frac{C \e}{\sqrt{|\log \e}|} \|\psi\|_*^2 \right)
\end{equation*}
We proceed in the same way to estimate the other terms in \(\mathcal{N}(\psi)\) when \(\rho_1<3\) or \(\rho_2<3\) and we use repeatedly \eqref{eq:key_decomposition} to arrive at \eqref{eq:estimate_Npsi_sharp}.

We now show that $G_\e$ is a contraction for $\e$ small enough. Indeed, if $\|\psi^j\|_{*}\leq \frac{C}{|\log \e|}$ for $j=1,2$ then
\begin{equation*}
\|\mathcal{N}(\psi^1)-\mathcal{N}(\psi^2)\|_{**}\leq \frac{C}{|\log \e|} \|\psi^1-\psi^2 \|_{*}.
\end{equation*}
 This is mainly due to the fact that $N(\psi)$ is quadratic and cubic in $\psi$, and in the first and second derivatives of $\psi$. Then we can use $a^2-b^2=(a-b)(a+b)$ and $a^3-b^3=(a-b)(a^2+ab+b^2)$. We finally apply the Banach fixed point theorem and we find the desired solution.
\end{proof}

\section{Solving the reduced problem}
The solution $\psi_\e$ of \eqref{eq:nonlinear} previously found depends continuously on $\d:=\sqrt{|\log\e|}d$. We want to find $\d$ such that the Lyapunov-Schmidt coefficient in \eqref{eq:nonlinear} satisfies $c=c(\d)=0$. We let
$$\varphi_\e:= \eta iV_d \psi_{\e}+(1-\eta)V_d e^{i\psi_{\e}}\qquad\mbox{ and }\qquad\phi_\e:=iV_d\psi_\e,$$
where $\eta$ was defined in \eqref{eta1}.
By symmetry we work only in $\R^+\times \R$. From the previous section we have found \(\psi_\e\) such that
\begin{equation*}\begin{split}
 iW(z)\left[\mathcal{L}^\e(\psi_\e)+R+\mathcal{N}(\psi_\e) \right](z+\dd) = c\chi W_{x_1}.
\end{split}\end{equation*}
For \(R_\e\) defined in \eqref{Reps} we set
\begin{equation*}
c_*:= \RE \int_{B(0,R_\e)} \chi |W_{x_1}|^2=\RE \int_{B(0,4)}\chi |W_{x_1}|^2,
\end{equation*}
and we remark that this quantity is of order $1$.
We find that
\begin{multline*}
c c_*=\RE \int_{B(0,R_\e)}iW(z) R(z+\dd) \overline{W}_{x_1}(z) + \RE \int_{B(0,R_\e)}iW(z)\L^\e(\psi_\e)(z+\dd) \overline{W}_{x_1}(z) \\ +\RE \int_{B(0,R_\e)} iW \mathcal{N}(\psi_\e)(z+\dd)\overline{W}_{x_1}.
\end{multline*}
We recall that \(iW(z)\mathcal{L}^\e(\psi)(z+\dd)=L^\e_j(\phi_j)\) for \(j=1\) and \(L_j\) defined in \eqref{defL_j}. Integrating by parts we find
\begin{equation*}
\RE \int_{B(0,R_\e)} L_j^\e(\phi_j)  \overline{W}_{x_1} =\RE \int_{B(0,R_\e)}\overline{\phi_{j}}(L_j^\e-L^0)(W_{x_1})+\RE \Bigl\{ \int_{\p B(0,R_\e)} \left(\frac{\p \phi_j}{\p \nu}\overline{W}_{x_1}-\phi_j\frac{\p \overline{W}_{x_1}}{\p\nu}\right) \Bigr\} .
\end{equation*}
Proceeding like in \eqref{eq:forfutureref} we conclude
\begin{equation*}
\left|\RE \int_{B(0,R_\e)} L_j^\e(\phi_j)\overline{W}_{x_1} \right| \leq C \e \sqrt{|\log \e|}\|\psi\|_* \leq \frac{C \e}{\sqrt{|\log \e|}}. \nonumber
\end{equation*}

Now we estimate the inner product of \(W_{x_1}\) and \(iW(z) \mathcal{N}(\psi)(z+\dd)\). We use the orthogonality of the Fourier modes to write
\begin{eqnarray*}
\RE \int_{B(0,R_\e)} iW(z) \mathcal{N}(\psi)(z+\d) \overline{W}_{x_1} & =& \RE \int_{B(0,R_\e)}iW \overline{W}_{x_1} \left(\mathcal{N}(\psi) \right)^o \nonumber \\
&=& \RE \int_{B(0,R_\e)} iw \left( w' \cos \theta-\frac{i w}{\rho}\sin \theta\right)\left[ \mathcal{N}(\psi)^o_1+i\mathcal{N}_2^o(\psi) \right] \nonumber \\
&=-& \int_{B(0,R_\e)} \left( ww' \cos \theta \mathcal{N}(\psi)^o_2-\frac{w^2}{r}\mathcal{N}(\psi)^o_1 \sin \theta \right).
\end{eqnarray*}
We use that
\begin{align*}
|\left(\mathcal{N}(\psi)\right)^o_2| &\leq |\left(\mathcal{N}(\psi)\right)^o_2|_{\sharp \sharp} \leq C \| \psi^e\|_{*}|\psi^o|_{\sharp}+|\psi^o|_\sharp^2 \leq C\e |\log \e|^{-1/2} \\
|\left(\mathcal{N}(\psi)\right)^o_1| & \leq C \left( \frac{|\psi_2^o|_\sharp \|\psi_1^e\|_*}{1+\rho^2} +\frac{|\psi_1^o|_\sharp \|\psi_2\|_*}{1+\rho^{2-\sigma}}+\frac{|\psi_1^o|_\sharp |\psi_2^o|_\sharp}{1+\rho^{2-\sigma}} \right)\leq C \frac{\e |\log\e|^{-1/2}}{1+\rho^{2-\sigma}}
\end{align*}
to obtain
\begin{equation*}
\left|\RE \int_{B(0,R_\e)} iW(z) \mathcal{N}(\psi)(z+\dd) \overline{W}_{x_1} \right| \leq C \frac{\e}{\sqrt{|\log \e|}}.
\end{equation*}

Now we claim 
%recall that \( iW(z) \mathcal{E}(z+\dd)=R_j\) for \(j=1\) here (since we work in the half-plane \(\R^*_+\times \R\)) and we claim that
\begin{equation*}
\RE \int_{B(0,R_\e)} iW(z) R(z+\dd)  \overline{W}_{x_1}=-\e \sqrt{|\log \e|}\left( \frac{a_0}{\d}-a_1\d\right)+o_\e(\e \sqrt{|\log \e|})
\end{equation*}
%for some constants $a_1,a_0$ satisfying $c\leq a_0,a_1<C$ for $c,C>0$ independent of $\e$. Indeed we recall that $R_j=R_j^0+R_j^1$ with $R_j^0=iWR^0$, $R_j^1=iWR^1$,
%$$S_0(V_d)=iV_d R^0, \ \ \ S_1(V_d)=iV_dR^1,$$
%where $S_0,S_1$ are given by \eqref{def:S_0S_1}. 
We set
\begin{eqnarray*}
B_0:= \RE \int_{B(0,R_\e)} iW(z) R(z+\dd)^0\overline{W}_{x_1}, \qquad B_1:= \RE \int_{B(0,R_\e)} iW(z) R(z+\dd)^1 \overline{W}_{x_1},
\end{eqnarray*}
where we recall that $S_0(V_d)=iV_d R^0, \ \ \ S_1(V_d)=iV_dR^1$ and $S_0,S_1$ are given by \eqref{def:S_0S_1}. 

From Lemma \ref{lem:propertiesofrho} and Lemma \ref{lem:ortho_2} we find that
\begin{equation*}\begin{split}
B_1= &\,\frac{\d \e}{\sqrt{|\log \e|}}\RE \int_{\{\l_1<R_\e \}}|W_{x_1}|^2 +O_\e\left(\frac{\e}{{\sqrt{|\log \e|}}}\right)=\, \d \e \sqrt{|\log \e|} a_1+o_\e(\e \sqrt{| \log \e|} )\nonumber
\end{split}
\end{equation*}
where we set
\begin{equation*}
a_1:=\frac{1}{|\log \e|}\int_0^{2\pi}\int_0^{\tilde{R}_\e} \frac{\r_1^2\sin^2\theta_1}{\l_1} \dif \l_1\dif \theta_1,
\end{equation*}
with \(\tilde{R}_\e\) which is of order \(\e^{-1}|\log \e|^{-1/2}\) and which does not depend on \(\d\). From the fact that $\lim_{\l\rightarrow +\infty}\r(\l)=1$ we can see that $0<c<a_1<C$ for some constants $c,C>0$, and \(a_1\) is independent of \(\d\).

On the other hand, by \eqref{S0Vd} we have
\begin{equation*}\begin{split}
B_0=&\, \RE \int_{\{\l_1<R_\e\}} 2\frac{(W^a_{x_1}W^b_{x_1}+W^a_{x_2}W^b_{x_2})}{W^b}\overline{W}^a_{x_1} \\
&\,+\RE \int_{\{\l_1<R_\e \}}(1-|W^aW^b|^2+|W^a|^2-1+|W^b|^2-1) W^a\overline{W}^a_{x_1}.
\end{split}\end{equation*}
The second integral is equal to
\begin{equation*}\begin{split}
\RE\int_{\{\l_1<R_\e \}}(1-(\r_1\r_2)^2&+\r_1^2-1+\r_2^2-1)\left(\r_1'\cos \t_1+\frac{i\r_1}{\l_1}\sin \t_1\right)\r_1 = O_\e(\e^2 |\log \e|),\nonumber
\end{split}\end{equation*}
where we used that $(1-(\r_1\r_2)^2+\r_1^2-1+\r_2^2-1)=O(\e^2|\log \e|)$ and $\r'(\l)=1/\l^3+O_\l(1/\l^4)$.
We can also see that
\begin{equation*}\begin{split}
\RE \int_{\{\l_1<R_\e \}} &\frac{W_{x_1}^aW_{x_1}^b}{ W^b}\overline{W}^a_{x_1}\\
=&\,\RE \int_{\{\l_1<R_\e \}}\left(\r_1'\cos \t_1+i\frac{\r_1}{\l_1}\sin \t_1\right) \Bigl[\r_1' \frac{\r_2'}{\r_2}\cos \t_1\cos \t_2-\frac{\r_1 }{\l_1 \l_2}\sin \t_1 \sin \t_2 \nonumber \\
&  \phantom{aaaaa}-i\left( \frac{\r_1'}{\l_2}\cos \t_1 \sin \t_2+\frac{\r_2' \r_1}{\r_2\l_1}\cos \t_2 \sin \t_1 \right) \Bigr] \nonumber \\
=&\, -\int_{\{\l_1<R_\e \}}  \r_1\r_1' \cos \t_1 \sin \t_1 \sin \t_2\frac{\mathrm{d} \l_1}{\l_2}\mathrm{d}\t_1  \nonumber \\
& \phantom{aaaaa} +\int_{\{\l_1<R_\e\}} \r_1'\r_1\cos\t_1 \sin\t_1\sin \t_2 \frac{\mathrm{d} \l_1}{\l_2}\mathrm{d}\t_1+O(\e^2|\log \e|).
\end{split}\end{equation*}
In the previous equality we used
$ \r_2' \leq C \e^3|\log \e|^{3/2}. $
Hence we get
\begin{equation*}
\RE \int_{\{\l_1<R_\e \}} \frac{W_{x_1}^aW_{x_1}^b}{W^b} \overline{W}^a_{x_1}=O_\e(\e^2|\log \e|).
\end{equation*}
Finally we have
\begin{equation*}\begin{split}
\RE \int_{\{\l_1<R_\e \}} & \frac{W_{x_2}^aW_{x_2}^b}{W^b} \overline{W}^a_{x_1}\\
=&\,\RE \int_{\{\l_1<R_\e\}} \left(\r_1'\cos \t_1+i\frac{\r_1}{\l_1}\sin \t_1\right)\Bigl[\r_1' \frac{\r_2'}{\r_2}\sin \t_1 \sin \t_2-\frac{\r_1 }{\l_1\l_2}\cos \t_1 \cos \t_2 \nonumber \\
& \,\phantom{aaaa} +i\left(\frac{\r_1' }{\l_2}\sin \t_1 \cos \t_2+\frac{\r_2' \r_1}{\r_2\l_1}\cos \t_1 \sin \t_1 \right) \Bigr] \nonumber \\
=&\, -\int_{\{\l_1<R_\e\}}  \frac{\r_1 \r_1'}{\l_1 \l_2}\cos^2\t_1 \cos \t_2 \l_1  \dif\l_1\dif \t_1 \nonumber \\
&\,\phantom{aaaa}- \int_{\{\l_1<\dd\}} \frac{\r_1 \r_1'}{\l_1 \l_2}\sin^2\t_1 \cos \t_2 \l_1 \mathrm{d}\l_1 \mathrm{d}\t_1+O(\e^2|\log \e|) \nonumber \\
=&\, \int_{\{\l_1<R_\e\}} \frac{\r_1 \r_1'}{ \l_2} \cos \t_2 \mathrm{d}\l_1 \mathrm{d} \theta_1+O(\e^2|\log \e|). \nonumber
\end{split}\end{equation*}
Using the properties of $\r_1,\r_1'$ and that in this region $\cos \t_2 >0$ and $0<c<\frac{\l_2 \e \sqrt{|\log \e|}}{\d}<C$ for some constants $c,C>0$ and \(\frac{\l_2 \e \sqrt{|\log \e|}}{\d}\) is independent of \(\d\) in the region \(0<\rho_1<R_\e\), we find
\begin{equation*}
\RE \int_{\{\l_1<R_\e \}} \frac{W_{x_2}^aW_{x_2}^b}{W^b} \overline{W}^a_{x_1}=-\frac{a_0}{\d} \e \sqrt{|\log \e|}+o_\e(\e\sqrt{|\log \e| }),
\end{equation*}
with $c<a_0<C$ for some constants $c,C>0$ and independent of \(\d\).
%\begin{equation*}
%a_0:= \d \int_{\{\l_1<R_\e\}} \frac{\r_1 \r_1'}{ \e \sqrt{|\log \e|}\l_2} \cos \t_2 \mathrm{d}\l_1 \mathrm{d} \theta_1
%\end{equation*}

Therefore, we conclude that
\begin{equation*}
 c c_*=\e \sqrt{|\log \e|}\left( \frac{a_0}{\d}-a_1\d\right)+o_\e(\e \sqrt{|\log \e|}).
\end{equation*}
Let us point out that in this expression $o_\e(\e \sqrt{|\log \e|})$ is a continuous function of the parameter $\d$. By applying the intermediate value theorem we can find $\d_0$ near $\sqrt{\frac{a_0}{a_1}}$ such that $c=c(\d_0)=0$.
For such $\d_0$ we obtain that $V_d+\varphi_\e$ is a solution of \eqref{eq:GLequation}.   To conclude the proof of Theorem \ref{th:main1}, thanks to the helical symmetry, it suffices to show that the solutions of the 2-dimensional problem we found satisfy \(\lim_{|z|\rightarrow +\infty} |V_\e(z)|=1\). But this is because far away from the vortices our solution takes the form \( V_\e(z)=W(z-\dd)W(z+\dd)e^{i{\psi_\e}}\). Thus \(|V_\e|=|W(z-\dd)W(z+\dd)|e^{-\IM \psi_\e} \). Thanks to the decay estimates obtained on \(\psi_\e\) we have that \(|\IM \psi |\leq \frac{C}{|\log \e|}\left(\frac{1}{1+|z-\dd|^{2-\sigma}}+\frac{1}{1+|z+\dd|^{2-\sigma}} \right) \). This proves that \(\lim_{|z|\rightarrow +\infty} |V_\e(z)|=1\) and thus that the solution of the 3-dimensional problem satisfies \eqref{eq:prop_Gibbons}.

\section*{Acknowledgments}

The third author is supported by the European Union's Horizon $2020$ research and innovation programme under the Marie Sklodowska-Curie grant agreement N.$754446$ and UGR Research and Knowledge Transfer - Found Athenea3i. R.R is supported by the F.R.S.–FNRS under the Mandat d'Impulsion
scientifique F.4523.17,``Topological singularities of Sobolev maps". Part of this work was realized during a visit of the fourth author at the Universidad de Chile and Pontificia Universidad Cat\'olica de Chile supported by the project REC05 Wallonia-Brussels/Chile.

\section*{Appendix}
\subsection{The standard vortex and its linearized operator}
As stated in the introduction, the building block used to construct our solutions to equation \eqref{GL2-dimensionalrescaled} is the  standard vortex of degree one, $W$, in $\R^2$. It satisfies
$$\Delta W+(1-|W|^2)W=0\qquad \mbox{ in }\R^2,$$
and can be written as
$$W(x_1,x_2)= w(r)e^{i\v}\mbox{ where }x_1=r \cos \v,\; x_2=r \sin \v.$$
Here $w$ is the unique solution of \eqref{eq:equation_modulus_w}.
In this section we collect useful properties of $\r$ and of the linearized Ginzburg-Landau operator around $W$.
\begin{lemma}\label{lem:propertiesofrho}
The following properties hold,
\begin{itemize}
\item[1)] $\r(0)=0$, $\r'(0)>0$, $0<\r(r)<1$ and $\r'(r)>0$ for all $r>0$,
\item[2)] $\r(r)=1-\frac{1}{2r^2}+O(\frac{1}{r^4})$ for large $r$,
\item[3)] $\r(r)=\alpha r-\frac{\alpha r^3}{8}+O(r^5)$ for $r$ close to $0$ for some $\alpha>0$,
\item[4)] if we define $T(r)=\r'(r)-\frac{\r}{r}$ then $T(0)=0$ and $T(r)<0$ in $(0,+\infty)$,
\item[5)] $\r'(r)=\frac{1}{r^3}+O(\frac{1}{r^4})$, $\r''(r)=O(\frac{1}{r^4})$.
\end{itemize}
\end{lemma}
For the proof of this lemma we refer to \cite{HerveHerve1994, ChenElliottQi1994}.

An object of special importance to construct our solution is the linearized Ginzburg-Landau operator around $W$, defined by
\begin{equation}\nonumber%\label{eq:GLlinearizedoperator}
L^0(\phi):= \Delta \phi+(1-|W|^2)\phi-2\RE(\overline{W}\phi)W.
\end{equation}
This operator does have a kernel, as the following result states.

\begin{lemma}\label{lem:ellipticestimatesL0}
Suppose that $\phi \in L^\infty(\R^2)$ satisfies $L^0(\phi)=0$ in $\R^2$ and the symmetry  \(\phi(\bar{z})=\bar{\phi}(z)\) .
Assume furthermore that when we write $\phi=iW\psi$ and $\psi=\psi_1+i\psi_2$ with $\psi_1,\psi_2\in \R$ we have
\begin{eqnarray*}
|\psi_1|+(1+|z|)|\nabla \psi_1| \leq  C, \qquad |\psi_2|+|\nabla \psi_2| \leq  \frac{C}{1+|z|} ,\quad |z|>1.
\end{eqnarray*}
Then
\begin{equation*}
\phi=c_1 W_{x_1}
\end{equation*}
for some real constant \(c_1\).
\end{lemma}

\begin{proof}
The equation $L_0(\phi)=0$ in $B(0,1)^c$ translates into
\begin{equation*}
\Delta \psi+2\frac{\nabla W}{W} \nabla \psi-2i|W|^2\IM \psi=0 \text{ in } B(0,1)^c.
\end{equation*}
This reads
\begin{align*}
0 & = \Delta \psi_1 +\frac{2\r'}{\r}\p_r \psi_1+\frac{2}{r^2}\p_\t \psi_2 \quad \text{ in } B(0,1)^c
\\
0 & = \Delta \psi_2+\frac{2\r'}{\r}\p_r \psi_2-\frac{2}{r^2}\p_\t \psi_1-2|W|^2\psi_2 \quad \text{ in } B(0,1)^c.
\end{align*}
We thus have, by using the decay assumption on $\psi_1,\psi_2$, that
\[ \left| \Delta \psi_2-2|W|^2\psi_2 \right|\leq \frac{C}{1+r^2} \text{ in } B(0,1)^c. 
\]
Since \(|W|^2\geq C>0\) in \(B(0,1)^c\) we can use a barrier argument and elliptic estimates to obtain that 
\begin{equation}\label{eq:est_psi_2_decay}
(1+|z|^2)\left(|\psi_2|+|\nabla \psi_2|\right) \leq C.
\end{equation}
We can then use the previous estimate to obtain that 
\[ |\Delta \psi_1|\leq \frac{C}{1+r^3} \text{ in } B(0,1)^c.
\]
We use that \(\psi_1(z = x_1 + ix_2)=0\) for $x_2=0$, a barrier argument in the half plane and elliptic estimates to obtain
\begin{equation}\label{eq:est_psi_1_decay}
|\psi_1|+(1+|z|)|\nabla \psi_1| \leq \frac{C}{(1+|z|)^\alpha}. 
\end{equation}
for any $\alpha \in (0,1)$.
From \eqref{eq:est_psi_2_decay} and \eqref{eq:est_psi_1_decay} we get
\begin{align}
\label{decay-phi}
|\phi (z)| + (1+|z|) |\nabla \phi|) \leq \frac{C}{(1+|z|)^\alpha}, \quad |z|>1.
\end{align}
From the fact $L^0(\phi)=0$ in $\R^2$ we know that
$$\RE \int_{B_R(0)}\overline{\phi}\Delta\phi+\int_{B_R(0)}(1-|W|^2)|{\phi}|^2-2\int_{B_R(0)}|\RE (\overline{W}{\phi})|^2=0,$$
for any $R>0$. Then integrating by parts we get
$$\int_{B_R(0)}|\nabla\phi|^2-\RE \int_{\p B_R(0)}\overline{\phi}  \p_\nu\phi-\int_{B_R(0)}(1-|W|^2)|{\phi}|^2+2\int_{B_R(0)}|\RE (\overline{W}{\phi})|^2=0.$$
Using \eqref{decay-phi}
we find $|\RE \left(\overline{\phi}  \p_\nu\phi\right)| \leq C/(1+|z|^{2\alpha+1})$. Thus
$$ \Bigl| \RE \int_{\p B_R(0)}\overline{\phi}  \p_\nu\phi \Bigr| \leq \frac{C}{R^{2\alpha}}.$$
Making $R\rightarrow\infty$ we conclude
$$\int_{\R^2}|\nabla\phi|^2-\int_{\R^2}(1-|W|^2)|{\phi}|^2+2\int_{\R^2}|\RE (\overline{W}{\phi})|^2=0.$$
Thanks to the decay estimates \eqref{decay-phi} we also have
\begin{equation*}
\int_{\R^2}\left[|\nabla \phi|^2+(1-|W|^2)|\phi|^2+|\RE(\bar{W}\phi)|^2 \right]<+\infty.
\end{equation*}
We can then apply Theorem 1 in \cite{delPinoFelmerKowalczyk2004} to obtain that there exists \(c_1,c_2\in \R\) such that \(\phi=c_1W_{x_1}+c_2W_{x_2}\). Using the symmetry assumption \(\phi(\bar{z})=\bar{\phi}(z)\) we can conclude that actually \(\phi=c_1W_{x_1}\) for some \(c_1\in \R\).
\end{proof}

\begin{lemma}
\label{lem:ellipticestimatesL0-b}
Suppose that $\phi \in L_{\text{loc}}^\infty(\R^2)$ satisfies $L^0(\phi)=0$ in $\R^2$ and the symmetry  \(\phi(\bar{z})=\overline{\phi(z)}\) .
Assume furthermore that when we write $\phi=iW\psi$ and $\psi=\psi_1+i\psi_2$ with $\psi_1,\psi_2\in \R$ we have
\begin{align*}
|\psi_1|+(1+|z|)|\nabla \psi_1| \leq  C(1+|z|)^\alpha, \qquad |\psi_2|+|\nabla \psi_2| \leq  \frac{C}{1+|z|} ,
\quad |z|>1,
\end{align*}
for some $\alpha<3$.
Then
\begin{equation*}
\phi=c_1 W_{x_1}
\end{equation*}
for some real constant \(c_1\).
\end{lemma}
\begin{proof}
Here we work with the change of variables 
$\phi = e^{i\theta} \psi$.
Then the equation $L^0(\phi)=0$ becomes
\begin{align}
%\label{linear-psi2} 
\nonumber
0= \Delta \psi 
-\frac{1}{r^2}\psi
+ \frac{2i}{r^2}\partial_\theta\psi
+ (1-w^2) \psi 
-  2 i w^2 \IM(   \psi) .
\end{align}
Writing $\psi = \psi_1 + i \psi_2$ with $\psi_1,\psi_2\in\R$ we get the system
\begin{align}
%\label{main-linear-psi}
\nonumber
\left\{
\begin{aligned}
0 &=
\Delta \psi_1  - \frac{1}{r^2} \psi_1 -\frac{2}{r^2} \partial_\theta \psi_2 + (1-w^2)\psi_1
\\
0 &=
\Delta \psi_2 -\frac{1}{r^2} \psi_2 + \frac{2}{r^2}\partial_\theta \psi_1 
+ (1-3 w^2)\psi_2  ,
\end{aligned}
\right.
\end{align}
which holds in $\R^2 \setminus \{0\}$ with the symmetry condition $\psi_1(\bar z ) = - \psi_1(z)$, $\psi_2(\bar z) = \psi_2(z)$.

We decompose in Fourier modes
\begin{align}
\nonumber
\psi_1 = \sum_{k=1}^\infty \psi_{1,k}^2(r) \sin(k\theta) ,
\quad 
\psi_2 = \sum_{k=0}^\infty  \psi_{2,k}^1(r) \cos(k\theta),
\end{align}
and obtain
\begin{align}
\label{sys1b1}
0
 &=
\partial_{rr} \psi_{1,k}^2
+\frac{1}{r}\partial_r \psi_{1,k}^2
-\frac{k^2+1}{r^2} \psi_{1,k}^2
+ 2\frac{k}{r^2} \psi_{2,k}^1
+ (1-w^2) \psi_{1,k}^2
\\
\label{sys1b2}
0
 &=
\partial_{rr} \psi_{2,k}^1
+\frac{1}{r}\partial_r \psi_{2,k}^1
-\frac{k^2+1}{r^2} \psi_{2,k}^1
+ 2\frac{k}{r^2} \psi_{1,k}^2
+ (1-3w^2) \psi_{2,k}^1 .
\end{align}

In particular equation \eqref{sys1b1} for $k=1$ can be written as 
\begin{align*}
\partial_{rr} \psi_{1,1}^2
+\frac{1}{r}\partial_r \psi_{1,1}^2
-\frac{1}{r^2} \psi_{1,1}^2
=
g_0
\end{align*}
where
\[
g_0(r) = -2\frac{1}{r^2} \psi_{2,1}^1
+ \Bigl(w^2-1+\frac{1}{r^2} \Bigr)\psi_{1,1}^2
=
O(r^{\alpha-4})
\]
as $r\to\infty$. The variation of parameters formula yields a function
\begin{align*}
\psi_0(r) = -\frac{1}{r}
\int_0^r \rho \int_\rho^\infty g_0(s) \,\dif s \dif \rho ,
\end{align*}
which satisfies
\begin{align}
\nonumber
\partial_{rr} \psi_0
+\frac{1}{r}\partial_r \psi_0
-\frac{1}{r^2} \psi_0
& =
g_0
\quad\text{for }r>1,
\\
\label{decay}
|\psi_0(r) | &\leq C r^{\alpha-2},
\quad 
|\partial_r\psi_0(r) | \leq C r^{\alpha-3},
\quad\text{for }r>1 .
\end{align}
Hence 
\begin{align}
\label{decompPsi1}
\psi_{1,1}^2(r) = \psi_0(r) +  \alpha_1 r + \alpha_2 r^{-1} , \quad r>1,
\end{align}
for some $\alpha_1,\alpha_2 \in \R$.
We claim that $\alpha_1 = 0$. To prove this, we note that for $k=1$ the system  \eqref{sys1b1}-\eqref{sys1b2} has the explicit solution 
\begin{align*}
\bar\psi = 
\left[
\begin{matrix}
\bar\psi_1
\\
\bar\psi_2
\end{matrix}
\right]
, \quad
\bar\psi_1 = \frac{w(r)}{r} , \quad 
\bar\psi_2 = -w'(r)
\end{align*}
Let $\psi = \left[
\begin{matrix}
\psi_{1,1}^2 
\\ 
\psi_{2,1}^1
\end{matrix}
\right]  $ and define the Wronskian
\begin{align*}
W(r) = \psi \cdot \bar\psi_r - \psi_r \cdot \bar\psi.
\end{align*}
We claim that 
\begin{align}
\label{wronsk}
W(r) = \frac{c}{r}
\end{align}
for some $c\in \R$.
To prove this note that the system \eqref{sys1b1}-\eqref{sys1b2} for $\psi$ can be written as 
\begin{align*}
0 = \psi_{rr} + \frac{1}{r}\psi_r + B \psi
\end{align*}
where $B$ is the $2\times 2$ matrix
\begin{align*}
B = 
\frac{2}{r^2}
\left[
\begin{matrix}
-1 & 1 \\ 1 & -1
\end{matrix}
\right]
+\frac{2}{r^2}
\left[
\begin{matrix}
1-w^2 & 0  \\ 0 & 1 - 3 w^2
\end{matrix}
\right].
\end{align*}
Then
\begin{align*}
W_r & = -\frac{1}{r}W + \psi^T ( B-B^T) \bar \psi =  -\frac{1}{r}W ,
\end{align*}
because the matrix $B $ is symmetric, and we get  \eqref{wronsk}.
Using the decomposition \eqref{decompPsi1}, the decay \eqref{decay} and the explicit form of $\bar\psi$ we see that
\begin{align*}
W(r) = -\frac{2\alpha_1}{r} .
\end{align*}

On the other hand, from the smoothness of $\phi$ near the origin we get that
 $\psi_{1,1}^2(r) $, $ \psi_{2,1}^1(r) $ and their derivatives remain bounded as $r\to 0$. Since the same is true for $\bar \psi$ we see that $W(r)$ is bounded as $r\to 0$, which implies that $\alpha_1=0$ as claimed.
This in turn implies that 
\begin{align*}
|\psi_{1,1}^2 | &\leq C r^{\alpha'},
\quad 
|\partial_r\psi_{1,1}^2 | \leq C r^{\alpha'-1},
\quad\text{for }r>1,
\end{align*}
where $\alpha'=\max(-1,\alpha-2)<1$.
Using barriers for the ODE \eqref{sys1b2} we get for $k=1$
\begin{align*}
|\psi_{1,1}^2 |\leq C ,\quad \text{for } r>1.
\end{align*}
and  for $k\geq 2$
\begin{align*}
|\psi_{1,k}^2(r)|\leq \frac{C}{k^2 r} ,\quad \text{for } r>1.
\end{align*}
Adding these inequalities we see that 
\begin{align*}
| \psi_1(z) | \leq C , \quad |z|>1,
\end{align*}
and then a standard scaling and elliptic estimates show that 
\begin{align*}
| \nabla \psi_1(z) | \leq \frac{C}{|z|} , \quad|z|>1.
\end{align*}
Now we can apply Lemma~\ref{lem:ellipticestimatesL0} and conclude that $\phi = c_1 W_{x_1}$ for some constant $c_1\in \R$.
\end{proof}

\subsection{Elliptic estimates used in the linear theory}\label{VIII}

In this subsection we prove elliptic estimates that we needed in Section \ref{IV} to develop the linear theory. More specifically we prove estimates of solutions to some model equations.

 We use the notation $z=(x_1,x_2)=re^{is}$ and throughout this section $\e>0$ is a parameter. We also use
$$\Delta = \p^2_{x_1 x_1}+\p^2_{x_2 x_2}=\p^2_{rr}+\frac{1}{r}\p_r+\frac{1}{r^2}\p^2_{ss}.$$
Furthermore in the equations the following term will appear: \begin{equation*}
\p^2_{ss} u=  x_2^2\p^2_{x_1x_1}u+x_1^2\p^2_{x_2 x_2}u-2x_1x_2 \p^2_{x_1x_2}u-x_1\p_{x_1}u-x_2\p_{x_2}u.
\end{equation*}
We start with recalling the statement and the proof of the comparison principle in the half-plane for the operator $\Delta +\e^2\p^2_{ss}$ with Dirichlet boundary condition.
\begin{lemma}\label{lem:comparison_principle}
Let $u:\R \times \R^*_+ \rightarrow \R$ be a bounded function which is in $C^2(\R \times \R^*_+)\cap C^0(\overline{\R \times \R^*_+})$ and which satisfies
\begin{equation}
\left\{
\begin{array}{rcll}
\Delta u+\e^2 \p_{ss}^2 u & \geq& 0 \text{ in } \R \times \R^*_+, \\
u&\leq & 0 \text{ on } \R \times \{0 \} .
\end{array}
\right.
\end{equation}
Then $u\leq 0$ in $\R \times \R^*_+$.
\end{lemma}
\begin{proof}
We adapt the proof of Lemma 2.1 in \cite{berestycki1997monotonicity}.

Let us use polar coordinates $(r,s)\in (0,+\infty) \times (0,\pi)$, and let $\varphi>0$ be the first eigenfunction of $\p^2_{ss}$ in $(-\frac{\pi}{4},\frac{5\pi}{4})$ associated to the eigenvalue $\mu>0$, i.e.,
\begin{equation*}
\left\{
\begin{array}{rcll}
\p^2_{ss} \varphi+\mu \varphi =0 \text{ on } (-\frac{\pi}{4},\frac{5\pi}{4}), \\
\varphi(-\frac{\pi}{4})=\varphi(\frac{5\pi}{4})=0.
\end{array}
\right.
\end{equation*}
We define $\beta:=\sqrt{\mu}$ and we set $g(r,s):=r^\beta \varphi(s)$ in $(0,+\infty)\times (-\frac{\pi}{4},\frac{5\pi}{4})$ and hence
\begin{equation*}
\p^2_{rr}g+\frac{1}{r}\p_rg+\left(\frac{1}{r^2}+\e^2\right)\p^2_{ss}g=-\mu\e^2 g \leq 0 \text{ in } (0,+\infty)\times (-\frac{\pi}{4},\frac{5\pi}{4}).
\end{equation*}
Consider $\sigma:=u/g$ in $(0,+\infty) \times (0,\pi)$ (note that $g>0$ in this domain). Since $\Delta u+\e^2\p^2_{ss}u\geq 0$ we find:
\begin{equation*}
\Delta \sigma+\e^2\p^2_{ss}\sigma +\frac{2}{g}\left[\p_rg\p_r \sigma+\left(\frac{1}{r^2}+\e^2\right)\p_sg\p_g\sigma\right]+\frac{\Delta g+\e^2\p^2_{ss}g}{g}\sigma\geq 0.
\end{equation*}
We note that $\frac{\Delta g+\e^2\p^2_{ss}g}{g}\sigma\leq 0$ and since $u$ is bounded $\limsup_{r\rightarrow +\infty} \sigma =0$. We can thus apply the maximum principle to deduce that $\sigma \leq 0$ in $(0,+\infty)\times(0,\pi)$. Hence $u \leq 0$ as well in $(0,+\infty)\times(0,\pi)$.
\end{proof}
In the same spirit we have the following comparison principle for Neumann boundary condition.

\begin{lemma}\label{lem:comparison_principle_Neumann}
Let \(u:\R\times \R^*_+ \rightarrow \R\) be a bounded function which is in \(C^2(\R\times\R^*_+)\cap C^1(\overline{\R\times\R^*_+})\). Let \(c\geq 0 \). We assume that \(u\) satisfies
\begin{equation*}
\left\{ 
\begin{array}{rcll}
\Delta u+\e^2 \p^2_{ss}u-c u & \geq & 0 \text{ in } & \R \times \R^*_+, \\
\p_\nu u &\leq & 0 \text{ on } & \R\times \{0 \},
\end{array}
\right.
\end{equation*}
then \(u\leq 0\) in \(\R\times \R^*_+\). 
\end{lemma}

For a function $f:\R^2 \rightarrow \R$ and $\nu \in \mathbb{N}^*,\alpha >0$ we introduce the norms:
$$\|f\|_{\nu,\alpha}:=\|(1+|z|^\nu)f\|_{L^\infty(\R^2)}+\sup_{z\in \R^2} |z|^{\nu+\alpha} [f]_{z,\alpha}$$
with
$$[f]_{z,\alpha}:=\sup_{|h|<1} \frac{|f(z+h)-f(z)|}{|h|^\alpha}.$$
Our first goal is to prove 
\begin{proposition}\label{prop:est_model_psi_1}
Let $f:\R^2\rightarrow \R$ be such that $f(\overline{z})=-f(z)$ and
$
\|f\|_{2,\alpha} <+\infty.
$
Then there exists a unique bounded solution of
\begin{equation}\label{eq:modelequation2}
\Delta u+\e^2\p^2_{ss} u=f \text{ in } \R^2
\end{equation}
which satisfies $u(\overline{z})=-u(z)$ and \begin{equation}\label{eq:fbounded}
|u(z)|\leq C \|f\|_{2,\alpha},\quad
%\end{equation}
%\begin{equation}\label{eq:decayonf1}
|\nabla u(z)| \leq C\frac{\|f\|_{2,\alpha}}{1+|z|} \ \text{ for all } z \text{ in } \R^2,
\end{equation}
\begin{equation}\label{eq:decayonf2}
 |\e \p_su(z)|\leq C\frac{\|f\|_{2,\alpha}}{|z|} \ \text { for } |z|\geq \frac{C}{\e},\quad
%\end{equation}
%and
%\begin{equation}\label{eq:decayonf3}
\|D^2u \|_{2,\alpha}\leq C \|f\|_{2,\alpha}.
\end{equation}
\end{proposition}
We first prove:

\begin{lemma}\label{lem:ellipticestimates}
 Let $f:\R^2\rightarrow \R$ be such that $f(\overline{z})=-f(z)$ and
$
\|f\|_{2,\alpha} <+\infty.
$
Let $u:\R^2\rightarrow \R$ be a bounded function such that $u(\bar{z})=-u(z)$ and 
\begin{equation}\label{eq:modelequation}
\Delta u+\e^2\p^2_{ss} u=f \text{ in } \R^2.
\end{equation}
Then there exists $C>0$ independent of $u,f,\e$ such that \eqref{eq:fbounded}, \eqref{eq:decayonf2} hold.
\end{lemma}

\begin{proof}
Thanks to the symmetry $u(\overline{z})=-u(z)$ it is sufficient to consider the problem
\begin{equation}\label{11}
\left\{
\begin{array}{rcll}
\Delta u  + \e^2 \p^2_{ss} u - f(z)   &= &0, \quad z \in \R\times \R^*_+ , \\
u(x_1,0)   & =& 0 \quad \foral x_1\in \R,
\end{array}
\right.
\end{equation}
which we can alternatively write as
\begin{equation*}
\left\{
\begin{array}{rcll}
\Delta u  + \e^2 \p^2_{ss} u - f  &= & 0, \quad  (r,s) \in (0,+\infty)\times (0,\pi) , \\
u(r,0)  =  u(r,\pi ) &=& 0.
\end{array}
\right.
\end{equation*}
Let us assume
$$ |f(z) | \le \frac 1 {1+|z|^2}. $$
We want to prove that for an absolute constant $C$ we have
$$|u(z)| \le C . $$
We define
\begin{equation*}
v(z)=v(r,s):=s(\pi-s).
\end{equation*}
We can check that
\begin{equation*}\begin{split}
\Delta v+\e^2\p^2_{ss}v+\frac{1}{1+r^2} =&\,\frac{-2}{r^2}-2 \e^2+\frac{1}{1+r^2} < 0 \quad \forall z=re^{is} \in \R \times \R^*_+. \nonumber
\end{split}\end{equation*}
Hence $v$ is a positive supersolution (and $-v$ a subsolution) for equation \eqref{11} in $(0,+\infty)\times(0,\pi)$ and in this set, for any bounded solution $u$ of \eqref{11} we have from Lemma \ref{lem:comparison_principle}
\begin{equation*}
|u(z)| \leq  |v(z)| \text{ in } \R \times \R^*_+.
\end{equation*}

\noindent The decay estimates in \eqref{eq:fbounded}-\eqref{eq:decayonf2} follow by Schauder estimates and a standard scaling argument.
\end{proof}

\begin{lemma}\label{lem:uniqueness}
\label{l1}
If $u$ is a bounded function that satisfies
\[
\Delta u  + \varepsilon^2 \partial_{ss}^2 u = 0
\quad \text{in }\R^2,\qquad
u(\bar z ) =- u(z) ,
\]
then $u\equiv 0$.
\end{lemma}
\begin{proof}
Suppose $u\not\equiv 0$ and assume without loss of generality that $\sup_{\R^2} u=1$.
By the strong maximum principle the supremum cannot be attained in $\R^2 \setminus \{  0 \}$.
Let $z_n\in \R^2$ be a sequence such that $u(z_n)\to 1$.
Up to a subsequence we have two possibilities: $z_n\to0$ or $|z_n|\to \infty$.

Case $z_n \to \infty$.
Let us write
\[
z_n = R_n e^{i\s_n},
\]
where $R_n\to \infty$ and $\s_n \in (0,\pi)$.
We express $u$ in polar coordinates $(r,s)$ and define
\[
\tilde u_n(r,s) := u(r+R_n,s).
\]
Up to a subsequence we have $\tilde u_n \to \tilde u$ uniformly in compact sets of $\R^2$, where $\tilde u\leq 1$, $\tilde u (p)=1$ for some point $p=(1,s)$ with $s \in [0,\pi]$, and
\[
\partial_{rr}^2 \tilde u + \varepsilon^2 \partial_{ss}^2 \tilde u = 0 \quad \text{in }\R^2 ,
\]
with the additional condition $\tilde u(r,0)=\tilde u (r,\pi) = 0$.
This contradicts the strong maximum principle.

Case $z_n \to 0$.
Let us write
\[
z_n = R_n e^{i\s_n},
\]
where $R_n\to 0$ and $\s_n \in (0,\pi)$.
Define
\[
\tilde u_n(\zeta) := u(R_n \zeta) .
\]
Up to a subsequence $\tilde u_n\to \tilde u$ uniformly in compact sets of $\R^2$, where $\tilde u\leq 1$ attains its maximum at some point and satisfies $\Delta \tilde u=0$ in $\R^2$.
This is a contradiction.
\end{proof}

\begin{proof}[Proof of proposition \ref{prop:est_model_psi_1}]
We use $v:=\|f\|_{2,\alpha}s(\pi-s)$ as a super-solution  to solve the problem in large half-balls centred at the origin. More precisely, for any $M>0$ there exists a solution of
\begin{equation*}
\left\{
\begin{array}{rcll}
\Delta u_M+\e^2 \p^2_{ss}u_M&=& f & \text{ in } B^+_M(0) \\
u_M&=&0 & \text{ on } \p B^+_M(0),
\end{array}
\right.
\end{equation*}
where $B^+_M(0):=\{(x_1,x_2) \in \R \times \R^+; |z|<M \}$. Thanks to gradient estimates \eqref{eq:fbounded} we have
\begin{equation*}
|\nabla u_M| \leq \frac{C |v|}{1+|z|} \text{ in } B^+_M(0) ,
\end{equation*}
for some $C>0$ independent of $M$ and thus we can apply Arzela-Ascoli theorem to take the limit of $u_M$ along a suitable subsequence, obtaining a solution of \eqref{11}. The uniqueness is proved in Lemma \ref{lem:uniqueness} and the estimates follow from Lemma \ref{lem:ellipticestimates}.
\end{proof}
Proposition \ref{prop:est_model_psi_1} is a model for the treatment of \(\psi_1\) the real part of \(\psi\) in Lemma \ref{FirstEstimate}. To deal with \(\psi_2\) we have to use an analogous proposition:

\begin{proposition}\label{prop:Neumann_comparison}
Let \(g:\R^2\rightarrow \R\) be such that \(g(\bar{z})=g(z)\) and \(\|g\|_{1,\alpha}<+\infty\) Then there exists a unique bounded \(v:\R^2 \rightarrow\R\)  such that $v(\bar{z})=v(z)$ and 
\begin{equation*}
\Delta v+\e^2 \p^2_{ss}v-v=g.
\end{equation*}
Furthermore  there exists a constant $C>0$ such that
\begin{equation*}
(1+|z|)\left(|v(z)|+|\nabla v(z)|\right)\leq C \|g\|_{1,\alpha},   \quad \|D^2v\|_{1,\alpha}\leq C\|g\|_{1,\alpha},
\end{equation*}
\begin{equation*}
|
\e |z|\p_s v(z)|\leq C  \|g\|_{1,\alpha} \text{ for} \ |z|>1/\e.
\end{equation*}
\end{proposition}

\begin{proof}
The symmetry assumption allows us to work in the half-plane $\R \times \R^*_+$ with homogeneous Neumann condition on the boundary. We can then apply a barrier argument and rescaled Schauder estimates to prove the proposition.
\end{proof}

In the course of the linear theory for our problem, when we separate even and odd modes we need an analogue of the following lemma.

Let us consider $R_0>0$ fixed and $R_0<R_\varepsilon<\varepsilon^{-1}$, and let $\Omega$, $\Omega'$ be the regions 
\begin{align*}
\Omega &= \{ z \in \R^2 \, | \, R_0<|z|<R_\varepsilon \}
\\
\Omega' &= \Bigl\{ z \in \R^2 \, | \, 2R_0<|z|<\frac{1}{2} R_\varepsilon \Bigr\},
\end{align*}
and recall the polar coordinates notation $z = r e^{i s}$, $r>0$, $s\in\R$.

\begin{lemma}\label{lem:estimates-3b}
Let $f:\R^2\rightarrow \R$ be such that $f(\bar{z})=-f(z)$ and $|f(z)| \leq \frac{1}{|z|}$. Let $u$ be a solution of
\begin{align*}
\Delta u +\e^2\p^2_{ss}u =f \text{ in } \Omega
\end{align*}
such that  $u(\bar{z})=-u(z)$ and 
\begin{align*}
|u(z)| &\leq R_0 |\log \varepsilon| , \quad |z|= R_0
\\
|u(z)| &\leq R_\varepsilon, \quad |z|= R_\varepsilon .
\end{align*}
Then there is $C$ such that 
\begin{equation*}
| u(z)| \leq C |z| \log\Bigl(\frac{2R_\varepsilon}{|z|}\Bigl) ,
\quad
\forall z\in \Omega'.
\end{equation*}
\end{lemma}
\begin{proof}
%We write $u = u_f + u_h$ where $u_f$ is the solution to 
%\begin{align*}
%\left\{
%\begin{aligned}
%& \Delta u +\e^2\p^2_{ss}u =f \text{ in } \Omega \\
%& u = 0 \quad\text{on } \partial\Omega
%\end{aligned}
%\right.
%\end{align*}
%and $u_h$ satisfies

We use a Fourier series decomposition, which
thanks to the symmetries we can take of the form
\[
f(r,s)=\sum_{k\geq 1} f_k(r) \sin (ks), \ \ \ u(r,s)=\sum_{k \geq 1} u_k(r) \sin (ks) .
\]
The equations on the Fourier coefficients are 
\[
u_k''+\frac{1}{r}u'_k-k^2\left(\frac{1}{r^2}+\e^2\right) u_k =f_k \quad \text{ in } (R_0,R_\varepsilon).
\]
We estimate each $u_k$ using barriers.
For $k=1$ we define
\begin{equation*}
\bar u_1(r):= 
 r \log\Bigl(\frac{3R_\varepsilon}{r}\Bigr)
\end{equation*}
The function $\bar u_1$ satisfies
\begin{equation*}
\bar u_1''+\frac{1}{r}\bar u_1'-\left(\frac{1}{r^2}+\e^2\right)\bar u_1
<-\frac{1}{r}, \quad \text{for } r < R_\varepsilon.
\end{equation*}
Thus we can use $\bar u_1$ as a barrier for $u_1$  in the interval $ (R_0,R_\varepsilon)  $ and deduce that
\begin{align}
\label{mode1}
|u_1(r)| \leq  r \log\Bigl(\frac{3R_\varepsilon}{r}\Bigr),
\quad r\in (R_0,R_\varepsilon).
\end{align}

For $k\geq 2$  we use the barrier
\[
\bar u_k(r) = C\Bigl( \frac{r}{k^2} + C |\log \varepsilon| \Bigl( \frac{r}{R_0}\Bigr)^{-k}
+ R_\varepsilon  \Bigl( \frac{r}{R_\varepsilon}\Bigr)^{k} \Bigr)
\]
where $C$ is a large fixed constant (the last two terms in $\bar u_k$ solve almost the homogeneous equation and are there for the boundary conditions).
By the maximum principle $|u_k| \leq \bar u_k$ in $(R_0,R_\varepsilon)$ and for $r\in ( 2R_0,\frac{1}{2}R_\varepsilon)$ we get
\[
\sum_{k=2}^\infty \bar u_k(r)  \leq C \frac{|\log\varepsilon|}{r}  + C r
\leq C r \log\Bigl( \frac{2R_\varepsilon}{r} \Bigr).
\]
This and \eqref{mode1} imply the desired conclusion.
\end{proof}

\bibliographystyle{abbrv}
\bibliography{biblio}

\begin{thebibliography}{10}

\bibitem{agudelo}
O.~Agudelo, M.~del Pino, and J.~Wei.
\newblock Solutions with multiple catenoidal ends to the {A}llen-{C}ahn
  equation in {$\Bbb{R}^3$}.
\newblock {\em J. Math. Pures Appl. (9)}, 103(1):142--218, 2015.

\bibitem{Alberti-Baldo-Orlandi2005}
G.~Alberti, S.~Baldo, and G.~Orlandi.
\newblock Variational convergence for functionals of {G}inzburg-{L}andau type.
\newblock {\em Indiana Univ. Math. J.}, 54(5):1411--1472, 2005.

\bibitem{almeida}
L.~Almeida and F.~Bethuel.
\newblock Topological methods for the {G}inzburg-{L}andau equations.
\newblock {\em J. Math. Pures Appl. (9)}, 77(1):1--49, 1998.

\bibitem{AmbrosioCabre2000}
L.~Ambrosio and X.~Cabr\'e.
\newblock Entire solutions of semilinear elliptic equations in {$\bold R^3$}
  and a conjecture of {D}e {G}iorgi.
\newblock {\em J. Amer. Math. Soc.}, 13(4):725--739, 2000.

\bibitem{BarlowBssGui2000}
M.~T. Barlow, R.~F. Bass, and C.~Gui.
\newblock The {L}iouville property and a conjecture of {D}e {G}iorgi.
\newblock {\em Comm. Pure Appl. Math.}, 53(8):1007--1038, 2000.

\bibitem{berestycki1997monotonicity}
H.~Berestycki, L.~A. Caffarelli, and L.~Nirenberg.
\newblock Monotonicity for elliptic equations in unbounded {L}ipschitz domains.
\newblock {\em Comm. Pure Appl. Math.}, 50(11):1089--1111, 1997.

\bibitem{BerestyckiHamelMonneau2000}
H.~Berestycki, F.~Hamel, and R.~Monneau.
\newblock One-dimensional symmetry of bounded entire solutions of some elliptic
  equations.
\newblock {\em Duke Math. J.}, 103(3):375--396, 2000.

\bibitem{BethuelBrezisHelein1994}
F.~Bethuel, H.~Brezis, and F.~H\'elein.
\newblock {\em Ginzburg-{L}andau vortices}.
\newblock Modern Birkh\"auser Classics. Birkh\"auser/Springer, Cham, 2017.
\newblock Reprint of the 1994 edition.

\bibitem{BethuelBrezisOrlandi}
F.~Bethuel, H.~Brezis, and G.~Orlandi.
\newblock Asymptotics for the {G}inzburg-{L}andau equation in arbitrary
  dimensions.
\newblock {\em J. Funct. Anal.}, 186(2):432--520, 2001.

\bibitem{ChenElliottQi1994}
X.~Chen, C.~M. Elliott, and T.~Qi.
\newblock Shooting method for vortex solutions of a complex-valued
  {G}inzburg-{L}andau equation.
\newblock {\em Proc. Roy. Soc. Edinburgh Sect. A}, 124(6):1075--1088, 1994.

\bibitem{Chiron_2005}
D.~Chiron.
\newblock Vortex helices for the gross--pitaevskii equation.
\newblock {\em Journal de math{\'e}matiques pures et appliqu{\'e}es},
  84(11):1555--1647, 2005.

\bibitem{CintiDaviladelPino2016}
E.~Cinti, J.~Davila, and M.~Del~Pino.
\newblock Solutions of the fractional {A}llen-{C}ahn equation which are
  invariant under screw motion.
\newblock {\em J. Lond. Math. Soc. (2)}, 94(1):295--313, 2016.

\bibitem{ContrerasJerrard2017}
A.~Contreras and R.~L. Jerrard.
\newblock Nearly parallel vortex filaments in the 3{D} {G}inzburg-{L}andau
  equations.
\newblock {\em Geom. Funct. Anal.}, 27(5):1161--1230, 2017.

\bibitem{delPinoFelmerKowalczyk2004}
M.~del Pino, P.~Felmer, and M.~Kowalczyk.
\newblock Minimality and nondegeneracy of degree-one {G}inzburg-{L}andau vortex
  as a {H}ardy's type inequality.
\newblock {\em Int. Math. Res. Not.}, (30):1511--1527, 2004.

\bibitem{delPinoFelmer1997}
M.~del Pino and P.~L. Felmer.
\newblock Local minimizers for the {G}inzburg-{L}andau energy.
\newblock {\em Math. Z.}, 225(4):671--684, 1997.

\bibitem{delPinoKowalczyk2008}
M.~Del~Pino and M.~Kowalczyk.
\newblock Renormalized energy of interacting {G}inzburg-{L}andau vortex
  filaments.
\newblock {\em J. Lond. Math. Soc. (2)}, 77(3):647--665, 2008.

\bibitem{delPinoKowalczykMusso2006}
M.~del Pino, M.~Kowalczyk, and M.~Musso.
\newblock Variational reduction for {G}inzburg-{L}andau vortices.
\newblock {\em J. Funct. Anal.}, 239(2):497--541, 2006.

\bibitem{delPinoKowalczykPacardWei2010}
M.~del Pino, M.~Kowalczyk, F.~Pacard, and J.~Wei.
\newblock The {T}oda system and multiple-end solutions of autonomous planar
  elliptic problems.
\newblock {\em Adv. Math.}, 224(4):1462--1516, 2010.

\bibitem{delPinoKowalczykWei2008}
M.~del Pino, M.~Kowalczyk, and J.~Wei.
\newblock The {T}oda system and clustering interfaces in the {A}llen-{C}ahn
  equation.
\newblock {\em Arch. Ration. Mech. Anal.}, 190(1):141--187, 2008.

\bibitem{delPinoKowalczykWei2011}
M.~del Pino, M.~Kowalczyk, and J.~Wei.
\newblock On {D}e {G}iorgi's conjecture in dimension {$N\geq 9$}.
\newblock {\em Ann. of Math. (2)}, 174(3):1485--1569, 2011.

\bibitem{delPinoKowalczykWeiYang2010}
M.~del Pino, M.~Kowalczyk, J.~Wei, and J.~Yang.
\newblock Interface foliation near minimal submanifolds in {R}iemannian
  manifolds with positive {R}icci curvature.
\newblock {\em Geom. Funct. Anal.}, 20(4):918--957, 2010.

\bibitem{delPinoMussoPacard2012}
M.~del Pino, M.~Musso, and F.~Pacard.
\newblock Solutions of the {A}llen-{C}ahn equation which are invariant under
  screw-motion.
\newblock {\em Manuscripta Math.}, 138(3-4):273--286, 2012.

\bibitem{Farina1999}
A.~Farina.
\newblock Symmetry for solutions of semilinear elliptic equations in {${\bf
  R}^N$} and related conjectures.
\newblock {\em Atti Accad. Naz. Lincei Cl. Sci. Fis. Mat. Natur. Rend. Lincei
  (9) Mat. Appl.}, 10(4):255--265, 1999.

\bibitem{GhoussoubGui1998}
N.~Ghoussoub and C.~Gui.
\newblock On a conjecture of {D}e {G}iorgi and some related problems.
\newblock {\em Math. Ann.}, 311(3):481--491, 1998.

\bibitem{HerveHerve1994}
R.-M. Herv\'e and M.~Herv\'e.
\newblock \'etude qualitative des solutions r\'eelles d'une \'equation
  diff\'erentielle li\'ee \`a l'\'equation de {G}inzburg-{L}andau.
\newblock {\em Ann. Inst. H. Poincar\'e Anal. Non Lin\'eaire}, 11(4):427--440,
  1994.

\bibitem{jerrardsoner}
R.~L. Jerrard and H.~M. Soner.
\newblock The {J}acobian and the {G}inzburg-{L}andau energy.
\newblock {\em Calc. Var. Partial Differential Equations}, 14(2):151--191,
  2002.

\bibitem{Kenig_Ponce_Vega2003}
C.~E. Kenig, G.~Ponce, and L.~Vega.
\newblock On the interaction of nearly parallel vortex filaments.
\newblock {\em Comm. Math. Phys.}, 243(3):471--483, 2003.

\bibitem{Klein_Majda_Damodaran1995}
R.~Klein, A.~J. Majda, and K.~Damodaran.
\newblock Simplified equations for the interaction of nearly parallel vortex
  filaments.
\newblock {\em J. Fluid Mech.}, 288:201--248, 1995.

\bibitem{LinRiviere1999}
F.~Lin and T.~Rivi\`ere.
\newblock Complex {G}inzburg-{L}andau equations in high dimensions and
  codimension two area minimizing currents.
\newblock {\em J. Eur. Math. Soc. (JEMS)}, 1(3):237--311, 1999.

\bibitem{Lin1995}
F.-H. Lin.
\newblock Solutions of {G}inzburg-{L}andau equations and critical points of the
  renormalized energy.
\newblock {\em Ann. Inst. H. Poincar\'{e} Anal. Non Lin\'{e}aire},
  12(5):599--622, 1995.

\bibitem{LinRiviere2001}
F.-H. Lin and T.~Rivi\`ere.
\newblock A quantization property for static {G}inzburg-{L}andau vortices.
\newblock {\em Comm. Pure Appl. Math.}, 54(2):206--228, 2001.

\bibitem{Mironescu1996}
P.~Mironescu.
\newblock Les minimiseurs locaux pour l'\'equation de {G}inzburg-{L}andau sont
  \`a sym\'etrie radiale.
\newblock {\em C. R. Acad. Sci. Paris S\'er. I Math.}, 323(6):593--598, 1996.

\bibitem{Montero_Sternberg_Ziemer2004}
J.~A. Montero, P.~Sternberg, and W.~P. Ziemer.
\newblock Local minimizers with vortices in the {G}inzburg-{L}andau system in
  three dimensions.
\newblock {\em Comm. Pure Appl. Math.}, 57(1):99--125, 2004.

\bibitem{PacardRiviere2000}
F.~Pacard and T.~Rivi\`ere.
\newblock {\em Linear and nonlinear aspects of vortices}, volume~39 of {\em
  Progress in Nonlinear Differential Equations and their Applications}.
\newblock Birkh\"auser Boston, Inc., Boston, MA, 2000.
\newblock The Ginzburg-Landau model.

\bibitem{Riviere1996}
T.~Rivi\`ere.
\newblock Line vortices in the {${\rm U}(1)$}-{H}iggs model.
\newblock {\em ESAIM Contr\^ole Optim. Calc. Var.}, 1:77--167, 1995/96.

\bibitem{Sandier1998}
E.~Sandier.
\newblock Locally minimising solutions of {$-\Delta u=u(1-|u|^2)$} in {${\bf
  R}^2$}.
\newblock {\em Proc. Roy. Soc. Edinburgh Sect. A}, 128(2):349--358, 1998.

\bibitem{Sandier2001}
E.~Sandier.
\newblock Ginzburg-{L}andau minimizers from {$\Bbb R^{n+1}$} to {$\Bbb R^n$}
  and minimal connections.
\newblock {\em Indiana Univ. Math. J.}, 50(4):1807--1844, 2001.

\bibitem{Sandier_Serfaty2007}
E.~Sandier and S.~Serfaty.
\newblock {\em Vortices in the magnetic {G}inzburg-{L}andau model}, volume~70
  of {\em Progress in Nonlinear Differential Equations and their Applications}.
\newblock Birkh\"{a}user Boston, Inc., Boston, MA, 2007.

\bibitem{SandierShafrir2017}
E.~Sandier and I.~Shafrir.
\newblock Small energy {G}inzburg-{L}andau minimizers in {$\Bbb{R}^3$}.
\newblock {\em J. Funct. Anal.}, 272(9):3946--3964, 2017.

\bibitem{Savin2009}
O.~Savin.
\newblock Regularity of flat level sets in phase transitions.
\newblock {\em Ann. of Math. (2)}, 169(1):41--78, 2009.

\bibitem{Shafrir1994}
I.~Shafrir.
\newblock Remarks on solutions of {$-\Delta u=(1-|u|^2)u$} in {${\bf R}^2$}.
\newblock {\em C. R. Acad. Sci. Paris S\'er. I Math.}, 318(4):327--331, 1994.

\bibitem{WeiJun2016}
J.~Wei and J.~Yang.
\newblock Traveling vortex helices for {S}chr\"odinger map equations.
\newblock {\em Trans. Amer. Math. Soc.}, 368(4):2589--2622, 2016.

\end{thebibliography}

\end{document}